\documentclass[reqno]{amsart}

\usepackage{amsthm}   
\usepackage{amsfonts, amsmath, amscd}
\usepackage{amssymb, latexsym}
\usepackage[all,cmtip]{xy}     
\usepackage{enumerate} 
\usepackage{color}
\usepackage[usenames,dvipsnames]{xcolor}
\usepackage{verbatim}
\usepackage{xcolor}

\usepackage{hyperref}
 \hypersetup{colorlinks=false}
 \usepackage{cite}
 \usepackage{ amssymb }

 \usepackage[OT2,T1]{fontenc}

 \usepackage{graphicx, psfrag}
 \usepackage{pstool}

    \usepackage{xspace}   
    \usepackage{tikz-cd}
\usepackage{tikz}
\usetikzlibrary{arrows}

\usepackage[titletoc]{appendix}
\usepackage[normalem]{ulem}     
\usepackage{amsthm}
\usepackage{amsmath}
\usepackage{amscd}
\usepackage[latin2]{inputenc}
\usepackage{t1enc}
\usepackage[mathscr]{eucal}
\usepackage{indentfirst}
\usepackage{graphicx}
\usepackage{graphics}
\usepackage{pict2e}
\usepackage{epic}

\usepackage{amssymb}
\usepackage{amsmath}
\usepackage{amsfonts}
\usepackage{amscd}
\usepackage{amsthm}
\usepackage{graphicx,psfrag}
\usepackage{wrapfig}
\usepackage{subeqnarray}
\usepackage{amssymb}
\usepackage{yfonts}
\usepackage[makeroom]{cancel}
\usepackage{trimclip}
\usepackage{cancel}
\newif\iflclip
\newif\ifbclip
\newif\ifrclip
\newif\iftclip
\def\CLIP{\dimexpr\fboxrule+.2pt\relax}
\def\nulclip{0pt}
\newcommand\partbox[2]{%
  \lclipfalse\bclipfalse\rclipfalse\tclipfalse%
  \let\lkern\relax\let\rkern\relax%
  \let\lclip\nulclip\let\bclip\nulclip\let\rclip\nulclip\let\tclip\nulclip%
  \parseclip#1\relax\relax%
  \iflclip\def\lkern{\kern\CLIP}\def\lclip{\CLIP}\fi
  \ifbclip\def\bclip{\CLIP}\fi
  \ifrclip\def\rkern{\kern\CLIP}\def\rclip{\CLIP}\fi
  \iftclip\def\tclip{\CLIP}\fi
  \lkern\clipbox{\lclip{} \bclip{} \rclip{} \tclip}{\fbox{#2}}\rkern%
}
\def\parseclip#1#2\relax{%
  \ifx l#1\lcliptrue\else
  \ifx b#1\bcliptrue\else
  \ifx r#1\rcliptrue\else
  \ifx t#1\tcliptrue\else
  \fi\fi\fi\fi
  \ifx\relax#2\relax\else\parseclip#2\relax\fi
}
\usepackage{slashed} 

\theoremstyle{definition}

\theoremstyle{remark}

\numberwithin{equation}{section}

\usepackage{amsmath}
\usepackage{amssymb}
\usepackage{amsfonts}
\usepackage{amsthm}
\usepackage{mathtools}
\usepackage{graphicx}
\usepackage{colonequals}
\usepackage{booktabs}
\usepackage{cancel}
\usepackage[numbers]{natbib}
\usepackage{hyperref}
\usepackage{fancyhdr}
\usepackage{multirow}
\usepackage{multirow}
\usepackage{braket}
\usepackage{tikz}
\usetikzlibrary{arrows,shapes,decorations.pathmorphing}
\usepackage{tikz-cd}
\usepackage[margin=2.9cm]{geometry}
\usepackage{mathabx}
\usepackage{ esint }
\usepackage{slashed}
\usepackage{mathrsfs}
\usepackage{amssymb}
\usepackage{blkarray}
\usepackage{xcolor}
\usepackage{pinlabel}   
\usepackage{scalerel,stackengine}
\stackMath
\newcommand\reallywidehat[1]{%
\savestack{\tmpbox}{\stretchto{%
  \scaleto{%
    \scalerel*[\widthof{\ensuremath{#1}}]{\kern-.6pt\bigwedge\kern-.6pt}%
    {\rule[-\textheight/2]{1ex}{\textheight}}
  }{\textheight}%
}{0.5ex}}%
\stackon[1pt]{#1}{\tmpbox}%
}
\newcommand{\R}{\mathbb{R}}

\newcommand{\Z}{\mathbb{Z}}

\newcommand{\N}{\mathbb{N}}

\newcommand{\C}{\mathbb{C}}

\renewcommand{\Re}{\operatorname{Re}}

\newcommand{\br}{\langle}
\newcommand{\kt}{\rangle}

\newcommand{\app}{\mathrm{app}}
\usepackage{graphicx}

\usepackage{mathabx,epsfig}

\theoremstyle{plain}
\newtheorem{thm}{Theorem}[section]
\newtheorem{prop}[thm]{Proposition}
\newtheorem{lm}[thm]{Lemma}
\newtheorem{cor}[thm]{Corollary}

\theoremstyle{definition}
\newtheorem{defn}[thm]{Definition}

\newtheorem{rem}[thm]{Remark}

\newtheorem{conj}[thm]{Conjecture}

\newtheorem*{hyp1*}{Hypothesis (I)}
\newtheorem*{hyp2*}{Hypothesis (II)}
\newtheorem*{hyp3*}{Hypothesis (III)}

\newtheoremstyle{exercise}{}{}{\itshape}{}{\bfseries}{:}{.5em}{\thmname{#1} \thmnumber{#2}\thmnote{(#3)}}
\theoremstyle{exercise}

\usepackage{amsfonts}
\DeclareFontFamily{U}{wncy}{}
\DeclareFontShape{U}{wncy}{m}{n}{<->wncyr10}{}
\DeclareSymbolFont{mcy}{U}{wncy}{m}{n}
\DeclareMathSymbol{\sha}{\mathord}{mcy}{"58}

\renewcommand{\d}[2]
{\frac{d#1}{d#2}}

\newcommand{\p}[2]
{\frac{\partial#1}{\partial#2}}

\newcommand{\nc}{\newcommand}
\nc {\Prod}{(\underset{I}\ph\, r_I^{n_I})}
\nc {\wlw}{\wedge\ldots\wedge}
\nc {\olo}{\otimes\ldots\otimes}
\nc{\bea}{\begin{eqnarray*}}
\nc{\eea}{\end{eqnarray*}}
\nc {\e}{\varepsilon}
\nc{\de}{\delta}
\nc{\A}{\hat a}
\nc{\Ad}{\hat a^\dag}
\nc{\grad}{\nabla}
\nc{\nlim}{\underset{n\to\infty}\lim}
\nc{\ilim}{\underset{i\to\infty}\lim}
\nc{\isum}{\sum\limits_{i=1}^N}
\nc{\xx}{\underset{x\to x_0}\lim}
\nc{\Bnrm}{\Big |\Big|}
\nc{\sym}{\text{Sym}}
\nc{\inte}{\overset{\circ}}
\nc{\del}{\partial}
\nc{\plp}{+\ldots+}
\nc{\lre}{\longrightarrow}
\nc{\be}{\begin{equation}}
\nc{\ee}{\end{equation}}
\nc{\CP}{\mathbb{CP}}
\nc{\End}{\text{End}}
\nc{\mf}{\mathfrak}
\nc{\Hom}{\text{Hom}}
\nc{\spec}{\text{Spec}}
\nc{\sub}{\subseteq}
\nc{\weakto}{\rightharpoonup}
\nc{\ph}{\varphi}
\nc{\leqc}{\lesssim}
\nc{\delbar}{\overline \del}
\nc{\Tr}{\text{Tr}}
\nc{\Lhe}{\mathcal L_{(\Phi^{h_\e}, A^{h_\e})}}
\nc{\SL}{\mathrm{SL}}

\begin{document}


\title{$\Z_2$-Harmonic Spinors and 1-forms on Connected sums and Torus sums of 3-manifolds}
\author{Siqi He and Gregory J. Parker}
\address{Department of Mathematics, Stanford University}
\email{gjparker@stanford.edu}
\address{Morningside Center of Mathematics, CAS}
\email{sqhe@amss.ac.cn}

\begin{abstract}
Given a pair of $\mathbb{Z}_2$-harmonic spinors (resp. 1-forms) on closed Riemannian 3-manifolds $(Y_1, g_1)$  and $(Y_2,g_2)$, we construct $\mathbb{Z}_2$-harmonic spinors (resp. 1-forms) on the connected sum $Y_1 \# Y_2$ and the torus sum $Y_1 \cup_{T^2} Y_2$ using a gluing argument. The main tool in the proof is a parameterized version of the Nash-Moser implicit function theorem established by Donaldson \cite{DonaldsonMultivalued} and the second author \cite{PartII}.

We use these results to construct an abundance of new examples of $\Z_2$-harmonic spinors and 1-forms. In particular, we prove that for every closed 3-manifold $Y$, there exist infinitely many $\mathbb{Z}_2$-harmonic spinors with singular sets representing infinitely many distinct isotopy classes of embedded links, strengthening an existence theorem of Doan-Walpuski \cite{DWExistence}. Moreover, combining this with the results of \cite{PartIII}, our construction implies that if $b_1(Y) > 0$, there exist infinitely many $\mathrm{Spin}^c$ structures on $Y$ such that the moduli space of solutions to the two-spinor Seiberg-Witten equations is non-empty and non-compact.

\end{abstract}


\maketitle
\tableofcontents

\section{Introduction}
\label{Introduction}
$\mathbb{Z}_2$-harmonic spinors and 1-forms were introduced by C. Taubes to study the limits of degenerating sequences of solutions to gauge-theoretic equations \cite{Taubes3dSL2C, TaubesZeroLoci}. These objects now play a significant role in multiple areas of geometry and topology, where they arise as singular limiting solutions of various geometric PDEs.

In three dimensions, $\Z_2$-harmonic 1-forms are closely related to the geometry of the $\text{SL}(2,\C)$ representation variety. Taubes's work shows that on a compact 3-manifold, sequences of flat connections with diverging energy must converge after renormalization to a $\Z_2$-harmonic 1-form \cite{Taubes3dSL2C, TaubesZeroLoci}, suggesting that the latter should  provide a refinement of the classical Morgan-Shalen compactification \cite{MorganShalen1984}, and generalizing work on the ends of the Hitchin moduli space to dimension 3 \cite{MWWW, LauraFredrickson18}. Moreover, the role of $\mathbb{Z}_2$-harmonic 1-forms is one of the essential puzzles in Witten's conjecture giving a gauge-theoretic interpretation of the Jones polynomial \cite{TaubesKW1, TaubesKW2, Sun2022z2, sun2023extended, dimakis2024moduli}.  

Subsequent work of Taubes \cite{TaubesVW, TaubesU1SW}, Haydys and Walpuski \cite{HWCompactness}, and Walpuski and Zhang \cite{WalpuskiZhangCompactness} has shown that various types of $\mathbb{Z}_2$-harmonic spinors  also appear as degenerate limits of many other equations. In each case, the existence of $\Z_2$-harmonic spinors leads to non-compactness of the moduli space that must be addressed to study the geometric consequences of the equations \cite{WalpuskiNotes}. Additionally, $\mathbb{Z}_2$-harmonic spinors play an essential role in proposals for constructing enumerative invariants of manifolds with special holonomy  \cite{DonaldsonSegal, DWAssociatives, JoyceAssociatives, Haydys2019G2instanton,Bera2022associative}, where they arise as deformation models for calibrated submanifolds \cite{SiqiSlag}.

More abstractly, $\Z_2$-harmonic spinors are the simplest type of singular Fueter section. Fueter sections are solutions of a non-linear Dirac equation valued in a bundle whose fiber is a hyperk\"ahler orbifold \cite{DoanThesis, TaubesNonlinearDirac, HaydysFSCategory}; $\Z_2$-harmonic spinors are the special case where the orbifold is $\mathbb H/\Z_2$. More general Fueter sections arise in gauge theory \cite{DonaldsonSegal, DWAssociatives, JoyceAssociatives, Haydys2019G2instanton}, and in recent proposals for generalizing Lagrangian Floer theory to the hyperk\"ahler setting \cite{DoanRezchikov, kontsevich2008stability, wang2022complex}.

Despite their importance, many questions about $\mathbb{Z}_2$-harmonic 1-forms and spinors remain unresolved, including general criteria for their existence, their relationship to global geometry, and their local behavior. One barrier to addressing such questions is the lack of explicit examples. A general existence result for $\mathbb{Z}_2$-harmonic spinors was established by Doan and Walpuski \cite{DWExistence} for 3-manifolds $Y$ with $b_1(Y)> 1$, but their proof is non-constructive. Some explicit examples have been constructed using symmetries \cite{SiqiZ3, MazzeoHaydysTakahashiExamples, TaubesWu, TaubesWu2, Siqi2024existence}. The purpose of this article is to use gluing methods to construct an abundance of new, explicit examples of $\Z_2$-harmonic spinors and 1-forms on general compact 3-manifolds.

\subsection{$\Z_2$-Harmonic Spinors and 1-forms on 3-manifolds}
Let $(Y,g)$ be a closed, oriented Riemannian 3-manifold equipped with a Clifford module $(S,\gamma, \nabla)$, where $S\to Y$ is a real vector bundle of rank $4k$ endowed with a Euclidean inner product, $\gamma: T^*Y\to \text{End}(S)$ is a Clifford multiplication, and $\nabla$ is a compatible connection. Next, let $\mathcal Z\subset Y$ be a submanifold of codimension $2$, and choose a real Euclidean line bundle $\ell\to Y\mathrm{-}\mathcal Z$. Associated to each such line bundle, there is a unique flat connection $A_\ell$ with holonomy contained in $\Z_2$. The bundle $S\otimes_\R \ell$ carries a twisted Dirac operator $D$ formed using the connection $\nabla$ on $S$ and $A_\ell$ on $\ell$. A generalized $\Z_2$-harmonic spinor is a triple $(\mathcal Z,\ell, \Phi)$ where $\Phi\in \Gamma(S\otimes_\R \ell)$ satisfies
\be D\Phi=0, \hspace{1cm} \hspace{2cm} \nabla \Phi \in L^2(Y\mathrm{-}\mathcal Z)\label{Z2harmonicequation}\ee

\noindent on $Y\mathrm{-}\mathcal Z$. The submanifold $\mathcal Z$ is called the {\bf singular set}. If $\mathcal Z$ has sufficient regularity, the second requirement of (\refeq{Z2harmonicequation}) implies that $|\Phi|$ extends continuously to $Y$ with $\mathcal Z\subseteq |\Phi|^{-1}(0)$. For fixed $\mathcal Z, \ell$ (\refeq{Z2harmonicequation}) is linear, and solutions are considered up to the scaling action of $\R^{>0}$ on $\Phi$ and the action of $\Z_2$ by $\Phi\mapsto -\Phi$. 

An equivalent viewpoint is to regard a generalized \(\mathbb Z_2\)-harmonic spinor as a section of the fiberwise quotient \(S/\{\pm1\}\), or equivalently as a two-valued section of \(S\). The equivalence is as follows: away from its zero set, such a two-valued section can be lifted locally to an ordinary section of \(S\). There are two such choices, which differ by multiplication by $-1$. On overlaps of two regions with a single chosen local lifts, the two choices differ by multiplication by \(\pm1\), and these transition signs define a principal \(\mathbb Z_2\)-bundle, equivalently a real Euclidean line bundle \(\ell Y\setminus \mathcal Z\). After twisting by \(\ell\), the local lifts glue to a single-valued section of \(S\otimes\ell\). Conversely, a section of \(S\otimes\ell\) determines a section of \(S/\{\pm1\}\) after forgetting the signs. Thus the twisting line bundle \(\ell\) records the monodromy of the two-valued section. A third, equivalent, viewpoint is to consider anti-invariant sections on the double cover \(Y_{\mathcal Z}\to Y\) branched along \(\mathcal Z\), with monodromy determined by \(\ell\), endowed with the pullback metric of cone angle \(4\pi\).

In dimension 3, there are two Clifford modules that are of particular interest: 

\begin{enumerate}
\item[(1)] $S=\slashed S$ is the spinor bundle of a spin or $\mathrm{Spin}^c$ structure $\mathfrak s_0$, and $\nabla=\nabla^\text{Spin}+B$ is a real-linear perturbation of the spin connection by $B\in \Omega^1(\mathfrak{so}(\slashed S))$. In this case, 

\be  D=\slashed D_B\label{Diracoperatorone}\ee
\noindent is a perturbation of the spin Dirac operator. 
\item[(2)] $S=\Omega^0(\R)\oplus \Omega^1(\R)$, and $\nabla$ is the Levi-Civita connection. In this case, \(D\) is the Hodge--de Rham Dirac operator
\be
D_{\mathrm{dR}}
=
\begin{pmatrix}
	0 & -d^\star \\
	-d & \star d
\end{pmatrix}
\label{Diracoperatortwo}
\ee
acting on \(\ell\)-valued forms \(\Omega^0(\ell)\oplus\Omega^1(\ell)\). Here \(d\) denotes the exterior derivative. 
\end{enumerate}

\noindent For the remainder of the article, we continue to write \(D\) for either the spin Dirac operator or \( D_{\mathrm{dR}}\), depending on the case under consideration.

Each choices of $D$ implicitly depends on parameters $p=(g,B)$ where $g$ is the Riemannian metric, and $B$ a perturbation to the connection in case (1). In the second case, applying $d^\star$ to the $\Omega^1$-components and integrating by parts shows that a solution $\Phi=(\nu_0, \nu_1)\in \Omega^0\oplus \Omega^1$ of (\refeq{Z2harmonicequation}) has $\nu_0=0$ when $\ell$ is non-trivial. Solutions in case (2) are therefore also called $\Z_2${\bf -harmonic 1-forms}. The term {\bf generalized $\Z_2$-harmonic spinors} is used to refer to the general case of either (1) or (2), while (true) $\Z_2$-harmonic spinors refers to case (1).    

 We construct examples of generalized $\Z_2$-harmonic spinors by proving gluing results for how solutions behave under connected sum and torus sum operations, and applying these with explicit solutions on some ``standard' 3-manifold (e.g. $S^3, S^1\times S^2$). Thus, beginning with a pair of 3-manifolds $(Y_i, g_i)$ for $i=1,2$ and a pair of generalized $\Z_2$-harmonic spinors $(\mathcal Z_i, \ell_i, \Phi_i)$, we construct solutions of (\refeq{Z2harmonicequation}) on $Y=Y_1\#Y_2$ and on $Y_K=Y \cup_{K_1=K_2} Y_2$ where $K_i \in Y_i$ are knots. These results also emphasize that there is an appreciable disparity between the two cases (\refeq{Diracoperatorone}--\refeq{Diracoperatortwo}), which mimics that of the classical case of harmonic spinors and 1-forms: spinors are well-behaved under generic perturbations of the metric, whereas 1-forms are beholden to constraints coming from $L^2$-Hodge theory. 

The main technical difficulty in the construction is that the singular Dirac operator $D$ fails to be Fredholm on any natural function spaces, thus an approximate solution $\Phi=\Phi_1 \# \Phi_2$ cannot be corrected to a true solution by an application of the implicit function theorem on Banach spaces. In fact, on function spaces such that the second condition of (\refeq{Z2harmonicequation}) is satisfied, $D$ has an infinite-dimensional obstruction to solving. Previous work of Donaldson \cite{DonaldsonMultivalued} and of the second author \cite{PartII}, see also \cite{HeParkerWalpuski26Deformation,BeraWalpuski2025}, has shown that deformations of the singular set may be used to cancel this obstruction, provided one works in the category of tame Fr\'echet manifolds. This consideration leads to a version of the Nash-Moser implicit function theorem suitable for correcting approximate solutions. 

The remainder of Section \ref{Introduction} summarizes our main results. Sections \ref{section2}--\ref{section3} introduce the relevant Nash-Moser theory and prove the gluing results, and Section \ref{section4} is devoted to applications.  

\begin{rem}\label{regularityremark}
When generalized $\Z_2$-harmonic spinors arise as limiting objects, there is no assurance that the singular set $\mathcal Z$ is a smooth submanifold. In most situations, $\mathcal Z$ is known to be a closed, rectifiable set of Hausdorff codimension 2 \cite{TaubesZeroLoci, ZhangRectifiability}. Here, we focus on the case that $\mathcal Z$ is a smooth,  embedded submanifold, which is expected to be true for generic parameters (this was originally conjectured by Taubes, and is supported by \cite{PartII}). Some results about the regularity and structure of the singular set appear in \cite{TaubesWu, Siqi2024existence, MazzeoHaydysTakahashi}, and suggest many new and intriguing directions.
\end{rem}

\subsection{Main Results}
Let $(Y,g)$ be a closed, oriented Riemannian 3-manifold as above, and denote by $D$ one of the twisted Dirac operators (\refeq{Diracoperatorone}--\refeq{Diracoperatortwo}). 

We consider generalized $\Z_2$-harmonic spinors satisfying the following criteria. These temper the potentially wild behavior at the singular set, and are expected to be generic (cf. Remark \ref{regularityremark} and \cite{SiqiZ3}).

\begin{defn} \label{regulardef}A generalized $\Z_2$-harmonic spinor $(\mathcal{Z}, \ell, \Phi)$ with respect to parameters $p=(g, B)$ is said to be
\medskip 


\begin{enumerate}
\item[(i)] { {({{{Smooth}}})}} if the singular set $\mathcal Z\subset Y$ is a smooth, embedded link, and $\ell$ restricts to the M\"obius bundle on every disk normal to $\mathcal Z$.
\smallskip 

\item[(ii)]{({{{Non-degenerate}})}} if $\Phi$ has non-vanishing leading-order, i.e. there is a constant $c>0$ such that \be |\Phi| \geq c\cdot \text{dist}(-,\mathcal Z)^{1/2}.\label{nondegenerate}\ee
\noindent Additionally, we say that $\Phi$ is {\it weakly} non-degenerate if there exists a tubular neighborhood of $\mathcal Z$ on which (\refeq{nondegenerate}) holds. 
\end{enumerate}

\end{defn}

\noindent Note that non-degeneracy implies that $\mathcal Z=|\Phi_0|^{-1}(0)$, whereas a weakly non-degenerate generalized $\Z_2$-harmonic spinor may have additional zeros away from $\mathcal Z$ which are non-singular in the sense that $\ell$ extends over these. 
\medskip

Our first theorem constructs $\Z_2$-harmonic spinors and 1-forms on connected sums, given one on each of the summands. 

\begin{thm}
\label{maina}
Suppose that for $i=1,2$, $(Y_i, g_i)$ are closed, oriented Riemannian manifolds, and that $(\mathcal Z_i, \ell_i, \Phi_i)$ are smooth, weakly non-degenerate generalized $\Z_2$-harmonic spinors with respect to the parameters $p_i=(g_i, B_i)$. 

Choose points \(y_i\in Y_i\setminus \mathcal Z_i\), and form the connected sum
\(Y=Y_1\#Y_2\) by deleting sufficiently small geodesic balls
\(U_i\Subset Y_i\setminus \mathcal Z_i\) centered at \(y_i\). Let $\ell$ be the flat $\mathbb{Z}_2$ bundle on this connected sum, whose first Steifel--Whitney class is given by $w_1(\ell) = w_1(\ell_1) + w_1(\ell_2)$ under the natural Mayer--Vietoris identification $H^1(Y\setminus Z;\mathbb Z_2)
\simeq
H^1(Y_1\setminus \mathcal Z_1;\mathbb Z_2)
\oplus
H^1(Y_2\setminus \mathcal Z_2;\mathbb Z_2),$
obtained from the decomposition of \(Y\setminus \mathcal Z\) along the gluing sphere. 

Then, for each pair $\alpha = (a, b) \in S^1 \subseteq \mathbb{R}^2$ with both components non-zero, $Y$ admits $\mathbb{Z}_2$-harmonic spinors $(\mathcal{Z}_\alpha, \ell, \Phi_\alpha)$, which are small perturbations of
\begin{equation}
	\mathcal{Z} = \mathcal{Z}_1 \sqcup \mathcal{Z}_2, \hspace{2cm}  \Phi = a\Phi_1 + b\Phi_2, \label{approximateZ2harmonic}
\end{equation}
respectively, with respect to the parameters $p_\alpha = (g', B_\alpha)$ that coincide with $p_i$ on the complement of small open balls $U_i \subseteq Y_i - \mathcal{Z}_i$. Moreover, each $(\mathcal{Z}_\alpha, \ell, \Phi_\alpha)$ is smooth and weakly non-degenerate. 

\end{thm}

\begin{rem}
Since $\Z_2$-harmonic spinors are considered up to sign, the set of equivalence classes of spinors constructed above is parameterized by $[\alpha]\in \R\mathbb P^1$. A similar result holds for multi-connected sums $Y=\#_i^n Y_i$, where $[\alpha] \in \R \mathbb P^{n-1}$ is chosen from the Zariski open subset where no coordinate is zero. 
\end{rem}

Our next theorem proves a similar gluing formula for the spinor case $S=\slashed S$ (\refeq{Diracoperatorone}) now joining the manifolds by associating a knot neighborhood. With $(Y_i,g_i)$ as before, now assume additionally that $K_i\subseteq Y_i$ are oriented knots such that $K_i \cap \mathcal Z_i=\emptyset$. Let \be N(K_i)\simeq  S^1 \times D_R\label{solidtorusnbhd}\ee

\noindent be a tubular neighborhood of each $K_i$ with radius $R$ such that $N(K_i)\subset Y_i-\mathcal Z_i$, where $K_i$ is given by $\{0\}\times S^1$. Since \(Y_i\) is oriented and \(K_i\) is an oriented knot, the normal bundle of \(K_i\) is an oriented real rank-two bundle over \(S^1\), hence is trivial. Thus the choice of an identification \(N(K_i)\simeq S^1\times D_R\) is simply a choice of framing. $N(K_i)$ may be endowed with coordinates $(x,y,t)$ such that $t$ is an arclength coordinate along $K_i$, and $x,y$ are normal coordinates such that $g_i |_{N(K_i)}=dx^2 + dy^2 + dt^2 + O(r)$ where $r^2={x^2 + y^2}$. Assume that the \(g_i\)-lengths of the knots agree: $\operatorname{Length}_{g_1}(K_1)=\operatorname{Length}_{g_2}(K_2).$ This can always be achieved by rescaling one of the metrics. Let $\ph: N(K_1)\to N(K_2)$ be a diffeomorphism given by the identity on the $S^1$ factor, and by a (possible $t$-dependent) orientation-reversing linear isometry to first order on the $D_R$ factor. The torus sum is defined to be $$Y_K= Y_1 \cup_\ph Y_2,$$

\noindent where the neighborhoods $N(K_i)$ are associated via $\ph$.  

\begin{thm} \label{mainb}
Suppose that for $i=1,2$, $(Y_i, g_i)$ are closed, oriented Riemannian manifolds, and that $(\mathcal Z_i, \ell_i, \Phi_i)$ are smooth, weakly non-degenerate (true) $\Z_2$-harmonic spinors with respect to parameters $p_i=(g_i, B_i)$. Assume additionally that $\slashed S_i\otimes \ell_i$ are induced by spin structures $\mathfrak s_i$ $($defined over $Y_i-\mathcal Z_i$$)$ with $\ph^*(\mathfrak s_2 |_{N(K_2)})\simeq \mathfrak s_1 |_{N(K_1)}$. 

Let $\ell$ be the flat line bundle on $Y_K$ defined by $\ell_1,\ell_2$ and $\ph$.  Then, for each $\alpha=(a,b)\in S^1\subseteq \R^2$ with both non-zero, $Y_K$ admits $\Z_2$-harmonic spinors $(\mathcal Z_\alpha, \ell, \Phi_\alpha)$ which are small perturbations of 
\be \mathcal Z=\mathcal Z_1 \sqcup \mathcal Z_2, \hspace{2cm} \Phi= a\Phi_1 + b\Phi_2 \label{approximateZ2harmonic2}  \ee

\noindent respectively, with respect to parameters $p_\alpha=(g',B_\alpha)$ such that $p_\alpha$ agrees with $p_i$ on the complement of $N(K_i)\subseteq Y_i$, possibly up to a constant scaling of the metric. Moreover, each $(\mathcal{Z}_\alpha, \ell, \Phi_\alpha)$ is smooth and weakly non-degenerate. 
\end{thm}
\medskip 

\begin{rem}\label{Riemannsurfacegluing}
Theorem \ref{mainb} does not apply in the case of 1-forms in general (see Remark \ref{thmmainbobstruction}). It does, however, apply to 1-forms in the product case $Y_i = S^1 \times \Sigma_i$ and $K_i = S^1 \times \{x_i\}$, then $Y_K = S^1 \times (\Sigma_1 \# \Sigma_2)$. This implies the analogue of Theorem \ref{maina} on Riemann surfaces (proved in  Section \ref{section3.3.3}).
\end{rem}

\subsection{Examples and Applications: Spinors}
\label{section1.3}
Theorem \ref{maina} enables us to construct many new examples of $\mathbb{Z}_2$-harmonic spinors for the spin Dirac operator $D=\slashed D$ (i.e. case \refeq{Diracoperatorone}) on compact manifolds. 

The main class of examples uses solutions on Seifert--fibered $3$-manifolds as building blocks. Recall that a $3$-manifold $Y$ is called Seifert--fibered if it is the total space of an orbifold fiber bundle $\pi: Y\to \Sigma$ with fiber $S^1$ over a closed 2-dimensional orbifold $\Sigma$. Using the structure results for Seifert--fibered spaces and orbifold theory, we obtain

\begin{thm} \label{seifertfiberedspinors}
Let $\pi : Y \to \Sigma$ be a Seifert--fibered $3$-manifold. Then for each $k \geq 1$, there exist metrics $g_k$ that admit smooth, non-degenerate $\mathbb{Z}_2$-harmonic spinors $(\mathcal{Z}_k, \ell_k, \Phi_k)$, where $\mathcal{Z}_k \subseteq Y$ is the union of disjoint fibers of $\pi$.
\end{thm}

In particular, this implies the existence of $\mathbb{Z}_2$-harmonic spinors in the following cases, all of which are smooth and non-degenerate (see Corollary \ref{cor_examples_Z_2_harmonic_spinors} for details):

\begin{cor} \label{example1.8} The following three-manifolds admit $\Z_2$-harmonic spinors: 
\begin{enumerate}
    \item[(a)] $Y = S^3$ admits $\mathbb{Z}_2$-harmonic spinors $(\mathcal{Z}_k, \ell_k, \Phi_k)$ with respect to the Berger metrics $g_{B,V}$ such that $\mathcal{Z}_k$ is a Hopf link with $2k$ components.
    \item[(b)] $Y = S^1 \times S^2$ admits $\mathbb{Z}_2$-harmonic spinors $(\mathcal{Z}_k, \ell_k, \Phi_k)$ with respect to metrics $g_k = dt^2 + V_k \cdot g_{S^2}$ for $V_k \in \mathbb{R}$, such that $\mathcal{Z}_k = S^1 \times \mathcal{Z}_{S^2}$ where $\mathcal{Z}_{S^2} \subseteq S^2$ is a collection of $2k$ points.
    \item[(c)] $Y = \Sigma(2,3,5)$, the Poincar\'e homology sphere, admits a $\mathbb{Z}_2$-harmonic spinor $(\mathcal{Z}, \ell, \Phi)$ with a connected singular set $\mathcal{Z} = \pi^{-1}(p_0)$ for some $p_0 \in \Sigma$.
\end{enumerate}
\end{cor}

Examples (1.8a) and (1.8b) may be used in conjunction with Theorems \ref{maina} and \ref{mainb} respectively to generate examples on general compact 3-manifolds. First, we recall the following result of C. B\"ar. 

\begin{thm}[\cite{Bar1996}] \label{Bar}Every closed oriented 3-manifold admits metrics with harmonic spinors. 
\end{thm}

\noindent Of course, such a classical harmonic spinor is a particular instance of a $\Z_2$-harmonic spinor with $\mathcal Z=\emptyset$ and $\ell$ being the trivial line bundle; thus Theorem \ref{maina} applies. We conclude: 

\begin{cor} \label{spinorsa}Every closed, oriented 3-manifold admits infinitely many parameters $p=(g,B)$ with smooth, non-degenerate $\Z_2$-harmonic spinors, and the singular sets of these represent infinitely many distinct isotopy classes of embedded links. 
\end{cor}
\begin{proof}
Write $Y\simeq Y \# S^3$, where the first factor is endowed with a metric admitting a harmonic spinor via Theorem \ref{Bar}, and $S^3$ with one of the metrics from Corollary (\ref{example1.8}a). The result then follows from Theorem \ref{maina}. 
\end{proof}

Corollary \ref{spinorsa} strengthens the existence result of Doan--Walpuski \cite{DWExistence}, which requires that $b_1(Y) > 1$. Moreover, it shows that the collection of isotopy classes of links that may arise includes at least the $2k$-Hopf link on an open ball for every $k$. The examples constructed by Corollary \ref{spinorsa} are also reasonably explicit: they have a metric equal to a metric admitting a harmonic spinor on the complement of a small ball in $Y$ and to Berger metric on the complement of a small ball in $S^3$; the spinors themselves are a small perturbation of the solutions on each summand. Moreover, the existence of a single such parameter implies the existence of an infinite-dimensional family of such parameters, because the set of such parameters is an open neighborhood in a submanifold of finite codimension \cite[Thm. 1.4]{PartII}.


We can deduce an even stronger existence result by applying Theorem \ref{mainb}. 

\begin{cor}\label{spinorsb}
Let $K\subseteq Y$ be a knot in a closed oriented 3-manifold. Then for each $k\geq 1$, there exist parameters $(g_k, B_k)$ on $Y$ that admit smooth, weakly non-degenerate $\Z_2$-harmonic spinors whose singular set is isotopic to $2k$ disjoint copies of $K$, which is the $(2k,0)$ cable link of the knot $K$.
\end{cor}

\begin{proof}
Let $K=K_1$ and $K_2=S^1 \times \{p_0\}\subseteq S^1 \times S^2$. Write $Y$ as $Y\simeq Y\cup_K (S^1\times S^2)$, where the first factor is endowed with a metric admitting a harmonic spinor, and $S^1\times S^2$ has the metric of Corollary (\ref{example1.8}b). The result then follows from Theorem \ref{mainb}. 
\end{proof}

\noindent Corollary \ref{spinorsb} strengthens Corollary \ref{spinorsa} by providing examples where $[\mathcal Z]\in H_1(Y;\Z)$ is non-trivial. Repeated applications of Corollary~\ref{spinorsb} implies the same statement for multi-component links. In contrast to Corollary \ref{spinorsa}, the examples of Corollary \ref{spinorsb} may have $\mathcal Z$ non-trivial in $H_1(Y;\Z)$ (note that smoothness implies $[\mathcal Z]\in H_1(Y;\Z)$ is even). 

Corollary \ref{spinorsb} has a rather surprising implication in gauge theory. Recall that for the standard Seiberg-Witten equations, the moduli space of solutions with negative formal dimension is only non-empty for finitely many $\text{Spin}^c$ structures . The following corollary shows that this classic fact fails rather dramatically for the two-spinor Seiberg-Witten equations, a similar phenomenon first observed by Doan \cite{Doan19seibergwitten} in the case that $Y = S^1 \times \Sigma$. 

\begin{cor}
\label{nonempty}
Let $Y$ be a closed, oriented 3-manifold with $b_1(Y) > 0$ and let $\mathcal{M}_{\text{SW}^2}$ be the moduli space of two-spinor Seiberg-Witten solutions. Then there exist infinitely many $\text{Spin}^c$ structures on $Y$ such that there are parameters $p = (g, B)$ for which the set of regular points in the moduli space $\mathcal{M}_{\text{SW}^2}$ is non-empty and non-compact.
\end{cor}

\noindent As with Corollaries \ref{spinorsa} and \ref{spinorsb}, the existence of a single parameter for which this result holds implies the existence of infinitely many such parameters. Corollary \ref{nonempty} follows directly from Corollary \ref{spinorsb} and the gluing result of the second author \cite{PartIII}, which constructs Seiberg--Witten solutions from a given $\Z_2$-harmonic spinor in the $\text{Spin}^c$ structure $\slashed{S}$ satisfying $\text{det}(\slashed{S}) = -2\text{PD}[\mathcal{Z}]$ (see Section \ref{section2.3}). 
\subsection{Examples and Applications: 1-Forms}
\label{subsection_examplesandapplications1-forms}
The behavior of $\Z_2$-harmonic 1-forms on $3$-manifolds has a rather different flavor than the theory for spinors, because such harmonic forms are linked to the $L^2$-cohomology of the double branched cover via Hodge theory. Furthermore, the compactness theorem of Taubes \cite{Taubes3dSL2C}
suggests that $\Z_2$-harmonic 1-forms should be regarded as an ideal boundary for the irreducible component of the $\mathrm{SL}(2,\C)$ representation variety $\mathcal R(Y)$ of the 3-manifold. The fact that the geometry of the representation variety can reflect deep aspects 3-manifold topology hints that $\Z_2$-harmonic 1-forms might also be subject to other, more subtle topological restrictions. 

To elaborate on the connection to $L^2$-cohomology, let $(\mathcal{Z}, \ell, \nu)$ be a $\Z_2$ harmonic 1-form defined on $Y$. Let $p: Y_\mathcal{Z} \to Y$ be the double branched cover map branched along $\mathcal{Z}$ whose monodromy is given by that of $\ell$, with $\sigma$ being the involution over $Y_\mathcal{Z}$. By \cite[Lemma 1.5]{wang93modulispaces}, the space of $L^2$-harmonic forms
$$
\{\alpha \in \Omega^1(Y_{\mathcal{Z}}) \mid \alpha \in L^2, \; d\alpha = d^\star\alpha = 0 \; \mathrm{over} \; Y_{\mathcal{Z}} - p^{-1}(\mathcal{Z}) \}
$$
is isomorphic to the singular cohomology $H^1(Y_{\mathcal{Z}}; \mathbb{R})$, where $d^\star$ is formed using pullback metric $p^*g$. This group carries additional structure: the involution $\sigma: Y_\mathcal{Z} \to Y_\mathcal{Z}$ induces a decomposition 
$$
H^i(Y_\mathcal{Z}; \mathbb{R}) = H^i_+(Y_\mathcal{Z}; \mathbb{R}) \oplus H^i_-(Y_\mathcal{Z}; \mathbb{R}),
$$ 
into the $\pm 1$ eigenspaces of $\sigma^*$. The pullback $p^*\nu$ of a $\Z_2$-harmonic 1-form $(\mathcal{Z}, \ell, \nu)$ is an $L^2$ harmonic 1-form, i.e., $d(p^*\nu) = d\star_{p^*g} (p^*\nu) = 0$ with respect to the singular cone metric $p^*g$ in the $-1$ eigenspace, that additionally satisfies $\nabla \nu \in L^2$ (and thus $\nu'=0$ on $p^{-1}(\mathcal Z)$).

In this context, we say that a cohomology class $[\alpha] \in H^1_-(Y_\mathcal{Z}; \mathbb{R})$ is represented by a $\Z_2$ harmonic 1-form if there exists a $\Z_2$ harmonic 1-form $(\mathcal{Z}', \ell', \nu')$ on $Y$ such that there exists a diffeomorphism $\phi: Y \to Y$ with $\phi^* \ell' = \ell$ and $[p^* \phi^* \nu'] = [\alpha]$.  

The following result, proved in Section \ref{subsec_Z2Seifert} as a consequence of Proposition \ref{prop_1_form_existence_Seifert_fibered}, gives examples of \(\mathbb Z_2\)-harmonic 1-forms on Seifert--fibered spaces.
\begin{prop}
\label{seifertfiberedoneforms}
    Let $Y$ be a Seifert--fibered space with Seifert invariant $(\gamma, b,  (\alpha_1, \beta_1), \cdots, (\alpha_n, \beta_n))$, where $b$ is the fiber degree, $\gamma$ is the orbifold genus, and $(\alpha_i, \beta_i)$ are local orbifold invariants. Suppose either 
    \begin{enumerate}
        \item $\gamma = 0$ and $n \geq 4$,
        \item $\gamma = 1$ and $n \geq 2$, or
        \item $\gamma \geq 2$,
    \end{enumerate}
    then there exist non-degenerate $\Z_2$ harmonic 1-forms on $Y$.
\end{prop}

To emphasize the distinction between this and the spinor case, we make the following conjecture.
\begin{conj}\label{conj_Z21form_nonexistence} Suppose $\mathcal R(Y)$ is zero-dimensional. Then there exist no $\mathbb{Z}_2$ harmonic 1-forms on $Y$ with $\mathcal{Z} \neq \emptyset$ with respect to any metric. In particular, there exist no $\mathbb{Z}_2$-harmonic 1-forms on $S^3$, and no $\mathbb{Z}_2$-harmonic 1-forms on $S^1 \times S^2$ and $T^3$ except for the classical harmonic forms with $\mathcal{Z} = \emptyset$.
\end{conj}

\noindent Taubes used a Weitzenb\"ock formula to prove this non-existence result for the round metric on $S^3$ \cite{Taubes3dSL2C,TaubesKW1}. Conjecture \ref{conj_Z21form_nonexistence} extends this statement to arbitrary metrics. It is motivated by the relationship between $\Z_2$-harmonic 1-forms and the $\mathrm{SL}(2,\C)$ representation variety, as well as by the gluing results of the second author \cite{PartIII, PartIV}. We also refer to \cite{HWZ24} for recent progress toward this conjecture.

 In particular, given the conjecture, it seems unlikely to the authors that there is any analogue of Theorems \ref{spinorsa} and \ref{spinorsb} in the case of 1-forms.  

Proposition \ref{seifertfiberedoneforms} also provides evidence for Conjecture \ref{conj_Z21form_nonexistence}.  For example, the irreducible character variety of the Brieskorn homology spheres $\Sigma(a_1, \dots, a_n)$ is zero-dimensional if and only if $n=3$ (cf. \cite{nasatyr1995orbifold}), while Proposition \ref{seifertfiberedoneforms} shows that there exist $\Z_2$ harmonic 1-forms on $\Sigma(a_1, \dots, a_n)$ for each $n\ge 4$.

Theorem \ref{maina} can be reinterpreted in the context of $\Z_2$-harmonic 1-forms as a statement about $L^2$-cohomology.
The operations of connected summing and taking branched double covers do not commute. With $Y=Y_1\#Y_2$, $\mathcal Z=\mathcal Z_1\sqcup \mathcal Z_2$ and $w_1(\ell) = w_1(\ell_1) + w_1(\ell_2)$, the connected sum $Y_{\mathcal Z_1}\# Y_{\mathcal Z_2}$ differs from $Y_\mathcal Z$ by a surgery operation. Topologically, $Y_\mathcal Z$  is the double connected sum, with topological type given by  $Y_\mathcal Z\simeq Y_{\mathcal Z_1}\# Y_{\mathcal Z_2}\#(S^1\times S^2)$. Regarding the anti-invariant part of the first cohomology of the double branched covering, we ascertain that
\begin{equation}
    H^1_-(Y_\mathcal{Z}; \mathbb{R}) \cong H^1_-(Y_{\mathcal{Z}_1}; \mathbb{R}) \oplus H^1_-(Y_{\mathcal{Z}_2}; \mathbb{R}) \oplus \mathbb{R}.\label{branchedcovercohomology}
\end{equation}

\noindent This yields the following connected sum theorem:
\begin{thm}
    \label{thm_non_degenerate_gluing}
    Assuming for $i = 1, 2$, $(Y_i, g_i)$ are closed, oriented Riemannian manifolds and $(\mathcal{Z}_i, \ell_i, \nu_i)$ are $\Z_2$-harmonic 1-forms representing $[\alpha_i] \in H^1_-(Y_{\mathcal{Z}_i}; \mathbb{R})$. Assuming $\mathcal{Z}_1$ is non-empty and $(\mathcal{Z}_1, \ell_1, \nu_1)$ is smooth and non-degenerate, then for any $(a, b) \in \mathbb{R}^2$ with both non-zero, any $[\alpha]\in H^1_-(Y_{\mathcal{Z}}; \mathbb{R})  $ closely approximating $a[\alpha_1] + b[\alpha_2]$ can be represented by a non-degenerate $\Z_2$-harmonic 1-form.
\end{thm}

\noindent In particular, classes with a component in the $\R$ summand in (\refeq{branchedcovercohomology}) are also represented by $\Z_2$-harmonic 1-forms. Further implications of $\Z_2$-harmonic 1-forms for the $\text{SL}(2,\C)$ representation variety are discussed in Section \ref{section4}.

\begin{rem}
Conjecture \ref{conj_Z21form_nonexistence} and Theorem \ref{thm_non_degenerate_gluing} refer to the case of the unperturbed Hodge-de Rham operator in (\refeq{Diracoperatortwo}). If this operator is perturbed, then the results of Section \ref{section1.3} hold just as in the spinor case, but any relationship to Hodge theory is destroyed.  This is true because in dimension 3, the bundles $\slashed S$ and $\Omega^0\oplus \Omega^1$ are isomorphic as real Clifford modules, and under this isomorphism the operators $\slashed D$ and $D_{\mathrm{dR}}$ differ by zeroth order terms, so perturbed 1-forms can be viewed as a special case of perturbed spinors. 
\end{rem}

\bigskip 
\noindent \textbf{Acknowledgements.} This project began at the Simons-Laufer Mathematical Sciences Institute Semester program ``Analytic and Geometric Aspects of Gauge Theory'' (NSF Grant DMS-192893) in Fall 2022 and the authors wish to thank SLMath for its hospitality. This work benefited from the interest and expertise of a great many people to whom the authors express their gratitude, including Chris Gerig, Jianfeng Lin, Rafe Mazzeo, Clifford Taubes, Thomas Walpuski and Boyu Zhang. G.P. is supported by NSF Mathematical Sciences Postdoctoral Research Fellowship Award No. 2303102. 
		
\section{Nash-Moser Theory}
\label{section2}
This section establishes a suitable implicit function theorem for generalized $\Z_2$-spinors; this will be used later to correct approximate solutions of the singular Dirac equation to true solutions. This implicit function theorem is a version of the Nash-Moser implicit function theorem for tame Fr\'echet manifolds, which includes the deformations of the singular set of generalized $\Z_2$-spinors. Our approach generalizes the work of \cite{PartII,DonaldsonMultivalued} to 1-parameter families, and unifies these two results in a single statement about generalized $\Z_2$-spinors. 

\subsection{Elliptic Edge Theory}
\label{section2.1}

This section reviews the elliptic theory for $D$ established in \cite{MazzeoHaydysTakahashi}, \cite[Sections 2--4]{PartII}. For the entirety of Section \ref{section2}, $D$ denotes either of the Dirac operators in (\refeq{Diracoperatorone}--\refeq{Diracoperatortwo}).

With $(Y,g)$ as above, let $r: Y\mathrm{-}\mathcal Z_0\to \R$ be a weight function such that $r=\text{dist}(-,\mathcal Z_0)$ on a tubular neighborhood of $\mathcal Z_0$, and $r=r_0$ is constant on the complementary neighborhood. Let $w: Y \to \R$ denote a second weight function such that $w=1$ where $r\neq r_0$. 
 
Define the spaces of ``boundary'' and ``edge'' vector fields respectively by 
\bea
\mathcal V^\text{b}&=&\{V \in C^\infty(Y;TY) \ | \ V |_{\mathcal Z_0}\in C^\infty(\mathcal Z_0; T\mathcal Z_0)\},\\
\mathcal V^e &=& \{V \in C^\infty(Y;TY) \ | \ \  \hspace{.8cm}V |_{\mathcal Z_0}=0 \hspace{.8cm} \ \}.
\eea

\noindent Denote by $\nabla^\text{b}, \nabla^e$ the covariant derivatives with respect to vector fields in these spaces, so that in local coordinates $(t,x,y)$ where $t$ is the tangential coordinate to $\mathcal Z_0$ and $(x,y)$ are normal coordinates, these are given by 
\begin{eqnarray}\label{bderiv}
\nabla^\text{b}&=&dx \otimes r\nabla_x + dy \otimes r\nabla_y + dt \otimes \nabla_t, \\  
\nabla^e&=&dx \otimes r\nabla_x + dy \otimes r\nabla_y + dt \otimes  r\nabla_t. \label{edgederiv}
\end{eqnarray}

\noindent Note that $|\nabla^e \ph|\leq |\nabla^\text{b}\ph|$ holds pointwise. 
\begin{defn}\label{sobolevspaces}
The mixed boundary and edge Sobolev spaces of regularity $(m,m+n)$ for $m,n\in \N$ with weight $\nu$ are defined by

\be r^\nu H^{m,n}_{\text{b},e,w}(Y\mathrm{-}\mathcal Z_0; S):=\left\{ \psi \in L^2(Y;S)  \ \Big | \ \int_{Y\mathrm{-}\mathcal Z_0} \sum_{|\alpha|\leq n, |\alpha|+|\beta|\leq m} |(\nabla^e)^\alpha (\nabla^\text{b})^\beta \psi|^2  \ r^{-2\nu} w^2  \ dV  \ < \ \infty \  \right\},\ee

\noindent where $\alpha,\beta$ are multi-indices and $dV$ is the volume form. These are Hilbert spaces with norm given by the (square root of the) integral required to be finite, and inner product given by its polarization. When $n=0$ or $m=0$, the spaces are denoted simply by $r^\nu H^m_{\text{b},w}$ or $r^\nu H^n_{e,w}$ respectively, and when $m=n=0$ by $r^\nu L^2_w$.  

 \end{defn}

 The Dirac operator extends to a bounded operator 
 \be D: r^{1+\nu} H^{m,1}_{\text{b},e,w}(Y\mathrm{-}\mathcal Z_0; S)\lre r^{\nu} H^{m}_{\text{b},w}(Y\mathrm{-}\mathcal Z_0; S)\label{diracoperator}\ee
 
 \noindent for every $\nu,m$. The fundamental consequences of the elliptic edge theory of this operator are the following: 
 
 \begin{lm} [\cite{MazzeoEdgeOperators, MazzeoHaydysTakahashi, PartII}]
For $-\tfrac12<\nu<\tfrac12$, the operator (\ref{diracoperator}) is left semi-Fredholm, i.e. has finite-dimensional kernel, and closed range. Moreover, for each $m$, there is a constant $C_{m,\nu}$ such that for every $\ph\in r H^{m,1}_{\text{b},e,w}$ the following estimates hold:
  \be \label{semielliptictobeuniform}\|\ph\|_{r^{1+\nu}H^{m,1}_{\text{b},e,w}} \leq C_{m,\nu} \left( \|D\ph\|_{r^\nu H^m_{\text{b},w}} \ + \ \|\ph \|_{r^\nu H^m_{\text{b},w}}\right).\ee

 \noindent  A similar estimate holds replacing the $\|\ph\|_{r^\nu H^m_{b,w}}$ term with the projection to a finite-rank subspace. 
  \label{semiFredholm}\qed
 \end{lm}
 
 \medskip
 
 \noindent 
The estimate (\refeq{semielliptictobeuniform}) is an {\it a priori} estimate on the chosen finite-energy domain \(r^{1+\nu}H^{m,1}_{b,e,w}\). Just as in standard elliptic theory, it implies that $D$ has finite-dimensional kernel and closed range. It differs from the standard elliptic estimates, however, insofar as the usual elliptic regularity statement for the maximal \(L^2\)-domain fails. In other words, if \(\phi\in L^2\) and \(D\phi\in r^\nu H^m_{b,w}\), one cannot in general conclude that \(\phi\in r^{1+\nu}H^{m,1}_{b,e,w}\). In the language of standard unbounded operator theory,  \(r^{1+\nu}H^{m,1}_{b,e,w}=\mathcal D_\text{min}(D)\) is the $L^2$-{\it minimal} domain given by the completion of compactly supported smooth sections, while the $L^2$-{\it maximal} domain $\mathcal D_\text{max}(D)$ is strictly larger and the quotient $\mathcal D_{max}(D)/ \mathcal D_{min}(D)$ is infinite-dimensional. In contrast, in standard uniformly elliptic theory on compact manifolds, one manifestation of elliptic bootstrapping is that the minimal and maximal domains coincide.

 
 The larger maximal domain is reflected in the asymptotic expansion near \(\mathcal Z\). More precisely, the general theory of \cite{MazzeoEdgeOperators} implies the following regularity result, which gives regular asymptotic expansions in local cylindrical coordinates $(t,r,\theta)$ around $\mathcal Z$, where $t$ is tangential to $\mathcal Z$ and $(r,\theta)$ are polar coordinates on the normal plane: 
 
 \begin{lm} \label{expansions} If $\Phi \in r^{1+\nu} H^{m,1}_{\text{b},e,w}$ for $-\tfrac12<\nu<\tfrac12$ and $D\Phi=0$, then
 
 \be \Phi \sim B(t,\theta)r^{1/2} + C_0(t,\theta)r^{3/2} + \sum_{k\geq 2}\sum_{j=0}^{k-1}  C_{jk}(t,\theta)\log(r)^j r^{k+1/2},\label{polyhomogeneous}\ee
 \noindent where $B,C_0,C_{jk}\in C^{\infty}$ are smooth sections, and $\sim$ means convergence in the sense that the partial sums $\Phi_N$ truncating \eqref{polyhomogeneous} at $k=N$ satisfy 
 $$| \nabla_t^\alpha \nabla_\theta^\beta \nabla_r^\gamma(\Phi-\Phi_N)|< C_{N,\alpha,\beta,\gamma} r^{N+1+\tfrac14-|\gamma|}$$
 \noindent for some constants $C_{N,\alpha,\beta,\gamma}$.\qed 
 \end{lm}
 
\noindent The non-degeneracy condition of Definition (\ref{regulardef}) is equivalent to the statement that $B(t,\theta)$ is nowhere-vanishing. Lemma \ref{expansions} shows that the kernel of (\refeq{diracoperator}) is independent of $\nu$ in the range $-\tfrac12<\nu<\tfrac12$. This kernel is, by definition, the set of $\Z_2$-harmonic spinors (resp. 1-forms), as this range includes the smallest weights for which the integrability condition of (\refeq{Z2harmonicequation}) holds. 

The failure of elliptic regularity also means solutions cannot be bootstrapped in the normal sense. In particular, an $L^2$-solution of $D\ph=0$ need not lie in $rH^1_e$. As a consequence, the kernel and cokernel of (\refeq{diracoperator}) need not coincide, despite the formal self-adjointness of $D$, as the cokernel may be associated with the (a priori larger) space of $L^2$-solutions. For $\nu$ in the same range as Lemma \ref{semiFredholm} this larger space consists of two pieces: a finite-dimensional summand and an infinite-dimensional summand. The finite-dimensional summand is the inclusion of the $rH^1_e$-kernel into the $L^2$-kernel. The infinite-dimensional summand consists of those $L^2$-solutions whose covariant derivative fails to be $L^2$. This space may be identified with the space of $L^2$-sections of a vector bundle on the singular set $\mathcal Z_0$, as the next proposition describes for $\nu=0$. 

Let $\mathcal C_0\subseteq S|_{\mathcal Z_0}$ denote the complex line bundle on $\mathcal Z_0$ whose fiber is the $+i$ eigenspace of $\gamma(dt)$. Note this vector bundle is canonically identified with the trivial bundle $\underline \C$. We use 

$$\text{\bf Ob}(\mathcal Z_0) := \text{Range}(D|_{rH^1_e})^\perp\cap L^2_{b,w}$$

\noindent to denote the orthogonal complement of the range (the ``obstruction''). 

\begin{prop}[{\cite[Sec. 4]{PartII}}] 
\label{cokernel}
There is a bounded linear isomorphism 
$$(\text{ob}, \iota): L^2(\mathcal Z_0;\mathcal C_0)\oplus \ker(D|_{rH^{1}_e}) \lre \text{\bf Ob}(\mathcal Z_0),$$
\noindent where $\iota$ is the inclusion. Moreover, $(\text{ob},\iota)$ respects regularity in the sense that its restriction to $H^m(\mathcal Z_0;\mathcal C_0)$ in the first summand has image equal to $\text{\bf Ob}(\mathcal Z_0)\cap H^m_{b,w}(Y\mathrm{-}\mathcal Z_0)$.\qed
\end{prop}

\medskip 

\noindent A complete proof of Proposition \ref{cokernel} is given in \cite[Sec. 4]{PartII}. To elaborate briefly, the $L^2$-solutions have expansions similar to (\refeq{polyhomogeneous}), but with an additional leading term $A(t,\theta)r^{-1/2}$, whose covariant derivative fails to be $L^2$. Roughly speaking, the proof of the proposition consists of showing that only the $e^{\pm i \theta/2}$-Fourier modes contribute and we may write $A=a(t)e^{\pm i\theta/2}$, after which the obstruction may be identified with this space of possible leading coefficients $a(t)$. Geometrically, the obstruction elements have support increasingly concentrated near $\mathcal Z_0$ as the Fourier modes of $a(t)$ increases (see \cite[Prop. 4.3]{PartII} for a precise statement).

\subsection{Deformations of Singular Sets}
As explained in the introduction, the infinite-dimensional obstruction of Proposition \ref{cokernel} prevents the use of the standard implicit function theorem, and the deformations of the singular set must be used to cancel the obstruction components. This section reviews the deformation theory of the singular set developed in \cite{PartII} (see also \cite{PartIII}, \cite{DonaldsonMultivalued}). 

Let $(\mathcal Z_0, \ell_0, \Phi_0)$ be a smooth, non-degenerate generalized $\Z_2$-harmonic spinor. Let $\mathcal U_0\subseteq \text{Emb}^{2,2}(\mathcal Z_0;Y)$ denote an open neighborhood of $\mathcal Z_0$ in the space of embeddings of Sobolev regularity $(2,2)$. For each $\mathcal Z\in \mathcal U_0$, there is a line bundle $\ell_{\mathcal Z}$ which may be identified with $\ell_0$ up to homotopy in the obvious way.

Let $p_1:r\mathbb H^{1}\to \mathcal U_0$ and $p_0:\mathbb L^2\to \mathcal U_0$ denote the Banach vector bundles whose fibers over $\mathcal Z$ are respectively $rH^1_{e,w}(Y\mathrm{-}\mathcal Z,S\otimes \ell_\mathcal Z)$ and likewise for $L^2_w$. Define the {\bf universal Dirac operator} as the section (over the total space of $r\mathbb H^1$) 
\smallskip
\be  \mathbb D: r\mathbb H^1\to p_1^\star\mathbb L^2 \hspace{3cm}\mathbb D(\mathcal Z,\Phi):=D_\mathcal Z\Phi.\ee

\noindent where $D_\mathcal Z$ is the version of $D$ formed using the singular set $\mathcal Z$. $\mathbb D$ is linear in the second argument, but fully non-linear with respect to the embedding.

\cite[Sec. 5]{PartII} describes a local trivialization which induces a splitting of the tangent space at $(\mathcal Z_0,\Phi_0)$ as $T(r\mathbb H^1)\simeq L^{2,2}(\mathcal Z_0; N\mathcal Z_0)\oplus rH^1_{e,w}(Y\mathrm{-}\mathcal Z_0)$, where the former is the tangent space at $\mathcal Z_0$ of $\text{Emb}^{2,2}(\mathcal Z_0;Y)$ and the latter is the tangent space of the fibers of $r\mathbb H^1$. The (covariant) derivative of $\mathbb D$ may be written as 
$$(\text{d}\mathbb D)_{(\mathcal Z_0, \Phi_0)}(\eta,\ph)=\mathcal B_{\Phi_0}(\eta) + D\ph,$$

\noindent where $(\eta,\psi)\in T(r\mathbb H^1)$, $\mathcal B$ is the partial derivative with respect to deformations, and the unadorned $D$ means the operator at $\mathcal Z_0$. Since $D$ carries $rH^1_e$ to its own range by definition, splitting the codomain $L^2\simeq \text{\bf Ob}(\mathcal Z_0)\oplus \text{Range}(D)$ gives the derivative the block-diagonal form 
$$ (\text{d}\mathbb D)_{(\mathcal Z_0, \Phi_0)} (\eta,\psi)=\begin{pmatrix}\Pi_0 \mathcal B_{\Phi_0} & 0 \\ \Pi_0^\perp \mathcal B_{\Phi_0} & D\end{pmatrix}\begin{pmatrix} \eta \\ \psi \end{pmatrix},$$

\noindent where $\Pi_0$ denotes the $L^2$-orthogonal projection to $\text{\bf Ob}(\mathcal Z_0)$. To show that deformations of the singular set may be used to cancel the infinite-dimensional obstruction (up to a finite-dimensional space), it suffices to show that the top left block is Fredholm. 

The partial derivative $\mathcal B$ may be calculated using the following trick. Let $\mathcal V_0$ be an open ball around $0\in L^{2,2}(\mathcal Z_0; N\mathcal Z_0)$. Take a family of diffeomorphisms $F_\eta: Y\to Y$ parameterized by $\eta \in \mathcal V_0$ such that $F_0=\text{Id}$ and $X_{\eta}:=\d{}{s}|_{s=0} F_{s\eta}$ is a vector field extending $\eta$ to $Y$. For $\mathcal V_0$ sufficiently small, the map $\eta \mapsto F_\eta[\mathcal Z_0]$ is a coordinate chart on the space of embeddings (see \cite[Sec 5.1]{PartII}). By the diffeomorphism invariance of the Dirac operator, { differentiating with respect to the embedding while keeping the metric $g_0$ fixed is equivalent to differentiating with respect to the family of pullback metrics $g_\eta=F_\eta^*(g_0)$ while keeping $\mathcal Z_0$ fixed}. The formula of Bourguignon--Gauduchon for the derivative of the Dirac operator with respect to the metric gives the following expression.


\begin{lm} \label{BG}The partial derivative $\mathcal B_{\Phi_0}$ is given by 
$$B_{\Phi_0}(\eta)=\left[ -\tfrac{1}{2}(\dot g_\eta)_{ij}e^i.\nabla_j + \tfrac{1}{2}d\text{Tr}(\dot g_\eta). +\tfrac12 \text{div}(\dot g_\eta).\right]\Phi_0$$
\noindent where $e^i, \nabla, .$ are a coframe, the spin/Levi-Civita connection, and Clifford multiplication of the metric $g_0$, and $\dot g_\eta=\d{}{s}|_{s=0}g_{s\eta}$. \qed
\end{lm}

\noindent The first term arises from differentiating the symbol of $D$, and the latter two from differentiating the Christoffel symbols. Note that this should be viewed as an equation in $\eta$ (thus the last two terms are actually leading order, as they contain second derivatives of $\eta$).

Pre-composing with the map from Proposition \ref{cokernel}, this partial derivative may be viewed as a map

\smallskip
$$\mathcal T_{\Phi_0}:= \text{ob}^{-1}  \Pi_0 \mathcal B: C^\infty(\mathcal Z_0; N\mathcal Z_0)\lre L^2(\mathcal Z_0 ;\mathcal C_0).$$

\smallskip

\noindent $\mathcal T_{\Phi_0}$ is a map on sections of vector bundles on $\mathcal Z_0$, and will be referred to as the {\bf deformation operator}. 

The main result that allows the cancellation of the infinite-dimensional obstruction is the following: 

\begin{thm}[{\cite[Thm. 6.1]{PartII}}] $\mathcal T_{\Phi_0}$ is an elliptic pseudodifferential operator of order $\tfrac12$ whose Fredholm extension has index 0. Moreover, there are constants $C_m$ such that the elliptic estimate
\be \|\eta\|_{H^{m+1/2}(\mathcal Z_0;N\mathcal Z_0)} \ \leq \ C_m \left( \| \mathcal T_{\Phi_0}(\eta)\|_{H^m(\mathcal Z_0;\mathcal C_0)} \ + \ \|\ph\|_{H^{m+1/4}(\mathcal Z_0;N\mathcal Z_0)}\right)\label{TPhiestimates}\ee
\noindent holds for all $m\geq 0$. \qed
\end{thm}

As a consequence: 

\begin{cor} \label{fredholmddD} The derivative 

\be (\text{d}\mathbb D)_{(\mathcal Z_0, \Phi_0)}=\begin{pmatrix}\Pi_0 \mathcal B_{\Phi_0} & 0\bigskip \\ \Pi_0^\perp \mathcal B_{\Phi_0} & D\end{pmatrix}: \begin{matrix} L^{2,2}(\mathcal Z_0;N\mathcal Z_0) \\ \oplus  \\ rH^1_e(Y\mathrm{-}\mathcal Z_0)  \end{matrix}\lre \begin{matrix} \text{ \bf {Ob}}(\mathcal Z_0)\cap H^{3/2}_{b,w} \\ \oplus \\ \text{Range}(D|_{rH^1_e})\cap L^2_{w}\end{matrix}\label{universalderiv}\ee
\noindent is a Fredholm operator of Index 0. \qed
\end{cor}

\noindent Note that the range component $\Pi_0^\perp\mathcal B_{\Phi_0}$ is only bounded into $L^2$ for $\eta \in L^{2,2}$, but $\mathcal T_{\Phi_0}$ is of order $\tfrac12$, which necessitates the different regularities on the summands of the codomain. The non-linear portion of $\mathbb D$, however, is not necessarily bounded into the higher regularity cokernel, thus $\mathbb D$ displays a {\bf loss of regularity}. 

Loss of regularity is a general phenomenon, intrinsic to some classes of non-linear PDEs \cite{HamiltonNashMoser}. It may be viewed, morally speaking, as a mismatch between the natural function spaces necessary for controlling the linear versus the non-linear portions of the theory. In applying a standard implicit function theorem on Banach spaces, it is necessary to have a Banach space for which there is simultaneously a Fredholm theory for the linearization, as well as sufficient control on non-linearities that the full non-linear map is $C^1$. In classical settings, this can often be achieved by developing elliptic/Schauder estimates in Sobolev or H\"older spaces above the borderline for multiplicative estimates, which allow control of large classes of non-linearities. In contrast, a loss of regularity often occurs when the Fredholm theory for the linearization demands the use of function spaces that are too restrictive to be closed under multiplication or other non-linearities. The result is that in each step of a fixed point iteration, the new error terms ``spill out'' into a larger function space, and must be mollified so that the linear theory applies again. 

In the specific setting of Corollary \ref{fredholmddD}, the demand that the lower left component $\Pi_0^\perp \mathcal B_{\Phi_0}$ be bounded into $L^2_w$ requires that the deformations be taken in $L^{2,2}(\mathcal Z_0; N\mathcal Z_0)$. The fact that $\mathcal T_{\Phi_0}$ is order $<2$, however, then shows that $\text{d}\mathbb D$ has image which is a proper subspace of $L^2_w$, consisting of higher regularity sections in the $\text{\bf Ob}(\mathcal Z_0)$ components. The non-linear terms, however, mix these components and will in general land only in $L^2_w$ in {\it both} components. This shift of regularity $3/2$ is precisely the loss under our conventions; it cannot be eliminated by simply working in higher regularity Sobolev spaces. Nash-Moser theory is the standard tool for overcoming such a loss of regularity. 
\subsection{An Implicit Function Theorem for Generalized $\Z_2$-Spinors}
\label{section2.3}
Nash-Moser theory provides a standard framework for dealing with operators that lose regularity by working in the category of tame Fr\'echet spaces \footnote{By changing the weight $\nu$, there are Fr\'echet spaces so that the loss of regularity here is of order $\delta$ for any $\delta>0$. It is an interesting question to ask if there is a setting where the use of Nash-Moser theory can be eliminated}. Versions of the Nash-Moser implicit function theorem suitable for $\Z_2$-harmonic spinors and 1-forms were developed in \cite[Thm. 1.4]{PartII} and \cite[Thm. 1]{DonaldsonMultivalued}. In this subsection, we unify these approaches and prove a slight extension applicable to the current setting. Here, $\mathcal P$ denotes the space of parameters $p=(g,B)$.

As with many version of the implicit function, the statement requires {\it uniform} control over the linear theory with respect to parameters. "

\begin{defn} \label{strictlyuniform}A family of tuples $\{(g_T, B_T, \mathcal Z_T) \ | \ T\in \Omega \subseteq \R^M\}$ is said to be {\bf strictly uniform} if there exists a smooth, finite-rank vector bundle $\mathcal K\to \Omega$ such that $K_T\subset \bigcap_{m\geq 0} r^{1+\nu}H^{m,1}_{\text{b},e,w} \ \forall T\in \Omega$, and there are constants $C_{m,\nu}$ uniform in $T$ such that the elliptic estimates (\refeq{semielliptictobeuniform}) reduces to 

  \be \label{strictlyuniform}\|\ph\|_{r^{1+\nu}H^{m,1}_{\text{b},e,w}} \leq C_{m,\nu}  \|D_T\ph\|_{r^\nu H^m_{\text{b},w}}\ee
  
  \noindent for $\ph \perp_{L^2_w} \mathcal K_T$, where $D_T$ denotes the Dirac operator on the complement of $\mathcal Z_T$ with parameters $(g_T, B_T)$.  
\end{defn}

\noindent This condition is stronger than simply requiring that (\refeq{semielliptictobeuniform}) holds uniformly, as it also uniformly controls the dimension of a finite-dimensional thickening of the kernel.

Our extension of the Nash-Moser theorems of \cite{PartII, DonaldsonMultivalued} is the following. 

\begin{thm}[{\cite[Thm. 1.4]{PartII} , \cite[Thm. 1]{DonaldsonMultivalued}}] 
\label{NMIFT}
Suppose that $Y$ is a closed, oriented Riemannian 3-manifold and $p_T=(g_T,B_T)$ are a 1-parameter family of metric and perturbation pairs parameterized by $T\in [T_0,\infty)$. If $(\mathcal Z_T, \ell_T, \Phi_T)$ are a corresponding family of smooth, weakly non-degenerate approximate (generalized) $\Z_2$-harmonic spinors such that
 
 \begin{itemize}
 \item[(I)] the tuples $(g_T, B_T, \mathcal Z_T)$ form a strictly uniform family as in Definition \ref{TPhiestimates},
 \item[(II)]$\Phi_T$ obey  

$$\|D_{Z_T}\Phi_T\|_{H^{m_1}_{\text{b},w}} \ \overset{T\to \infty}{\lre} \  0, \hspace{1.5cm} \text{and}\hspace{1.cm}\text{supp}(D_{Z_T}\Phi_T) \Subset Y\mathrm{-}\mathcal Z_T,$$ for $m_1$ sufficiently large, and
\item[(III)] $(p_T, \mathcal Z_T, \ell_T,\Phi_T)$ are constant on a tubular neighborhood of $\mathcal Z_{T_0}$, 
\end{itemize}
then there is a $T_1\geq T_0$ such that the following holds. 

 There is a finite-dimensional vector space $V$ with a linear inclusion $V\hookrightarrow \mathcal P$, and for $T \geq T_1$ there exist 
triples $(\mathcal Z_T', \ell_T', \Phi'_T)$ and parameters $b_T \in V$ all defined implicitly as smooth functions of $T$ such that 
\be D_{\mathcal Z'_T} \Phi'_T =0\ee
\noindent with respect to $p'_T=p_T + b_T$, i.e. $(\mathcal Z_T', \ell_T', \Phi'_T)$ are generalized $\Z_2$-harmonic spinors. Moreover, each of these is smooth and weakly non-degenerate. In fact, $b_T$ can be chosen to be supported on a small ball $B_\delta\Subset Y\mathrm{-}\mathcal Z_T'$ of radius $\delta<<1$, and can be taken to be identically zero in the case of $\Z_2$-harmonic 1-forms provided $\Phi_T$ are closed. 
\end{thm}
\begin{proof}
	The result is a generalization of the version of the Nash--Moser implicit function theorem established in \cite[Secs. 7--8]{PartII}, with three minor extensions. Following \cite{PartII}, \((\mathcal Z_0,\ell_0,\Phi_0)\) is called {\it isolated} if, for the fixed pair \((\mathcal Z_0,\ell_0)\), the spinor \(\Phi_0\) is the unique \(\mathbb Z_2\)-harmonic spinor up to normalization and sign. The extensions are: (i) the theorem holds uniformly in one-parameter families provided the family is strictly uniform; (ii) the assumptions that \((\mathcal Z_0,\ell_0,\Phi_0)\) is isolated and that \(\mathcal T_{\Phi_0}\) is an isomorphism may be removed, at the cost of adding the perturbations \(b_T\); and (iii) the theorem also applies in the case of \(1\)-forms, thus subsuming the results of \cite{DonaldsonMultivalued} in this context.


{(i) Due to the strict uniformity assumption (I), \cite[Thm.~7.4(B)]{PartII} immediately implies the result for 1-parameter families, provided the relevant tame estimates all hold uniformly. We spell out the source of this uniformity. The construction of the tame right inverse in \cite[Sec.~8.5]{PartII} uses the semi-Fredholm estimate for \(D\), relying only on the existence of a finite-dimensional space of spinors on whose $L^2$-complement the elliptic estimates are uniform, together with the elliptic estimate for the deformation operator \(\mathcal T_{\Phi_T}\). The former is precisely (I). Because the deformations may be taken to lie inside the tubular neighborhood of \(Z_{T_0}\) on which the data is constant as in (III) thus $B_{\Phi_T}$ is independent of $T$. The theory of \cite[Sec. 4]{PartII} shows that strict uniformity of $D$ and constancy on a neighborhood of $\mathcal Z_T$ implies that the deformation operators $\mathcal T_{\Phi_T}$ are also strictly uniform, i.e. (\refeq{TPhiestimates}) holds uniformly, and the compact term may be replaced by the projection to a finite-dimensional smooth bundle $\mathcal K'_T\subseteq \text{\bf Ob}(\mathcal Z_\tau) $. With this uniformity, the same construction as \cite[Sec. 8]{PartII} gives a family of tame right inverses with tame estimates uniform in \(T\). The fact that the corrected solutions are smooth and weakly non-degenerate follows just as in \cite[Thm.~1.4]{PartII}.}

(ii) Consider the case of the spin Dirac operator $D=\slashed D$. Let $k_1=\text{rank}(\mathcal K_T)$ and $k_2=\text{rank}(\mathcal K'_T)$ be the ranks of the bundles giving the strict uniformity of $D,\mathcal T$. The cokernel of (\refeq{universalderiv}) restricted to $(\mathcal K_T \oplus K_T')^\perp$ is a subspace $\mathbb K\subseteq \text{Coker}(\slashed D|_{rH^1_e})=\ker(\slashed D|_{L^2})$ of dimension $K=k_1+k_2$. Let $\Psi_1, \ldots, \Psi_K$ denote an $L^2$-orthonormal basis of this space, and $U_1,..,U_K$ open balls around a collection of points $y_1,..,y_K$ so that $U_j\cap N(\mathcal Z_{T_0})=\emptyset$.  By the unique continuation property of $\slashed D$, each $\Psi_j$ is non-vanishing on each ball $U_j$. We consider the class of perturbations which take the form \footnote{Perturbations of this form are those that arise from background $SU(2)$-connections in the gauge theory setting (Item ii in the Introduction)} 
$$B=\sum_{j=1}^3 (i\alpha_k + \beta_k J)e^j$$
\noindent in a local orthonormal frame, where $\alpha_k\in C^\infty(Y;\R)$, $\beta_k\in C^\infty(Y;\C)$, and $J: S\to S$ is a complex anti-linear endomorphism with $J^2=-\text{Id}$. Writing $\Psi_k=\Psi_k(y_j) +O(\rho)$, it is straightforward to check that this class of perturbations is sufficiently large to choose $b_1,...,b_K$ supported on the respective balls $U_k$ so that $$\br b_j \Phi_T, \Psi_j\kt\neq 0$$

\noindent (and is bounded below uniformly in $T$). The augmented universal operator $$\overline{\slashed {\mathbb D}}(\mathcal Z,\Phi, \lambda_k)=\slashed{\mathbb D}(\mathcal Z,\Phi) + \sum\lambda_k b_k(\Phi)$$ for $(\lambda_1,..,\lambda_K)\in \R^K$ has surjective derivative by design, and the implicit function theorem applies as before to yield solutions, which now define $b=\sum \lambda_k b_k$ implicitly as smooth functions of $T$. 

(iii) We now deduce the theorem in the case of a family of 1-forms $\Phi_T=(0,\nu_T)$ and $D=D_{\mathrm{dR}}$ from the case for spinors. For this, we can take advantage of the fact that $\Omega^0\oplus \Omega^1$ and the spinor bundle $\slashed S$ on a closed 3-manifold are isomorphic as real Clifford modules. In fact, in a local orthonormal coframe $1, \omega_t, \omega_x, \omega_y$, the map $\Upsilon: \Omega^0\oplus \Omega^1 \to \slashed S$ defined by 
\bea \begin{pmatrix} a_0 \\ a_t \omega_t + a_x \omega_x + a_y\omega_y\end{pmatrix}&\mapsto& \begin{pmatrix} -a_y+ia_x \\ -a_t - a_0\end{pmatrix}  \hspace{1.0cm}\Rightarrow \hspace{1.0cm} \Upsilon D_{\mathrm{dR}} \Upsilon^{-1}= \slashed D+ \mathfrak a  \eea

\noindent is such an isomorphism, which carries $D_{\mathrm{dR}}$ to the spin Dirac operator with a zeroth order perturbation $\mathfrak a$. The setting of \cite{PartII} may therefore be applied to $D_{\mathrm{dR}}$, with the following distinction. In this case we consider the unperturbed operator $D_{\mathrm{dR}}$, so must show that a solution can be found without altering the perturbation $\mathfrak a$, making the approach of (ii) invalid here. Moreover, $D_{\mathrm{dR}}$ has a topologically mandated kernel coming from $L^2$-Hodge theory as explained in the introduction. This case therefore also carries an additional finite-dimensional obstruction from $L^2$-harmonic forms, and an additional finite-dimensional parameter given by the cohomology class $[\nu] \in H^1_-(Y_{\mathcal Z_{T_0}})$. 

Let $K=\dim H^1_-(Y_{\mathcal Z_{T_0}};\R)$ and choose closed 1-forms $\alpha_1,\ldots, \alpha_K \in rH^1_e(\Omega^1)$ such that $p^*\alpha_i$ span $H^1_-(Y_{\mathcal Z_{T_0}};\R)$. We may assume that the first $\dim(\ker(D_{\mathrm{dR}}|_{rH^1_e}))$ of the $\alpha_j$ coincide with the $\Z_2$-harmonic 1-forms $\alpha_j =\nu_j\in rH^1_e(\Omega^1)$. Next, let $\psi_1,...,\psi_K \in L^2(\Omega^1)$ denote the $L^2$-harmonic forms such that $p^*(\psi_j)$ span $H^1_-(Y_{\mathcal Z_{T_0}};\R)$. It may again be assumed that the first several are the $\Z_2$-harmonic 1-forms. Set 
\bea
\mathcal H_1:= \text{Span}\{\nu_1,...,\nu_k, \alpha_{k+1},\ldots, \alpha_K\},\\
\mathcal H_0:= \text{Span}\{\nu_1,...,\nu_k, \psi_{k+1},\ldots, \psi_K\}.
\eea
Note that $\psi_{k+1},...\psi_K\in \text{Im}(ob)$ are part of the infinite-dimensional piece of the obstruction from Proposition \ref{cokernel} (these are precisely the obstruction elements with no zero-form component). The same applies for nearby pairs $(p',\mathcal Z')\in \mathcal P\times \mathcal U_{0}$, thus $\mathcal H_1, \mathcal H_0$ form smooth vector bundles over this space.

Consider the restricted universal Dirac operator 

$$\underline{\mathbb D}: r\mathbb H^1_\perp \to p_1^*\mathbb L^2_\perp,$$

\noindent where $r\mathbb H^1_\perp, \mathbb L^2_\perp$ denote the $L^2$-orthogonal complements of $\mathcal H^1, \mathcal H^0$ respectively. Since these spaces are of the same finite dimension, Corollary \ref{fredholmddD} implies $\underline{\mathbb D}$ still has index $0$. We now make two claims: 

\begin{enumerate}\item[ (iiia):] In each stage of the Nash-Moser iteration, the error is orthogonal to $\mathcal H_0\subseteq p_1^*\mathbb L^2$. 

To see this, note that only the $-d^\star$ component of $D_{\mathrm{dR}}$ depends on the metric (and thus on $\mathcal Z$). Since $\nu_T$ is closed by assumption, it follows that the initial error $\mathfrak e_T\in \Omega^0(\ell)$ has no 1-form components, and that the partial derivative $\mathcal B_{\nu_T}(\eta) \in \Omega^0(\ell)$ for all $\eta$ as well. Because $\mathcal H_0\subseteq \Omega^1(\ell)$, the image of $\underline{\mathbb D}$ is automatically orthogonal at $(\mathcal Z_T,\nu_T)$. Moreover, 1-form component of the solution to the linearized equation is always in $d\Omega^0(\ell)\subseteq \Omega^1(\ell)$, thus the correction term may be assumed to preserve closedness of the approximate solution. Finally, it may easily be arranged that the smoothing operators preserve closedness and the properties of being orthogonal to $\mathcal H_1, \mathcal H_0$. Applying the same argument inductively shows that the entire iteration remains in $r\mathbb H^1_\perp, \mathbb L^2_\perp$ 

\item[(iiib):] In this case, $\text{d}\underline{\mathbb D}$ is automatically an isomorphism. 

Since the bottom right block of Corollary \ref{fredholmddD} is injective on the complement of $\mathcal{H}_1$ by construction, a kernel element would necessarily be of the form $(\eta, \psi)$, where $\eta \neq 0$, and would have to solve 
$$d^{ \dot \star_\eta} \nu_T + d^{\star} \psi=0, \quad d\psi=0, \quad \psi \in r\mathbb{H}^1_\perp,$$
where $\dot \star_\eta := \frac{d}{ds}|_{s=0} \star_{F_{s\eta}^* g}$ is the Hodge star operator of the metric $\dot g_\eta$, and $\star$ is that of $g_T$.

Take $X_{\eta}:=\d{}{s}|_{s=0} F_{s\eta}$ be an extension of the vector field $\eta$ as in Lemma \ref{BG} (see also \cite[Sec. 5]{DonaldsonMultivalued}). Observe that $F_{s\eta}^*(d^{\star}\nu_T) = (d^{\star_{ {s\eta}}}F_{s\eta}^*(\nu_T))$, where $g_{s\eta}:=F_{s\eta}^*g_T$. Taking the derivative at $s=0$, we obtain $$d^{ \dot \star_\eta} \nu_T = \mathcal{L}_{X_{\eta}} d^{\star} \nu_T + d^{\star} (\mathcal{L}_{X_{\eta}} \nu_T)=\iota_{X_{\eta}} d d^{\star} \nu_T + d^{\star} d (\iota_{X_{\eta}} \nu_T) = d^{\star} d (\iota_{X_{\eta}} \nu_T).$$ 

\noindent Because $\psi \perp \mathcal{H}_1$, we can express $\psi = d f_{\psi}$ for $f_{\psi} \in r H_e^1$. Thus, this would imply
\be 
\Delta (f_{\psi} + \iota_{X_{\eta}} \nu_T) = 0,
\ee
with $f_{\psi} + \iota_{X_{\eta}} \nu_T \in \Omega^0(\ell)\cap r H^1_e$. By \cite[Sec. 2]{DonaldsonMultivalued}, it follows that $f_{\psi} + \iota_{X_{\eta}} \nu_T = 0$.

However, since $\nu_T$ is non-degenerate, by \cite[Page 18]{DonaldsonMultivalued}, near $\mathcal{Z}$, we can locally write $\nu_T = \mathrm{Re}(B z^{\frac{1}{2}}dz) + O(r^{\frac{1}{2} + \epsilon})$ with $B$ nowhere vanishing. If $\eta$ is non-trivial, then $\iota_{X_{\eta}}\nu_T$ will have a non-vanishing $r^{-\frac{1}{2}}$ leading coefficient. Consequently, $\iota_{X_{\eta}}\nu_T \notin rL^2$, whereas $f_{\psi} \in rL^2$, leading to a contradiction if both are non-zero. Therefore, $\eta = f_{\psi} = 0$, which implies the claim.

\end{enumerate}

\noindent The two claims combine to show that an approximate solution may be corrected to a true solution without introducing perturbations $b_T$ in the 1-form case. 
\end{proof}

\begin{rem}Theorem \ref{NMIFT} extends to multi-parameter families with only minor modifications. 
\end{rem}

We conclude this section with a proof of Corollary \ref{nonempty}
\begin{proof}[Proof of Corollary \ref{nonempty}] This follows from a slight extension of  \cite[Thm. 1.6.]{PartIII}. A generalized $\Z_2$ harmonic spinor $(\mathcal{Z}, \ell, \Phi)$ is called isolated if $\Phi$ is the unique $\Z_2$-harmonic spinor for the pair $(\mathcal Z, \ell)$ with respect to $p=(g,B)$ up to normalization and sign. The proof of \cite[Thm. 1.6]{PartIII} assumes that the given $\Z_2$-harmonic spinor is isolated and (strongly) non-degenerate. It is not expected, however, that the solutions constructed by Corollary \ref{spinorsb} are isolated (as those of Corollary \ref{example1.8}b are not), and may be only weakly degenerate.

The isolated assumption in \cite[Thm. 1.6.]{PartIII} may be eliminated by adapting the argument of part (ii) in the proof of Theorem \ref{NMIFT} above, using perturbations to cancel the finite-dimensional obstruction arising from nearby $\Z_2$-harmonic spinors. 

To conclude, we show that a smooth weakly non-degenerate $\Z_2$-harmonic spinor may be perturbed to a (strongly) non-degenerate one. In the case that $\mathcal Z=\emptyset$, the set of parameters $p=(g,B)$ whose harmonic spinors are all nowhere-vanishing is residual in the space $\mathcal P$ of all parameters. Indeed, a similar argument to part (ii)  in the proof of Theorem \ref{NMIFT} shows the universal derivative of the section
\bea
\mathcal P \times \R \times Y \times \mathbb S H^1_e(Y;S)&\lre & S \times L^2(Y;S)\\
(p,\lambda, y, \ph)&\mapsto& (\ph(y) , (\slashed D_p -\lambda)\ph) 
\eea
\noindent is transverse to the zero-section, where $\mathbb S$ denotes the unit sphere in the $L^2$-norm. The genericity of nowhere-vanishing spinors then follows from applying the Sard-Smale theorem to the projection to $\mathcal P$ restricted to the pre-image of $0$, and then intersecting with the locus $\lambda^{-1}(0)$. The argument for $\mathcal Z\neq \emptyset$ is similar, now using the operator $\slashed {\mathbb D}-\lambda$ and invoking the version of the Sard-Smale Theorem for Fr\'echet manifolds \cite[Thm. 4.3]{SardSmaleFrechet}. 
\end{proof}

\section{Gluing Analysis}
\label{section3}
This section proves the analytic input needed to apply Theorem~\ref{NMIFT} to the connected sum
\(Y=Y_1\#Y_2\) and the torus sum \(Y=Y_1\cup_{K_1=K_2}Y_2\). More precisely, we construct a family of
metrics \(g_T\) on the glued manifold and approximate solutions modeled on
\eqref{approximateZ2harmonic} and \eqref{approximateZ2harmonic2}. We then show the hypotheses (I)--(III) of Theorem \ref{NMIFT} are satisfied. (III) holds by fiat, thus we first
the error by which these approximate solutions fail to satisfy the Dirac equation tends to zero as
\(T\to\infty\) which is (II); second, we show the degenerating Dirac operators are strictly uniform as in Definition \ref{TPhiestimates}, i.e. they obey the estimates needed to construct uniformly bounded tame right
inverses.

The proof has the following structure. We first construct approximate solutions on the glued manifold.
For spinors on connected sums, this is done using conformal invariance and a standard neck-stretching
construction. For \(1\)-forms, the Hodge-de Rham operator is not conformally invariant, and the
approximate solutions must instead be constructed from local primitives so that they remain closed.
For torus sums, one uses a pinched toroidal neck and a corresponding model problem. After constructing
the approximate solutions, we prove uniform estimates for the degenerating operators. This is done by
analyzing the model operator on the neck region and then patching parametrices from the two summands
and from the neck. Once the error tends to zero and these uniform estimates are established, the
Nash--Moser theorem applies to correct the approximate solutions to genuine
\(\mathbb Z_2\)-harmonic spinors or \(1\)-forms.

Subsection~\ref{subsection_connected_sum} treats the spin Dirac operator \(D=\slashed D\) for connected
sums via neck stretching. The remaining cases require neck-pinching arguments, for which the operator
on the neck becomes singular. Subsection~\ref{subsection2dneck} analyzes the relevant model neck
operators, while Subsections~\ref{subsectionneckpinchingsphere} and~\ref{subsectiontoruspinching}
complete the proof of Theorem~\ref{maina} in the \(1\)-form case and the proof of
Theorem~\ref{mainb}, respectively.

\subsection{Neck Stretching for the Dirac Operator} 
\label{subsection_connected_sum}
In this section we consider the spin Dirac operator $D=\slashed D$. 
In this case, conformal invariance of the operator can be utilized to establish uniform elliptic estimates on connected sums via neck-stretching arguments. Such arguments have been standard in gauge theory for several decades, and we provide only a brief summary here, referring the reader to \cite{KM, DonaldsonAlternating, TomsNotes} for similar arguments. 

Let $(\mathcal{Z}_i, \ell_i, \Phi_i)$ be $\Z_2$-harmonic spinors on $(Y_i,g_i)$ for $i=1,2$. The connected sum $Y = Y_1 \# Y_2$ at points $y_i \in Y_i - \mathcal{Z}_i$ can be endowed with a family of metrics $g_T$ for which the tubular neck has length $O(T)$, constructed as follows. In geodesic normal coordinates around each $y_i$, the metric $g_i$ can be written as
\begin{equation} 
\label{eq_metric_gi} 
g_i = d\rho^2 + \rho^2 g_{S^2} + h_i,
\end{equation}

\noindent where $\rho$ is the distance to $y_i$ and $h_i = O(\rho^2)$. Defining $s = -\log(\rho)$ so that $\rho \to 0$ as $s \to \infty$, the metric can now be written
\begin{equation} 
\label{eq_metric_gii}
g_i = e^{-2s} (ds^2 + g_{S^2}) + O(e^{-4s}).
\end{equation}

\noindent Next, for $\rho_0$ small, let $\chi_i$ be a cut-off function supported in $B_{\rho_0}(y_i)$ and equal to 1 on $B_{\rho_0/2}(y_i)$. The conformal transformation $e^{-u_i}$ for $u_i=\chi_i\cdot s_i$ induces a conformal equivalence between $(Y- y_i,g_i)$ and $(Y'_i, g_i'):=(Y-y_i, e^{-u}g_i)$. The primed version has an infinite cylindrical end diffeomorphic to $[t_0, \infty)\times S^2$ for some $t_0$, equipped with the metric $ds_i^2 + g_{S^2} + h_i',$ where $h_i'=O(e^{-2s})$. 

The connected sum may now be formed by simply patching the manifolds with truncated ends $[s_0, 3T] \times S^2$ along their common boundary at $3T$ for $T>>s_0$. Revise notation so that $s$ now denotes the centered coordinate on the cylindrical neck of $Y$, with $s \in [-3T + s_0, 3T - s_0]$. The metric is then defined by

\begin{equation}
g_T = ds^2 + g_{S^2} + \zeta_1 h_1' + \zeta_2 h_2',\label{gTdef}
\end{equation}

\noindent where $\zeta_i$ are a partition of unity with $d\zeta_i$ supported in $[-T, T] \times S^2$ and $|d\zeta_i|=O(T^{-1})$. Finally, define the weight $w_i = e^{u_i/4}$ on each end; these can be smoothly melded into a single weight given by

\begin{equation}
w_T = \frac{e^{3T/4}}{2\cosh(s/4)}\label{wdef}
\end{equation}

\noindent in the centered coordinate on the neck, and constant on the two ends.
\subsubsection{Conformal Changes of the Dirac Operator}
On both $Y_i$, the two conformally equivalent metrics $g_i, g_i'$ each give rise to a spinor bundles and Dirac operator, denoted by $\slashed{S}_i, \slashed{S}_i'$ and $\slashed{D}_i, \slashed{D}_i'$, respectively (the transformation of the perturbation $B_i$ will be clarified shortly).

These two spinor bundles may be associated as follows (see \cite[Sec. 5]{LawsonMichelson}, \cite{DiracVariationBar, Bourguignon}). Let $X_i = [0, 1] \times (Y_i - y_i)$ be equipped with the metric $d\sigma^2 + e^{2\sigma u}g_i$, so that the cross-sections at $\sigma = 0, 1$ are $Y_i - y_i$ with the metrics $g_i, g_i'$ respectively in \eqref{eq_metric_gi}, \eqref{eq_metric_gii}. Let $W^+_i \to X_i$ denote the positive spinor bundle associated with the spin structure pulled back from that inducing $\slashed{S}_i$ on $Y_i$, and let $\nabla$ be the associated spin connection. The restrictions of $W^+_i$ to $\sigma = 0, 1$ are canonically isomorphic to $\slashed{S}_i, \slashed{S}_i'$, respectively. Let 
\begin{equation}
\tau_u: \slashed{S}_i \to \slashed{S}_i' \label{tauudef}
\end{equation}

\noindent be the fiberwise isometry defined by parallel transport using $\nabla$ along rays $[0, 1] \times y$ for $y \in Y_i - y_i$. We may now also define the transformed perturbation by $B_i' := \tau_u B_i \tau_u^{-1}$.

Next, we define 
\begin{equation}
\mathfrak{T}_u = e^{-u} \tau_u,
\end{equation}
then the conformal change formula for the Dirac operator (cf. \cite{HitchinHarmonicSpinors}, \cite[Thm. 5.24]{LawsonMichelson}) states:

\begin{prop}\label{conformalchange}
The Dirac operators $\slashed{D}_i, \slashed{D}_i'$ are related by 
\begin{equation}
\slashed{D}_i' = \mathfrak{T}_u \circ \slashed{D}_i \circ \mathfrak{T}_u^{-1}.
\end{equation}
\end{prop}
\begin{proof}
A proof is given in \cite[Thm. 5.24]{LawsonMichelson} for the unperturbed case. Since $B_i$ is of zeroth order, it commutes with multiplication by $e^{\pm u}$, and $B_i' = \tau_u B_i \tau_u^{-1}$ by definition.
\end{proof}

\medskip 
Although $\mathfrak{T}_u$ is a fiberwise isometry, the induced map on $L^2$-sections is not uniformly bounded since the conformal change $g_i' = e^{u} g_i$ also affects the volume form by $dV'=e^{3u/2}dV$. The exponent in $e^{u_i/4}$, and thus the weight $w_T$ used to define (\ref{wdef}) are chosen precisely to compensate for this. A straightforward computation (use the fact that $\tau$ being parallel mean $
\nabla \tau_u^{-1} = \tau_u^{-1} \nabla'$ for $\nabla, \nabla'$ the spin connections) shows:

\begin{lm} \label{weightiso}
The map induced by $\mathfrak{T}_u$ extends to a linear isomorphism
$$\mathfrak{T}_u: rH^{m,1}_{\text{b},e}(Y_i; \slashed{S}_i) \to rH^{m,1}_{\text{b},e,w}(Y_i'; \slashed{S}_i')$$

\noindent uniformly bounded in $T$, with a uniformly bounded inverse, where $w=w_T$ is as in (\refeq{wdef}).
\end{lm}

The following lemma shows that the conformally transformed $\Z_2$-harmonic spinor is an increasingly good approximate solution as $T \to \infty$. Let $\chi_\circ$ denote a new cut-off function equal to 1 on $(-\infty, 0)$ and vanishing on $[1, \infty)$; set $\chi^T_i = \chi_\circ(s_i / T - 2T)$ where $s_i \in [s_0, s_0 + 3T]$ is now the coordinate on the cylindrical end of $Y_i'$. Set
\be {\Phi_i^T := \chi_i^T \cdot \mathfrak{T}_u(\Phi_i).}\label{PhiiTdef}\ee
\begin{lm} \label{spinorapproximatesol}
For each $m \in \mathbb{N}$, there exist $T$-independent constants $C_m > 0$ such that

$$\|\slashed{D}_i' \Phi_i^T\|_{H^m_{\text{b},w}} \leq \frac{C_m}{T}.$$
\end{lm}

\begin{proof}
Using Proposition \ref{conformalchange}, we compute
\begin{align*}
\slashed{D}_i' \Phi_i^T &= \gamma'(d\chi_i^T) \mathfrak{T}_u \Phi_i + \chi_i^T \slashed{D}_i' \mathfrak{T}_u \Phi_i \\
&= \gamma'(d\chi_i^T) \mathfrak{T}_u \Phi_i + \chi_i^T \mathfrak{T}_u \slashed{D}_i \mathfrak{T}_u^{-1} \mathfrak{T}_u \Phi_i \\
&= \gamma'(d\chi_i^T) \mathfrak{T}_u \Phi_i + \cancel{\chi_i^T \mathfrak{T}_u \slashed{D}_i \Phi_i}.
\end{align*}

\noindent Since $d^m \chi_i^T = O(T^{-m})$, and by Lemma \ref{weightiso}, we have $\|\mathfrak{T}_u \Phi_i\|_{H^m_{\text{b},w}} \sim \|\Phi_i\|_{H^m_{\text{b}}}$. The result now follows where $C_m$ bounds the $H^m_{\text{b}}$-norm of the original spinor $\Phi_i$.
\end{proof}

\subsubsection{Parametrix Patching}

Let $\slashed{D}$ denote the spin Dirac operator on $(Y\mathrm{-}\mathcal Z, g_T)$, formed using the perturbation $B_T' = \zeta_1 B_1' + \zeta_2 B_2'$, where $\zeta_i$ are as defined in \eqref{gTdef} and $B_i'$ as below (\refeq{tauudef}). The following proposition proves that $D$ is strictly uniform in the sense of Definition \ref{strictlyuniform} in this setting. 
\begin{prop}\label{uniformdirac}
There exists a $T_0$ such that for $T > T_0$, there are constants $C_m$ independent of $T$ such that the semi-elliptic estimate

\begin{equation}
\|\ph\|_{rH^{m,1}_{\text{b},e,w}} \leq C_m \left( \|\slashed{D} \ph\|_{H^m_{\text{b},w}} + \|K\ph \|_{H^m_{\text{b},w}}\right)\label{uniformelliptic}
\end{equation}

\noindent holds for $\ph \in rH^{m,1}_{\text{b},e,w}(Y\mathrm{-}\mathcal Z)$, where $K$ has finite rank (independent of $T$). 
\end{prop}

\begin{proof}
The proof is a standard parametrix patching argument, of which we provide a brief sketch (see e.g., \cite[Sec. 14.2]{KM} for similar arguments). Assume, to begin, that the metric $g_T$ is a product on the cylindrical neck.

\begin{enumerate}
\item[] \underline{Step 1:} Proposition \ref{conformalchange} and Lemma \ref{weightiso} show that  on each $Y_i$ individually, the estimate
\begin{align*}
\|\ph\|_{rH^{m,1}_{\text{b},e,w}(Y_i')} &\leq  C \|\mathfrak{T}_u^{-1} \ph\|_{rH^{m,1}_{\text{b},e}} \\
&\leq C_m \left(\|\slashed D_i(\mathfrak{T}_u^{-1} \ph)\|_{H^m(Y_i)} + \|\mathfrak{T}_u^{-1} \ph\|_{H^m(Y_i)}\right)\\
&\leq C_m \left(\| \mathfrak{T}_u \slashed D_i' \ph\|_{H^m_{\text{b}}(Y_i)} + \|\mathfrak{T}_u^{-1} \ph\|_{H^m_{\text{b}}(Y_i)}\right)\\ 
&\leq C_m \left(\| \slashed D_i' \ph\|_{H^m_{\text{b},w}(Y_i')} + \|\ph\|_{H^m_{\text{b},w}(Y_i')}\right)
\end{align*}

\noindent holds uniformly in $T$ for each $m$. It follows that there are left-parametrices $P_i: H^m_{\text{b},w}(Y_i') \to rH^{m,1}_{\text{b},e,w}(Y_i')$ satisfying 
$$P_i \slashed D_i = Id + K_i, \hspace{2cm} \|P_i\|_{H^m_{\text{b},w} \to rH^{m,1}_{\text{b},e,w}} \leq C_m,$$
\noindent where $K_i$ are compact operators.
\smallskip

\item[]\underline{Step 2:} Let $\zeta_i$ for $i = 1, 2$ be a ($T-$dependent) partition of unity constructed as follows. Fix a smooth cut-off function of $s \in [-1, 1]$ such that $\xi_1(-1) = 0$ and $\xi_1 = 1$ for $s \geq 1/2$. Set $\xi_2 = 1 - \xi_1$. Then take $\zeta_i = \xi_i(t/T)$. Next, let $\chi_1 = \xi_1((t - 1)/T)$ and $\chi_2 = \xi_2((t + 1)/T)$, so that $\chi^T_i = 1$ on the supports of $d\zeta_i^T$ respectively.

Define a patched parametrix by 
\begin{equation}
P = \chi_1 P_1 \zeta_1 + \chi_2 P_2 \zeta_2.\label{patched}
\end{equation}

\noindent A quick calculation shows that 

\begin{equation}\label{patchingerrorterms}
P\slashed D = Id + \sum_i \chi_i K_i \zeta_i - \chi_i P_i d\zeta_i.  
\end{equation}

\noindent Since $K_i$ is compact and $P_i d\zeta_i$ factors through the compact inclusion $H^{m+1} \hookrightarrow H^m$ on $[-2T, 2T] \times S^2$, the elliptic estimate \eqref{uniformelliptic} follows. Moreover, because  $d\zeta_i\to 0$ and $K_i$ may be chosen to be finite rank, it is clear $K$ may also be taken to have finite rank.

\smallskip 
\item[] \underline{Step 3:} In the case of non-product metrics, the metrics are changed by an exponentially small factor in the middle of the neck, which does not disrupt the estimates.
\end{enumerate}
\end{proof}

\begin{proof}[Proof of Theorem \ref{maina} (spinor case)] Let $\alpha=[a;b]\in \R \mathbb P^1$ with neither coordinate zero. In the spinor case, set $$ \Phi^T_\alpha=a \Phi_1^T+ b \Phi_2^T$$
\noindent where $\Phi_i^T$ are as in (\refeq{PhiiTdef}). Lemma \ref{spinorapproximatesol} and Proposition \ref{uniformdirac} show that the assumptions of Theorem \ref{NMIFT} are satisfied on the manifold with cylindrical neck $(Y_1\# Y_2,g_T)$.  
\end{proof}

\subsection{Spectral Flow on the Model Neck}
\label{subsection2dneck}
The Hodge-de Rham operator (\refeq{Diracoperatortwo}) is not conformally invariant in 3-dimensions, nor is the Dirac operator on   the neighborhood of a knot conformally equivalent to an infinite cylindrical end. Thus in these two cases we consider pinching neck regions with model metrics parameterized by $\delta=T^{-1}$ given by

\noindent 
\begin{eqnarray} \label{sphereneckmetric} g_\delta&=& d\rho^2 + (\rho^2+\delta^2) g_{S^2}\\
g_\delta&=&dt^2 + d\rho^2 + (\rho^2 + \delta^2)d\theta^2 \label{torusmetric}
\end{eqnarray}

\noindent in the two cases respectively, where $\rho$ is the distance from the center of the neck, $t$ is now parallel to the knot, and $\theta$ is the angular coordinate on the normal planes. 

In this case, the parametrices arising from the closed manifolds cannot be extended over the neck to overlap, and a third parametrix is needed for the neck region. We begin in this section by analyzing $\delbar$-operators on the two-dimensional scale-invariant model neck $$N=(\R \times S^1 \ , \ dR^2  + (R^2+1)d\theta^2).$$

\noindent Patching the vector bundles properly requires a ``twist'' of the operator over the neck region which gives rise to spectral flow (recall the degree of $K_\Sigma$ does not simply add under connected sum for Riemann surfaces) \cite{cornalba1989moduli}. 

Let $K_N$ denote the canonical bundle of $N$. Since $N$ is spin, it admits a square root, and we consider
$$\delbar_N: \Omega^0(K_N^{d/2}) \to \Omega^{0,1}(K_N^{d/2})$$
for each $d \in \mathbb{Z}$. With $\mu \in \mathbb{R}$ as a weight, the Sobolev spaces $R^{1+\mu} H^1_b(N)$ and $R^\mu L^2(N)$ may be formed as before so that
\be 
\|u\|_{R^{1+\mu} H^1_b} = \left( \int_N \left( |\nabla u|^2 + \frac{|u|^2}{\langle R \rangle^2} \right) \langle R \rangle^{-2\mu} \, dV \right)^{1/2} \quad \text{and} \quad \|u\|_{R^{\mu} L^2} = \left( \int_N |u|^2 \langle R \rangle^{-2\mu} \, dV \right)^{1/2},\label{RHmuspaces}
\ee
where $\langle R \rangle = \sqrt{R^2 + 1}$ and $dV = \langle R \rangle \, dR \, d\theta$. Equivalently, we can use the desingularized boundary derivative $\nabla^\text{b} = R \nabla$ and weight both terms in the $R^{1+\mu} H^1_b$-norm equally as in Definition (\ref{sobolevspaces}).

A choice of trivialization $K_N^{1/2} \simeq \underline{\mathbb{C}}$ induces one for each $d$, in which case these operators may be written as
$$
\delbar_N u = \left( \frac{\partial}{\partial R} + \frac{1}{\langle R \rangle} \left( i \frac{\partial}{\partial \theta} + A_d \right) \right) u,
$$
where $\delbar$ denotes the standard $\delbar$ operator on complex-valued functions, and $A_d \in C^\infty(N; \mathbb{C})$.
\begin{lm}\label{infiniteneck}
$\delbar_N$ satisfies the following:
\begin{enumerate}
    \item[(i)] For each $d \in \mathbb{Z}$,
    $$A_d = \frac{d}{2} \frac{R}{\sqrt{R^2 + 1}},$$
    hence the slice operator $i\del_\theta + A_d$ has spectral flow from $\mathbb{Z} - \tfrac{d}{2}$ at $R \to -\infty$ to $\mathbb{Z} + \tfrac{d}{2}$ at $R \to +\infty$, and is Fredholm for weights $\mu \notin \mathbb{Z} + \tfrac{1}{2}$.
    
    \item[(ii)] In particular, for $d = 1$ and $-\tfrac{1}{2} < \mu < \tfrac{1}{2}$,
    $$
    \delbar_N: R^{1+\mu} H^1_\text{b}(N; K_N^{1/2}) \to R^{\mu} L^2(N; K_N^{-1/2})
    $$
    is surjective with $\dim_\mathbb{C} \ker(\delbar_N) = 1$, and there exists a constant $C$ such that
    \begin{equation}
        \|u\|_{R^{1+\mu} H^1_b} \leq C \|\delbar_N u\|_{R^{\mu} L^2} \quad \text{for} \quad u \perp \ker(\delbar_N). \label{delbarpoincare}
    \end{equation}
    
    \item[(iii)] The same holds for $d = 2$ and $-1 < \mu < 1$.
    
    \item[(iv)] The same statements hold for $\del_N$.
\end{enumerate}
\end{lm}

\begin{proof}
(i) For $d=1$, $2\delbar_N: \Omega^0(K_N^{1/2}) \to \Omega^{0,1}(K_N^{1/2})$ is the (positive) spin Dirac operator on $N$. In the trivialization given by the eigenspace of $\gamma(dr)$, the Dirac operator has the form (see \cite[Lem. 4.5.1]{KM}):
$$
2\delbar_N = \partial_r + \frac{i}{\langle R \rangle} \partial_\theta - \frac{H(r)}{2}
$$
where $H(r)$ is the mean curvature of $\{r\} \times S^1 \subseteq N$, which is given by $H(r) = \frac{-R}{R^2 + 1}$ (see \cite[Sec. 5]{Bar1992}). The general result then follows from the Leibniz rule and taking adjoints for $|d|>0$, and is trivial for $d = 0$. The spectral flow arises from the change in sign of $H$, and Fredholmness for weights not in the spectrum of the limiting operators at $R \to \pm \infty$ follows from standard theory \cite{donaldson02floerhomologybook, lockhart85elliptic}.

(ii) Item (i) shows the operator respects Fourier modes in the $S^1$-direction. Thus, writing $u = \sum u_k(R) e^{ik\theta}$, a kernel element must be a linear combination of solutions of the ODEs:
$$
\left( \frac{\partial}{\partial R} + \frac{1}{\langle R \rangle} \left( -k + \frac{R}{2\langle R \rangle} \right) \right) u_k(R) = 0.
$$

 This equation becomes more familiar under the following coordinate change (which results in a conformal equivalence with the flat infinite cylinder). Let $s$ be such that $R = \sinh(s)$. A quick computation shows that the above equation becomes:
\begin{equation}
\left( \partial_s - k + \frac{\tanh(s)}{2} \right) u_k(s) = 0.
\end{equation}

\noindent Since $\tanh(s) = \pm 1$ as $s \to \pm \infty$ respectively, one has solutions asymptotic to $e^{-(k+1/2)|s|} = R^{-(k+1/2)}$ as $s \to \infty$ and $e^{(k-1/2)|s|} = R^{(k-1/2)}$ as $s \to -\infty$. For $\mu$ in the given range, it is easy to check that precisely the $k = 0$ mode is integrable, thus the operator has a 1-dimensional kernel. Taking adjoints reverses the sign of the spectral flow, and by similar consideration, there are no solutions for the adjoint weight $\mu^\star = -\mu$.

(iii) Follows from similar considerations as (ii), and (iv) from conjugation. Note here that $\del_N = \del + A_d$ is given by conjugation and does not affect the sign of $A_d$, hence conjugation provides an isomorphism of the kernels and cokernels, whereas the $L^2$-adjoint used to determine the cokernel in (ii) is $\delbar^\star_N = -(\del - A_d)$.
\end{proof}

Corresponding to each range of weights $m-1 < \mu < m$ with $m \in \mathbb{Z}$, there is an associated APS boundary condition \cite[Sec. 17]{KM} on the truncated finite-cylinder $N_{R_0} = [-R_0, R_0] \times S^1$ . The compact boundary-value problem will be a more convenient description when dealing with the 3-dimensional case for the Dirac operator. In anticipation of this, we also write the Dirac operator in the trivialization provided by $\gamma(dt)$ as

\begin{equation}
\slashed{D}_N = \begin{pmatrix}
0 & e^{-i\theta} \left( -\partial_R + \frac{1}{\langle R \rangle} (i\partial_\theta - H) \right) \\
e^{i\theta} \left( \partial_R + \frac{1}{\langle R \rangle} (i\partial_\theta + H) \right) & 0
\end{pmatrix},
\end{equation}
where $2H(R) = -2 + \frac{R}{\langle R \rangle}$. Note that the above trivialization differs from that induced by $\gamma(dR)$ by a twist $e^{i\theta}$ in the top component and the conjugate in the bottom. For the remainder of the section, we restrict to the case $d=1$ and $-\frac{1}{2} < \mu < 0$.

Two different APS boundary conditions are depicted below for a spinor $\ph=(\alpha,\beta)$, where the allowed Fourier modes on the boundary are indicated and empty modes are constrained to be zero.

\begin{eqnarray}
	\text{Fourier mode}   & &  \ldots \  {\underline{k=-2} } \hspace{.5cm}   \    {\underline{k=-1} }  \hspace{.75cm}{\underline{k=0} } \  \  \  \hspace{.30cm} {\underline{k=1 }}   \ \ \  \hspace{.25cm} {\underline{k=2 }}  \ \ldots \hspace{1.7cm}\nonumber \\
	\alpha|_{R=R_0}&=& \ldots \alpha_{-2}e^{-2i\theta}  + \ { {\alpha_{-1}e^{-i\theta}}}  \ \hspace{.05cm}  +  \  \boxed{\alpha_{0}}   \\
	\alpha|_{R=-R_0} &=&     \hspace{1.75cm} \  \     \   \  \boxed{\alpha'_{-1}e^{-i\theta}}\   +     \  \ \alpha'_0  \  \     \  +  \    \    \alpha'_1 e^{i\theta} \   +    \ \alpha'_2 e^{2i\theta} \  +\ldots \\
	\beta|_{R=R_0} &=&  \hspace{3.35cm}  \  \  \   \  \  \   \ \ \  \  \hspace{-.15cm} \boxed{\beta_0}    \hspace{.15cm}\ +  \     \   {{   \beta_1 e^{i\theta}}}  \   +  \    \beta_2 e^{2i\theta} +\ldots. \\ 
	\beta|_{R=-R_0}&=& \ldots \beta'_{-2}e^{-2i\theta}   +   \beta'_{-1}e^{-i\theta} \  \hspace{.05cm}  \ + \  \hspace{.05cm}  \ \beta'_0 \ \ \hspace{.05cm}  \  +  \ \ \boxed{\beta'_{1}e^{i\theta}}
\end{eqnarray}

It is straightforward to check via integration by parts that:
\begin{enumerate}
    \item[(i)] The boundary condition allowing only the unboxed modes is self-adjoint, hence has $\text{Ind}_\mathbb{C} = 0$ and $\slashed{D}_N$ with this boundary condition is invertible.
    \item[(ii)] The boundary condition for which the boxed modes are constrained by 
    \begin{eqnarray}
    	\label{eq_boundarycondition}
        \alpha_{-1}' = \alpha_{-1} &\hspace{2cm}& \beta_0 = \beta_0' \\
        \label{eq_boundaryconditionb}
        \beta_{1}' = \beta_{1} &\hspace{2cm}& \alpha_0 = \alpha'_0
    \end{eqnarray}
    is self-adjoint and index 0, but has both a 2-dimensional kernel and cokernel.
\end{enumerate}

The kernel is spanned by $(\kappa_\circ, 0)$ and $(0, \overline{\kappa}_\circ)$, where $e^{i\theta}\kappa_\circ$ is the single solution from item (ii) of Lemma \ref{infiniteneck}. These are tacitly denoted simply by $\kappa_\circ, \overline{\kappa}_\circ$. The key point is that despite (i) being invertible, the elliptic estimate fails to be uniform as $R_0 \to \infty$ because $\kappa_\circ$ decays toward the boundary, so cutting it off will violate uniform estimates. Here we have imposed boundary conditions that allow $\kappa_\circ, \overline{\kappa}_\circ$ as true kernel elements, which is easier to analyze\footnote{Note the solution in the $\beta_0$ mode does not have equal boundary values at the two ends, so is not in the kernel.}.

On the other hand, by self-adjointness the cokernel (identified with the kernel of the \textit{weighted} adjoint) consists of the span of
\be \kappa^\dagger = \frac{R^{2\mu} \kappa_\circ}{R_0^{1/2 + \mu}} \quad \text{and} \quad \overline{\kappa}^\dagger = \frac{R^{2\mu} \overline{\kappa}_\circ}{R_0^{1/2 + \mu}},\label{kappakappadagger} \ee
which are normalized to have $R^\mu L^2$-norm $O(1)$ (independent of $R_0$). Notice that $\kappa_\circ$ fails to be integrable in $L^2$ as $R_0 \to \infty$, thus the norm of these cokernel elements is concentrated near the boundary.

In the following, $R^{1} H^1_b(N_{R_0})$ denotes the subspace satisfying the boundary conditions (ii). Set $P_1 = [-R_1, R_1] \times S^1$ for a fixed $R_1 < R_0$.

\begin{lm}\label{scaledneckBVP}
For each $-\frac{1}{2} < \mu \leq 0$, there is a constant $C_\mu$ independent of $R_0$ such that, subject to the boundary conditions (ii), the Dirac operator $\slashed{D}_N$ has index 0 and satisfies 
\begin{eqnarray}\label{BVPPoincare}
\|u\|_{R^{1+\mu} H^1_b(P_1)} &\leq& C_\mu \|\slashed{D}_N u\|_{R^\mu L^2} \quad \forall u \ \text{s.t.} \ \  \langle u, \kappa_\circ \rangle_{R^{1+\mu} H^1_b(P_1)} = \langle u, \overline{\kappa}_\circ \rangle_{R^{1+\mu} H^1_b(P_1)} = 0, \\ \label{BVPelliptic}
\|u\|_{R^{1+\mu} H^1_b} &\leq& C_\mu \left( \|\slashed{D}_N u\|_{R^\mu L^2} + \left\|\frac{u}{\br R \kt }\right\|_{R^\mu L^2(P_1)} \right).
\end{eqnarray}
\end{lm}

\begin{proof}
The index is immediate from self-adjointness. If (\ref{BVPPoincare}) did not hold, cutting off elements violating it for increasingly large $R_0$ would eventually contradict (\ref{delbarpoincare}). (\ref{BVPelliptic}) follows because the portion of $\kappa_\circ$'s norm supported on $P_1$ is bounded below as $R_0 \to \infty$. 
\end{proof}

We will now introduce a 2-parameter family of perturbations that will cancel the obstruction provided (\refeq{kappakappadagger}). Suppose
\begin{equation} 
\Phi_\circ = \chi_1(R) \begin{pmatrix} c \\ d \end{pmatrix} \label{necksolution}
\end{equation}
is a constant spinor where $|c|^2 + |d|^2 > 0$, and $\chi_1$ is a cut-off function equal to $1$ for $R \leq -R_0/4$ and vanishing for $R \geq -R_1$. In Section \ref{subsectiontoruspinching}, $\Phi_\circ$ is taken to be the cut-off of the leading order term of the left-side spinor $\Phi_1$. For $\xi = (\xi_1, \xi_2) \in \mathbb{C}^2$, consider the perturbation
\begin{equation}
B(\xi) = \chi_0(R) \frac{1}{R_0^{1/2 - \mu}} \sqrt{\frac{1}{R}} \left[ -\xi_1 \begin{pmatrix} 0 & -e^{-i\theta} \\ e^{i\theta} & 0 \end{pmatrix} - i \xi_2 \begin{pmatrix} 0 & ie^{-i\theta} \\ ie^{i\theta} & 0 \end{pmatrix} \right] J \label{perturbation},
\end{equation}
where $J(\alpha, \beta) = (-\overline{\beta}, \overline{\alpha})$, which is of the class of perturbations allowed in the proof of Theorem \ref{NMIFT}. Here, $\chi_0$ is a log cut-off supported in $[-(1+\epsilon)R_0, -1/2 R_0]$ for some small $\epsilon$ to be specified later, and equal to 1 on $[-R_0, -3/4 R_0]$.

Letting $R^{1+\mu}H^1_\circ$ denote the space satisfying the boundary conditions (ii) and the orthogonality constraint of (\ref{BVPPoincare}), consider the extended Dirac operator
\begin{eqnarray} \label{2DextendedDirac}(\slashed D_N, B): R^{1+\mu}H^1_\circ \oplus \C  & \lre  & R^{\mu}L^2 \\ 
(u,\xi)&\mapsto & \slashed D_N u  + B(\xi)\Phi_\circ \nonumber.
\end{eqnarray}
\begin{lm} \label{cancellationlem} \label{injectivity2D}
Provided $\Phi_\circ$ satisfies $|c|^2 + |d|^2 > 0$, \eqref{2DextendedDirac} is an isomorphism for each $-1/2 < \mu \leq 0$ with inverse uniformly bounded in $R_0$ (but depending on $\mu$).
\end{lm}

\begin{proof}
Splitting the range into $\text{Range}(\slashed{D}_N) \oplus \mathbb{C} \{\kappa^\dagger, \overline{\kappa}^\dagger\}$, the operator takes the form
$$
(\slashed{D}_N, B) = \begin{pmatrix} \pi B & 0 \\ \pi^\perp B & \slashed{D}_N \end{pmatrix},
$$
where $\pi, \pi^\perp$ are the orthogonal projections. It therefore suffices to show that $\pi B$ is bounded below, and $\pi^\perp B$ is bounded above, both uniformly in $R_0$.

Assume first that $|c| > 0$ and $|d| > 0$ are both non-vanishing. The normalization factor in (\ref{perturbation}) is chosen precisely so that
\begin{equation}
\|B(\xi) \Phi_\circ\|^2_{R^\mu L^2} = \frac{1}{R_0^{1 - 2\mu}} \int_{-R_0}^{-R_1} \frac{\chi_0^2 \chi_1^2}{R} |\xi|^2 R^{-2\mu} R \, dR \, d\theta \leq C |\xi|^2,
\end{equation}

and

\begin{equation}
\left\langle B(\xi) \Phi_\circ, \begin{pmatrix} a \kappa^\dagger \\ b \overline{\kappa}^\dagger \end{pmatrix} \right\rangle = \frac{1}{R_0} \int_{-R_0}^{-R_1} \frac{\chi_0 \chi_1}{R} \left\langle \begin{pmatrix} (\xi_1 + i \xi_2) e^{-i\theta} \overline{c} \\ (\xi_1 - i \xi_2) e^{i\theta} \overline{d} \end{pmatrix}, \begin{pmatrix} a \kappa_\circ \\ b \overline{\kappa}_\circ \end{pmatrix} \right\rangle \, dV =
\left\langle
\begin{pmatrix}
	C_+(\xi_1+i\xi_2)c\\
	C_-(\xi_1-i\xi_2)d
\end{pmatrix},
\begin{pmatrix}
	a\\ b
\end{pmatrix}
\right\rangle_{\mathbb C^2},
\end{equation}
where the last inner product is in $\mathbb{C}^2$, since $\kappa_\circ = O(R^{-1/2} e^{-i\theta})$ and here \(C_+\) and \(C_-\) are non-zero constants depending only on the chosen normalizations of the two cokernel generators; with a symmetric normalization one may arrange \(C_-= \overline{C_+}\). Their precise values are irrelevant for the argument, since they are uniformly bounded above and below.. Note that the $\mu$-dependent normalization and weights cancel. When $c$ and $d$ are both nonzero, the resulting equation on $\mathbb{C}^2$ is uniformly invertible. In the case that only $|c|^2 + |d|^2 \neq 0$, (\ref{perturbation}) may be easily adjusted by also including terms of the form $e^{-i\theta} \gamma(dt) J$.
\end{proof}

\begin{rem}
There is possibly an analogue of Lemma \ref{cancellationlem} in the case of 1-forms using metric perturbations as in (\ref{BG}). However, later steps in the proof of Theorem \ref{mainb} fail in this case because certain analytic steps are not valid for the necessary range of weights for 1-forms.
\end{rem}

\subsection{Neck Pinching I: Spherical Case}
\label{subsectionneckpinchingsphere} 
This subsection proves Theorem \ref{maina} in the case of 1-forms by studying the connected sum with the pinching neck (\refeq{sphereneckmetric}). 

More precisely, the metric $g_\delta$ is defined as follows: for points $y_i \in Y_i - \mathcal{Z}_i$ let $B_{\rho_0}(y_i)$ be a geodesic normal coordinate chart of fixed radius $\rho_0 > 0$. $g_i$ may be written (\refeq{eq_metric_gi}) on $B_{\rho_0}(y_i)$ as before. The connected sum is formed by replacing the punctured balls with a neck $[-\rho_0, \rho_0] \times S^2$ equipped a new centered coordinate (also denoted $\rho$) and the metric 

\begin{equation}
g_\delta = d\rho^2 + (\rho^2 + \chi \delta^2) g_{S^2} +(1-\chi)(h_1 + h_2), \label{gdeltadef}
\end{equation}

\noindent where $\chi$ is a smooth bump function equal to 1 for $|\rho| < \sqrt{\delta}$ and vanishing for $|\rho| > 2\sqrt{\delta}$. Here $g_{S^2}$ is the round metric of unit radius. 

\subsubsection{Approximate Solutions via Locally Exact 1-forms}
On $Y=Y_1\#Y_2$ the natural function spaces have desingularized $\text{b}$-derivatives on the neck. Thus let $\br \rho_\delta\kt=\sqrt{\rho^2+\chi \delta^2}$, and consider the derivatives given by 
$$\nabla^\text{b}=\br \rho_\delta\kt \nabla_\rho \otimes d\rho  + \nabla^{S^2}$$
\noindent near the neck region, and by (\refeq{bderiv}) near $\mathcal Z$. $\nabla^e$ is defined identically, but with (\refeq{edgederiv}) near $\mathcal Z$. Set 

\begin{equation} \label{doublebspaces}
r^\nu \rho^\mu H^{m,n}_{\text{b},e} = \left\{ \omega \in L^2(Y; \Omega) \ \Big | \ \int_{Y - (\mathcal{Z})} \sum_{|\alpha| \leq n, |\alpha|+|\beta| \leq m} |(\nabla^e)^\alpha (\nabla^\text{b})^\beta \omega|^2 \ r^{-2\nu} \br \rho_\delta\kt^{-2\mu} dV  < \ \infty \ \right\},
\end{equation}

\noindent and when $n=0$, the spaces are denoted simply by $r^\nu \rho^\mu H^{m}_{\text{b}}.$

The application of Theorem \ref{NMIFT} requires that the approximate solutions $\Phi_i^\delta$ are closed forms. Recall that if $(\mathcal{Z}_i, \ell_i, \Phi_i)$ is a $\Z_2$-harmonic 1-form on $Y_i$, then integration by parts shows that $\Phi_i = (0, \nu_i) \in \Omega^0 \oplus \Omega^1$. Closed approximate solutions on the connected sum $(Y, g_\delta)$ may now be constructed as follows. On each $B_{\rho_0}(y_i)$, let $f_i$ be a smooth primitive such that $$df_i = \nu_i, \hspace{2cm} f_i(y_i) = 0.$$

\noindent For $\alpha = [a; b] \in \mathbb{RP}^1$ with $a, b \neq 0$, define $\Phi_\alpha'$ by 
\begin{equation}
\Phi_\alpha^\delta  = \begin{cases} 
(0, d(a \chi_1 f_1 + b \chi_2 f_2) ) & \text{when } |\rho| < \rho_0, \\  
(0, \nu_i) & \text{on } Y_i - B_{\rho_0}(y_i),
\end{cases}
\label{approximatesol1form}
\end{equation}

\noindent where $\chi_1$ is a cutoff function equal to 1 for $\rho \leq -2c_0\delta$ and vanishing for $\rho > c_0\delta$ for $c_0$ large, and $\chi_2 = \chi_1(-\rho)$.

\begin{lm} \label{1formsapproximatesol}
$\Phi_\alpha^\delta$ is closed, and for $\mu,\nu \in \R, m \in \mathbb{N}$, there exist constants $C_m > 0$ such that
$$\|D_\text{dR} \Phi_\alpha^{\delta} \|_{r^\nu \rho^\mu H^m_{\text{b}}} \leq C_m \delta^{1-\mu}. $$
\end{lm}

\begin{proof}
That $\Phi_\alpha^{\delta}$ is closed is immediate from the definition. Since $\nu_i$ is harmonic, it is clear that $D_\text{dR} \Phi_\alpha^{\delta}= -d^\star \Phi_\alpha^{\delta}$ and is supported on the neck region. A quick calculation shows 
\begin{align}
d^\star \Phi_\alpha^\delta &= d^\star (a d\chi_1 \cdot f_1 + \chi_1 \nu_1 + d\chi_2 \cdot f_2 + \chi_2 \nu_2) \\
&= a \Delta \chi_1 \cdot f_1 + b \Delta \chi_2 \cdot f_2 + 2 d\chi_1 \cdot \nu_1 + 2 d\chi_2 \cdot \nu_2 + e_\delta, \label{todifferentiate}
\end{align}

\noindent where $\cdot = \mathbf{cl}$ denotes Clifford multiplication given by the symbol of $D_\text{dR}$. Here, $e_\delta$ is a smooth uniformly bounded error term arising from the difference between the metrics \eqref{eq_metric_gi} and \eqref{gdeltadef} on $\text{supp}(\chi_i)$. 

For $m = 0$, the fact that $|\nabla^m\chi_i|\le C_m\delta^{-m}$, and $f_i = O(\delta)$ on $\text{supp}(d\chi_i)$ since it vanishes at $y_i$, while $\nu_i = O(1)$ shows that 
$$\|d^\star \Phi_\alpha^\delta\|_{L^2}^2 \leq C \int_{c\delta}^{2c\delta} \delta^{-2} + 1 \, dV \leq C \delta$$

\noindent once $\delta$ is sufficiently small. For $m \geq 0$, note that the weighted derivatives $(\rho \nabla_\rho)^m \chi_i \leq C_m$ are bounded independent of $\delta$ and the derivatives in the $S^2$-directions only act on $f_i, \nu_i$. Repeatedly differentiating \eqref{todifferentiate} therefore yields the desired bound, where $C_m$ depends on the weighted $H^m_{\text{b}}$-norm of $\nu_i$. The case for $\mu\neq 0$ is similar.  
\end{proof}

\subsubsection{Spectral Flow on Spherical Necks} To obtain a parametrix on the pinching neck, we generalize the analysis of Subsection \ref{subsection2dneck} to the case of a spherical cross section. Thus consider the scale invariant model neck 
$$N=(\R \times S^2 \ , \ dR^2 + (R^2+1)g_{S^2}).$$

For $\omega_x,\omega_y$ a local orthonormal coframe on the unit $S^2$, then $dR, \sqrt{R^2+1}\omega_x, \sqrt{R^2+1}\omega_y$ is an orthonormal coframe on the neck. A brief computation shows that for $H=\tfrac{R}{\sqrt{R^2+1}}$, 

\be D_{\mathrm{dR}}\begin{pmatrix}\alpha \\ \beta \end{pmatrix}=\left[ \p{}{R} + \frac{1}{\sqrt{R^2+1}} \begin{pmatrix} 0 & d + d^\star \\ d+ d^\star  & H
\end{pmatrix} \right]\begin{pmatrix}\alpha \\ \beta \end{pmatrix},\ee

\noindent where we associate a form $(a_0, a_R dR +\beta)$ with $\alpha=(a_0, a_R dV_{S^2}) \in \Gamma(N; \Lambda^0_{S^2} \oplus \Lambda^2_{S^2})$ and $\beta\in \Gamma(N;\Lambda^1_{S^2})$, and $d+d^\star$ denotes the two-dimensional Hodge-de Rham operator. The analogue of Lemma \ref{infiniteneck} is

\begin{lm} \label{S2spectralflow} There is a $\mu_0\in (0,1/2)$ such that $\mu \in (-\mu_0, \mu_0)$, $$D_{\mathrm{dR}}: R^{1+\mu}H^{m+1}_b(N)\to R^{\mu}H^{m}_b(N)$$
\noindent is an isomorphism, and there are constants $C_m$ such that estimates $$\|\nu\|_{R^{1+\mu}H^{m+1}_b}\leq C_m \| D_{\mathrm{dR}} \nu\|_{R^\mu H^m_b} $$
\noindent hold. 
\end{lm}
\begin{proof} Denote $d+d^\star: \Omega^0(S^2)\oplus \Omega^2(S^2)\to \Omega^1(S^2)$ by $A$. This operator is diagonalized as follows: if $\Delta_{\Omega^1}\beta = \lambda^2 \beta$ is an eigenvector of $\Delta_{\Omega^1}=AA^\star$ then 
$$\beta^\pm = \begin{pmatrix} \pm\tfrac{1}{{\lambda}}A^\star \beta  \\ \beta \end{pmatrix}$$
are eigenvectors of $d+d^\star=\begin{pmatrix} 0 & A^\star \\ A & 0 \end{pmatrix}$ with eigenvalues $\pm \lambda$ respectively. The same applies, flipping the components, to an eigenvector $\alpha$ of $\Delta_{\Omega^0\oplus \Omega^2}=A^\star A$. Using known results for the spectra of these Laplacians  \cite[Thm. 5.1]{kuwabara82spectra} , this shows 

\begin{align}
\text{Spec}(d+d^\star)&=\Big \{ \lambda \ | \ \lambda^2 \in \text{Spec}(\Delta_{\Omega^0\oplus \Omega^2})) \cup \text{Spec}(\Delta_{\Omega^1})\Big \}\\
&= \Big \{ \lambda \ | \ \lambda^2 \in \big\{ 0, 2, 6,\ldots \big\} \cup \big\{1,5, 11,\ldots\big\} \Big \} \label{twospectra}.
\end{align}

In the basis $\tfrac{1}{2}(\psi^+ - \psi^-), \tfrac{1}{2}(\psi^++\psi^-)$ of these two eigenspaces, where $\psi^\pm=\beta^\pm$ or $\alpha^\pm$, the operator takes the form  \be d+d^\star= \p{}{R} +\frac{1}{\sqrt{R^2+1}} \begin{pmatrix} 0 & \lambda  \\  \lambda & H \end{pmatrix}, \label{S2ODE}\ee

\noindent where $H=\frac{R}{\sqrt{R^2+1}}$ as before. In particular, $|H|<1$. We now argue, via spectral flow, that the operator (\refeq{S2ODE}) is invertible on each of these 2-dimensional spaces. Under the coordinate change $R=\sinh(\tau)$, $\sqrt{R^2+1}=\cosh(\tau)$ and quick calculation shows that (\refeq{S2ODE}) becomes

 \be \p{}{\tau} +\begin{pmatrix} 0 & \lambda  \\  \lambda & H(\tau) \end{pmatrix}: C^\mu H^{\mu}(\R_\tau ; \R^2)\to  C^\mu H^m(\R_\tau; \R^2), \label{S2ODE2}\ee

\noindent where the weight is now given by $C=\cosh(\tau)$, and the volume form includes an additional factor of $C^2 dV_{S^2}=\br R\kt^2 dV_{S^2}$. It therefore suffices to check that there is a range of $\mu$ such that no eigenvalue crossing leads to a solution which is integrable with the proper weight for both $\mu$ and the adjoint weight $-\mu$. 

The matrix in (\refeq{S2ODE2}) has eigenvalues $\tfrac{H}{2}\pm \tfrac{\sqrt{4\lambda^2 +H^2}}{2}$. The values of $\lambda^2$ in (\refeq{twospectra}) therefore show that the corresponding 2-dimensional subspaces have $\tau$-parameterized eigenvalues given by 

\be  \frac{H(\tau)}{2} \pm \frac{|H(\tau)|}{2} \ , \  \frac{H(\tau) \pm \sqrt{5}}{2} \ , \ \frac{H(\tau) \pm  \sqrt{17}}{2} \ , \ \ldots . \label{endweights3}\ee

\noindent Since $|H|<1$, only the first possibility, corresponding to $\lambda=0$ must be investigated, since the values are either strictly positive or strictly negative once $\lambda^2\geq 1$ and thus cannot have a spectral crossing.  In the first case, the eigenvalue progresses from $-1$ to $0$ and from $0$ to $1$ for the two signs respectively. Because the volume form has an additional factor of $R^2=\cosh(\tau)^2$. This precludes the solutions in this zero-eigenspace of $\Delta_{\Omega^0\oplus \Omega^2}$, which are both asymptotic to constants or $e^{-|\tau|}$ on the two ends, from being integrable, provided $|\mu|<1/2$. Invertibility for both $\mu$ and the adjoint weight $-\mu$ follow for $\mu$ in this range. The estimates are a routine consequence, since the operator is Fredholm provided $2+\mu$ avoids the discrete set of problematic weights (\refeq{endweights3}). 
\end{proof}

\subsubsection{Parametrix Patching}\label{section3.3.3} Let $N_\delta=(\R\times S^2, g_\delta)$ be the shrinking neck with the model metric (\refeq{sphereneckmetric}). By scaling, Lemma \ref{S2spectralflow} immediately implies
\begin{cor} \label{rescaledS2}For $\mu\in (-\mu_0, \mu_0)$ with $\mu_0$ as in Lemma \ref{S2spectralflow} , $D_{\mathrm{dR}}: \rho^{1+\mu} H^{m+1}_b(N_\delta)\to \rho^\mu H^m_b(N_\delta)$ is an isomorphism and there are constants $C_m$ such that 
$$\|\omega\|_{\rho^{1+\mu}H^{m+1}_b}\leq C_m \|D_{\mathrm{dR}} \omega\|_{\rho^{\mu}H^m_b}$$
\noindent holds uniformly in $\delta$. 
\end{cor}

We may now prove uniform global estimates. These show that $D_{\text{dR}}$ are strictly uniform in the sense of Definition \ref{strictlyuniform} in the case of a spherical neck. 

\begin{prop}\label{uniformderham}
For $\mu\in (-\mu_0, \mu_0)$, there exists a $\delta_0$ such that for $\delta < \delta_0$, there are $\delta$-independent constants $C_m$ such that the semi-elliptic estimate 
\begin{equation}
\|\nu\|_{r \rho^{1+\mu} H^{m,1}_{\text{b},e}} \leq C_m \left( \|D_\text{dR} \nu\|_{\rho^\mu H^m_{\text{b},w}} + \| K\nu \|_{\rho^{\mu }H^m_{\text{b}}}\right) \label{uniformelliptic1form}
\end{equation}
\noindent holds for $\nu \in r \rho^{1+\mu} H^{m,1}_{\text{b},e}(Y\mathrm{-}\mathcal{Z})$, where $K$ has finite rank (independent of $\delta$). 
\end{prop}

\begin{proof} The proof has three steps. 
\begin{enumerate}
\item[]  \underline{Step 1:} It is easy to check that Corollary \ref{rescaledS2} holds equally well replacing the model metric with (\refeq{gdeltadef}). Let $P_N$ denote the inverse of $D_{\mathrm{dR}}$ in this region. Let $P_1, P_2$ denote inverses for $D_{\mathrm{dR}}$ on $\mathcal H^1_\perp$ on $Y_1, Y_2$ respectively (with $\mathcal H^1_\perp$ as in part (iii) of the proof of Theorem \ref{NMIFT}). 
\medskip
\item[]   \underline{Step 2:} Let $\rho_0$ be small and independent of $\delta$. Choose cutoff functions unity $\chi_1, \chi_2$ equal to 1 on the bulk of $Y_1, Y_2$ and with derivatives supported where $\rho=O(\pm \rho_0)$. Let $\chi_N=1-\chi_1-\chi_2$. Similarly, let $\zeta_N$ be a cut-off function so that ${\text{supp}(\chi_N)}\Subset \{\zeta_N=1\}$. Finally, let $\zeta_1, \zeta_2$ be likewise cut-off functions equal to $1$ where $\rho\geq O(\rho_0)$, so that $\text{supp}(\chi_i)\Subset \{\zeta_i=1\}$. 

Setting $$P:= \zeta_1 P_1 \chi_1  + \zeta_2 P_2 \chi_2 + \zeta_N P_N \chi_N.$$

\noindent A quick calculation similar to Step 2 in the proof of Proposition \ref{uniformdirac} shows that $P$ is a uniformly bounded parametrix, and $K$ consists of the projection to $\mathcal H^1$ and the projection to the support of $d\chi_i$. 
\medskip
\item[]   \underline{Step 3:} The proof of Theorem \ref{NMIFT} in this case actually requires the slightly stronger statement the operator is injective on the complement of $\mathcal H^1$. This may be achieved for the $\mu=0$ weight by replacing the cut-off functions in the above with logarithmic cut-off functions (see Section \ref{subsectiontoruspinching}). The same argument applies for all $m>0$. 
\end{enumerate}
\end{proof}

\begin{proof}[Proof of Theorem \ref{maina} (1-form case)] Define $\Phi^\delta_\alpha$ as in (\refeq{approximatesol1form}). Lemma \ref{1formsapproximatesol} and Proposition \ref{uniformderham} show that the assumptions of Theorem \ref{NMIFT} are again satisfied, this time on the manifold with pinched neck $(Y_1\# Y_2,g_\delta)$. The case of $Y_i=\Sigma_i\times S^1$ follows similarly using Lemma \ref{infiniteneck} in place of Lemma \ref{S2spectralflow}. 
\end{proof}

\subsection{Neck Pinching II: Toroidal Case} 
\label{subsectiontoruspinching}

This subsection proves Theorem \ref{mainb} by pinching necks in $1$-parameter families. This situation is more involved than that of the previous subsection for two reasons: first, the elliptic boundary operator at the neck is replaced by an elliptic edge operator, and second the weaker scaling from the volume form means the error only approaches zero for negative weights for which the operator has an obstruction (as in Lemma \ref{infiniteneck}).   

To describe the set-up more precisely, let $(Y_i, g_i)$ and $K_i \subset Y_i - \mathcal{Z}_i$ be as described in the statement of Theorem \ref{mainb}. Choose tubular neighborhoods $N(K_i) \simeq D_R \times S^1$ with coordinates $(t, x, y)$ and corresponding cylindrical coordinates $(t, \rho, \theta)$. The metrics may be written $g_i|_{N(K_i)} = dt^2 + d\rho^2 + \rho^2 d\theta^2 + h_i$, where $h_i = O(\rho)$. By scaling the metrics by a constant, it may be assumed that the two knots $K_i$ have equal length $2\pi$. For $\delta = 1/T << 1$, the torus sum $Y_K$ may be endowed with the metric given by $g_i$ on the bulk of $Y_i$ and by
\be 
g_\delta = dt^2 + dr^2 + (\rho^2 + \delta^2)d\theta^2 + \chi_1 h_1 + \chi_2 h_2 \label{gluedmetric}
\ee

\noindent in the neck region $[-R, R] \times T^2$. Here, $\chi_i$ are ($\delta$-dependent) cut-off functions as in \eqref{gdeltadef}.

\subsubsection{Approximate Solutions and Error Terms} 

On $(Y_K, g_\delta)$, let $r$ denote the distance from $\mathcal Z$ and let $\rho$ denote the distance from the center of the neck region. Using the weight $\br \rho_\delta \kt=\sqrt{\rho^2+\delta^2}$ (we will often drop the $\delta$), the analogue of the spaces \eqref{doublebspaces} on $Y_K$ become 

\begin{equation}
r^\nu \rho^\mu H^{m,n}_{\text{b},e} (Y_K; S) = \left\{ \psi \in L^2(Y_K;S) \ \Big | \ \int_{Y_K} \sum_{|\alpha| \leq n, |\alpha|+|\beta| \leq m} |(\nabla^e)^\alpha (\nabla^\text{b})^\beta \psi|^2 \ r^{-2\nu} \br \rho_\delta\kt^{-2\mu} \ dV \ < \ \infty \ \right\},
\end{equation}

\noindent where $\nabla^\text{b},\nabla^e$ are the boundary and edge-weighted derivatives along both $\mathcal Z$ and the neck region, i.e. near $K$ they are given by (\refeq{bderiv}--\refeq{edgederiv}) with $\br \rho_\delta \kt $ in place of $r$.  

We now construct model solutions. Let $\chi_1(\rho)$ be a logarithmic cut-off function \cite[Sec. 10.4]{MSBig} equal to $1$ for $\rho\leq \delta^{5/8}$ and vanishing for $r\leq \delta^{3/4}$ and such that 

\be |\nabla^m\chi_1|\leq \frac{C}{\log(1/\delta)}\frac{1}{\rho^m}.\label{logcutoff}\ee

\noindent Set $\chi_2(\rho)=\chi_1(-\rho)$. For $\alpha=[a;b]\in {\R \mathbb P}^1$, define model solutions by 
\be
\Phi_\alpha^\delta=a\chi_1 \Phi_1 + b\chi_2 \Phi_2  \label{approximatesoltorus},
\ee

\noindent where the spinors are written in the local trivializations induced by $\gamma(dt)$ which are patched on the neck region using condition (i) in Theorem \ref{mainb}. By a simple transversality argument, we may assume after an isotopy of $K_1$ that $|\Phi_1|>0$ on $K_1$.

Similar to Lemmas \ref{spinorapproximatesol} and \ref{1formsapproximatesol}, we have:  

\begin{lm}\label{toruspincherror}
For each $m\in \N$, there exist constants $C_m$ independent of $\delta$ such that 
\bea \|\slashed D\Phi_\alpha^\delta \|_{\rho^\mu H^m_\text{b}}&\leq& \frac{C_m}{\log(1/\delta)} \delta^{-\mu/2}.\eea
In particular, for $\mu\leq 0$, the error approaches zero. 
\end{lm}

\begin{proof}It suffices to estimate the two summands coming from \(D(\chi_1\Phi_1)\) and \(D(\chi_2\Phi_2)\). For the first two terms and $m=0$, similar computations to Lemmas \ref{spinorapproximatesol} and \ref{1formsapproximatesol} using (\refeq{logcutoff})$$\|\slashed D\Phi_\alpha^{\delta}\|^2_{L^2}\leq \frac{C}{\log(1/\delta)} \int_{T^2} \int_{\delta^{3/4}}^{\sqrt{\delta}} \rho^{-2} \rho^{-2\mu} \rho d\rho d\theta dt \leq \frac{C}{\log(1/\delta)^2} \delta^{-\mu}.$$ 

\noindent For $m\geq 1$, a similar result holds after factoring out $C^m$ bounds on $\Phi_i$ and using that $\nabla^\text{b}=\rho \nabla_\rho$ derivatives precisely cancel out the factors of $\rho$ in (\ref{logcutoff}). 
 It is easy to verify that the higher-order terms from the metric contribute a negligible error. 
\end{proof}

\subsubsection{Dirac Operators on Pinching Torus Necks} Next, we show that a uniformly bounded parametrix  may be constructed on the neck region using a high-dimensional family of perturbations. Let $N_K\simeq K \times [-\rho_0, \rho_0]\times S^1$ denote the joining of the tubular neighborhood of $K_i$, endowed with coordinates $(t,\rho,\theta)$. Assume, to begin, that the metric is the model metric (\refeq{torusmetric}). The three-dimensional Dirac operator may be written 
$$\slashed D=\gamma(dt)\del_t + \slashed D_N, $$

\noindent where $\slashed D_N$ is the Dirac operator on $([-\rho_0, \rho_0]\times S^1, g_\delta)$. We now fix $R_0=\rho_0\delta^{-1}$. Scaling by setting $R=\delta^{-1}\rho$, so that $P_1= [-\delta R_1, \delta R_1] \times S^1$, Lemma \ref{scaledneckBVP} yields:

\begin{cor}\label{scaledneckBVPII} 
For each weight $-\tfrac12<\mu\leq 0$, there is a constant $C_\mu$ independent of $\delta$ such that subject to the boundary conditions (ii) defined by \eqref{eq_boundarycondition} and \eqref{eq_boundaryconditionb}, the Dirac operator $\slashed D_N$ has Index 0 and satisfies 
\begin{eqnarray}\label{BVPPoincareII}
\|u\|_{\rho^{1+\mu} H^1_\text{b}} &\leq& C_\mu \  \|\slashed D_N u\|_{  \rho^\mu L^2} \hspace{1.0cm} \forall u  \ \text{s.t.} \ \ \br u, \kappa_\circ \kt_{\rho^{1+\mu} H^1_\text{b}(P_1)} =\br u, \overline \kappa_\circ \kt_{\rho^{1+\mu} H^1_\text{b}(P_1)}=0  \\ \label{BVPellipticII}
\|u\|_{\rho^{1+\mu}H^1_\text{b}} &\leq& C_\mu \left( \|\slashed D_N u\|_{ \rho^\mu L^2} \ + \  {\|\tfrac{u}{\br \rho\kt}\|_{\rho^\mu  L^2(P_1)}}\right).
\end{eqnarray}
\end{cor} 
\begin{proof}
The scaled norms are related by $\| -\|_{R^{1+\mu}H^1_b}= \delta^{\mu} \|-\|_{r^{1+\mu}H^1_b}$ and $\| -\|_{R^{\mu}L^2}= \delta^{\mu-1} \|-\|_{r^{\mu}L^2}$. Thus since $\nabla_\rho=\delta^{-1}\nabla_R$, the left and right sides both scale like $\delta^{\mu}$. The result is then immediate from Lemma \ref{scaledneckBVP}.
\end{proof}

Next, we define boundary conditions on $\del N_K\simeq T^2$. The boundary-trace of a spinor $\ph$ may be written 
$$\ph|_{\del N_K}=\sum_{k} \begin{pmatrix} \alpha_{k\ell} \\ \beta_{k\ell}\end{pmatrix}   e^{ik\theta}e^{i\ell t}.$$
\noindent Set $L=\lfloor \delta^{-1}\rfloor$. The boundary condition is: 
\begin{enumerate}
\item[(ii')] For $|\ell|\leq L$, the functions $\alpha_{k\ell}, \beta_{k\ell}$ satisfy boundary condition (ii) from Section \ref{subsection2dneck}, and for $|\ell|\geq L$ satisfy boundary condition (i). 
\end{enumerate}
\noindent These are semi-local variations of APS boundary conditions (see \cite[Sec. 7]{PartI} for more detailed discussion). We also define a parameterized version of the orthogonality constraint from \refeq{BVPPoincare} as follows. For each $\eta(t)=(\eta_1(t), \eta_2(t))\in L^{1,2}(K; \C^2)$, the configurations 
$\eta_1(t)\kappa_\circ \in \rho H^1_e(N_K)$ have finite norm, and for low Fourier modes $\eta(t)= e^{i\ell t}$ they are almost in the kernel; likewise for $\eta_2(t)\overline \kappa_\circ$. We consider the space 
\smallskip  
\be \rho^{1+\mu} H^{1}_\circ(N_K):= \left\{ \ph \in \rho^{1+\mu}H^1_e(N_K) \ \  \Big | \ \ \begin{matrix}\ph|_{\del N_k} \text{ satisfies boundary condition (ii') \ } \\  \br \ph , e^{i\ell t}\kappa_\circ \kt_{\rho^{1+\mu} H^1_\text{b}(P_1)}=0 \text{  \ for \ } |\ell|\leq L \\  \br \ph , e^{i\ell t}\overline \kappa_\circ \kt_{\rho^{1+\mu} H^1_\text{b}(P_1)}=0 \text{  \ for \ } |\ell|\leq L \\  \end{matrix} \right\} . \label{calHdef}\ee 
\smallskip 

\noindent Note the orthogonality condition uses the 2-dimensional Hermitian inner product on each Fourier mode, thus it does not include a pairing involving $\del _t \ph$.
\begin{prop} \label{3Dinjectivity} For  each weight $-\tfrac14<\mu \leq 0$, the operator 
$\slashed D: \rho^{1+\mu} H^1_\circ(N_K)\to \ \rho^\mu L^2(N_K)$ is Fredholm with $\text{Ind}_\C(\slashed D)=-(4L+2)$, and there are constants $C_\mu$ independent of $\delta$ such that
\smallskip
\be \|\ph\|_{\rho^{1+\mu} H^1_e} \leq C_\mu  \|\slashed D\ph\|_{ \rho^\mu L^2} \label{3Dellipticest}\ee
\smallskip
\noindent holds for $\ph \in \rho^{1+\mu}H^1_\circ$. In particular, $\slashed D$ is injective. 
\end{prop}

\begin{proof}
The boundary condition (ii') results in no boundary terms when integrating by parts because the oppositely oriented boundaries contribute canceling terms. The self-adjointness of the boundary condition implies the operator is Fredholm of index 0 without the orthogonality constraints of which there are $2(2L+1)$, which implies the index statement.  

Beginning with the case that $\mu=0$, integrating by parts yields:  

\bea
\int_{N_K} |\slashed D \ph|^2   \ dV &=& \|\del_t \ph\|^2_{ L^2}  \ + \ \|\slashed D_N\ph\|^2_{ L^2}\\ & &  +  \int_{N_K}{ \br \ph ,  \cancel{(\sigma_t\del_t \slashed D_N  + \slashed D_N \sigma_t \del_t)}\ph \kt
} \ dV. 
\eea

\noindent Since $\slashed D$ respects Fourier modes, it suffices to prove the estimate holds uniformly for each mode. For $|\ell|\leq L$, the orthogonality constraint in (\refeq{calHdef}) and (\refeq{BVPPoincareII}) show that $\slashed D_N$ is injective with a uniform estimate. For $|\ell| \geq L$, one has $c\|\tfrac{\ph}{r}\|_{ L^2}\leq c\delta^{-1}\|\ph\|_{ L^2} \leq \|\del_t\ph\|_{ L^2}$. Borrowing from the $|\del_t \ph|$ term and invoking (\refeq{BVPellipticII}) shows that 

$$\|\ph\|_{\rho H^1_b}^2 \ + \ \tfrac{1}{2}\|\del_t \ph\|_{  L^2}  \ \leq \  C( \|\slashed D_N\ph\|_{L^2} + \|\del_t \ph\|_{ L^2})$$
\noindent and the left side is precisely the $\rho H^1_e$-norm.
For $|\mu|<\tfrac{1}{4}$ a similar argument applies with an additional integration by parts used to absorb the additional cross-term arising from the derivative of the weight (see \cite[Claim. 7.19.1]{PartI}). 
\end{proof}

\begin{rem}\label{thmmainbobstruction} Proposition \ref{3Dinjectivity} only holds for $|\mu|<\tfrac14$, which is the reason Theorem \ref{mainb} fails in the case of 1-forms. Since it is not possible to ensure the primitive $f$ in the model solution (\refeq{approximatesol1form}) vanishes identically along $K_i$, weight $\mu<-\tfrac{3}{2}$ would be required for the error to approach zero in this case. 
\end{rem}

The $(4L+2)$-dimensional cokernel of $\slashed D$ on $\rho^{1+\mu} H^1_\circ(N_K)$ can be explicitly described as follows.

\begin{lm} For each weight $-\tfrac14<\mu\leq 0$, the orthogonal complement of the range of the operator $\slashed D: \rho^{1+\mu} H^1_\circ(N_K)\to \ \rho^\mu L^2(N_K)$ is given by the linear span of  $\Psi_\ell, \overline \Psi_\ell$ for $|\ell|\leq L$, where these are scalings of modified Bessel 
functions of the first kind with asymptotics 

\be |\Psi_\ell| , |\overline \Psi_\ell| \sim \delta^{-1+\mu}\frac{1}{R_0^{\mu}}\frac{\sqrt{2|\ell|\delta R_0}}{\text{Exp}(\delta|\ell| R_0)} \frac{e^{|\ell| \delta |R|}}{|R|}R^{2\mu}.\label{Besselasymptotics}\ee
for  $|\ell|>>0$ and $R<<0$. 

\end{lm}  

\begin{proof} Scale by $(t,r)\mapsto (\delta^{-1}t, R)$ so that $N_K\simeq S^1_{\delta}\times [R_0, R_0] \times S^1$ where the first circle has circumference $2\pi \delta^{-1}$. One has that 
$$\delbar_N \br R \kt^{-1/2}u= \br R \kt^{-1/2} e^{i\theta}(\del_r + \tfrac{i}{\br R\kt}) u.$$
\noindent Furthermore, in the $e^{ik\theta}$ Fourier mode, $$e^{i\theta}\left(\del_r - \frac{k}{\br R\kt}\right) W^k u =W^k\delbar u \hspace{1cm}\text{where}\hspace{1cm}W^k =\frac{\text{Exp}(k \int_0^R \tfrac{1}{\br s \kt} ds) }{R^k}$$
\noindent and $\delbar=e^{i\theta}(\del_R + \tfrac{i}{R})\del_\theta$ is the normal $\delbar$ operator. For $\del$, the same applies but with $W^{-k}$. Consequently for $\mu=0$, $\slashed D(\br  R\kt^{-1/2} W^k \psi )=0$ is a solution of the adjoint if and only if $\slashed D_0\psi=0$ where $\slashed D_0$ is the normal Dirac operator in the product metric on $(S^1_\delta \times (D^2-\{0\})$. The same conversion, mutatis mutandis holds for $R<0$ by first replacing $R\leftrightarrow -R$ and then conjugating.

For $R>0$, decomposing in Fourier so that $\psi_{k\ell}=e^{i\ell t}(e^{ik\theta}\alpha(R) , e^{i(k+1)\theta}\beta(R))$ shows
$$\slashed D_0 \psi=\begin{pmatrix} - \delta\ell & (-\del_r -\tfrac{k+1}{r}) \\ (\del_r - \tfrac{k}{r})  & \delta \ell\end{pmatrix}\begin{pmatrix}\alpha_{k\ell} \\ \beta_{k\ell}\end{pmatrix}.$$

\noindent which has solutions 
$$\psi_{k\ell} = \hat I_{k}(\delta |\ell| R) + \hat K_{k}(\delta |\ell| R) \hspace{1cm} \text{where} \hspace{1cm} \hat I_{k}(r)=\begin{pmatrix} e^{ik\theta} I_k(r) \\ -\text{sgn}(\ell) I_{k+1}(r) \end{pmatrix}$$
\noindent and likewise for $\hat K_k$, where $I_k, K_k$ are the modified Bessel functions of the first (exponentially growing) and second (exponentially decaying) kind respectively. 

Returning to $N_K$, the permissible solutions are those that extend continuously as $L^2$ functions across the origin, and satisfy the boundary conditions at $R=R_0$. The symmetric patched $\hat I_{-1}, \hat I_0$ solutions satisfy both of these for $|\ell|\leq 2L+1$, and they must exhaust the cokernel since they have equal dimension (it is easy to confirm that $\hat K_k$ is not integrable across the origin for any $k$ since $W_{-k}$ is needed for $R<0$, and that continuity at $0$ means the boundary conditions cannot be satisfied by $\hat I_k$ alone for other $k,\ell$). 

We conclude the cokernel elements are of the form $$\Psi_{\ell}=\frac{e^{i\ell t}}{C_\ell} \br R\kt^{-1/2}W\hat I_{-1}(\delta |\ell| R) \hspace{2cm} \overline\Psi_{\ell}=\frac{e^{i\ell t}}{C_\ell}\br R\kt^{-1/2}W\hat I_0(\delta |\ell| R)$$
\noindent where $W$ acts by $W^k$ in the $k^{th}$ Fourier mode. Since $W\sim 1$ for $R>>0$, since $R_0=\rho_0\delta^{-1}$ shows that the asymptotic expansion of $I_{k}(r)\sim e^{-r}/r^{1/2}$ at $r\to \infty$ dominates for large $|\ell|$. Combining with the $\br  R\kt^{-1/2}$ factor and normalizing in $L^2$ produces (\refeq{Besselasymptotics}), where the normalization is up to a $(1+O(R_0^{-1}))$ scaling factor (the factor of $\delta^{-1}$ arises from scaling $R=\delta^{-1}r$ back down). 

For $-1/4<\mu<0 $, the orthogonal complement of the range differs by multiplying by $R^{2\mu}$ and adjusting the normalization factor accordingly. 
\end{proof}

Now we introduce a family of $t$-dependent versions of the perturbations (\refeq{perturbation}). Let \bea \mathcal B_\delta^s(K;\C^2)&\subseteq& H^{s}(K;\C^2) \\ \mathcal K_\delta^s(K;\C^2)&\subseteq& \rho^\mu H^s_\text{b}(N_K; S)\eea 
\noindent denote the two finite-dimensional ($\delta$-dependent) subspaces defined as the complex span of $e^{i\ell t}$ and $\Psi_\ell, \overline{\Psi}_\ell$ respectively for $|\ell|\leq L$. Both are equipped with their inherited norm for each $s$. Similarly, we denote $\mathcal R^s(N_K;S)=\text{Range}(\slashed D)\cap \rho^\mu H^s_{\text {b}}(N_K;S)$.

Let $(\slashed D, B_\delta): \rho^{1+\mu} H^1_\circ(N_K; S) \oplus \mathcal B_\delta(K;\C)   \to    \rho^\mu L^2(N_K; S)$ be extended operator given by
\begin{eqnarray} \label{3DextendedDirac} 
(\ph,\xi)&\mapsto & \slashed D \ph  + \delta^{-1+\mu}B(\xi(t))\Phi_\circ \nonumber,
\end{eqnarray}
\noindent where we now allow the coefficients in (\refeq{perturbation}) to be $t$-dependent, and the factor of $\delta^{-1+\mu}$ is introduced to make the perturbation scale identically to the other terms. Splitting the codomain as $\rho^\mu L^2\simeq \mathcal K\oplus \mathcal R$, the extended operator satisfies the following, where the unadorned versions denote the $s=0$ spaces. Note the similarity to the form of Corollary \ref{fredholmddD}, including the loss of regularity. 
\begin{prop}  \label{obstructioncancellation} For each $-\tfrac14<\mu\leq 0$, the extended operator
\be (\slashed D, B_\delta)=\begin{pmatrix} \pi_\mathcal K B_\delta & 0 \bigskip \\ \pi_\mathcal R B_\delta & \slashed D\end{pmatrix}: \begin{matrix} \mathcal B(K;\C^2) \\ \oplus \\ \rho^{1+\mu} H^1_\circ(N_K;S)\end{matrix} \lre \begin{matrix} \mathcal K_\delta^{1/2}(K;\C) \\ \oplus \\ \mathcal R(N_K;S)\end{matrix} \ee
\noindent is an isomorphism with inverse bounded uniformly in $\delta$ (but depending on $\mu$), where $\pi_\mathcal K,\pi_\mathcal R$ denote the $L^2$-orthogonal projections.  
\end{prop}
\begin{proof} Scaling and the bound from Lemma \ref{cancellationlem} show that $$\|B_\delta(\xi(t))\|_{\rho^\mu L^2}\leq C \|\xi(t)\|_{\rho^\mu L^2(K;\C)},$$
\noindent hence $\pi_\mathcal R B_\delta$ is uniformly bounded. By Proposition \ref{3Dinjectivity}, $\slashed D$ is an isomorphism onto its range with uniformly bounded inverse. It therefore suffices to show the top left component is an isomorphism. 

For this, we calculate the inner product, to leading order beginning with $\xi=(e^{i\ell t}, 0)$ and $\mu=0$: 

\bea
\br \delta^{-1}B(\xi(t))\Phi_\circ , \Psi_\ell\kt_{L^2}&=&  \int_{0}^{2\pi} \br (\xi_1(t)+i\xi_2(t)) \overline d, e^{i\ell t}\kt dt \cdot \frac{1}{\sqrt{R_0}}  \frac{\sqrt{2|\ell|\delta R_0}}{\text{Exp}(|\ell| \delta R_0)}\int_{-R_0}^{-R_0/2} \chi_0(R) \frac{1}{\sqrt{|R|}} \frac{e^{|\ell|\delta |R|}}{R}R dR\\
&=& \frac{\overline d}{\sqrt{R_0}}  \frac{\sqrt{2|\ell|\delta R_0}}{\text{Exp}(|\ell| \delta R_0)} \frac{1}{\sqrt{|\ell|\delta}} \int_{-p_0}^{-p_0/2}\frac{e^{|p|}}{\sqrt{|p|}}dp \hspace{2cm}(p=|\ell| \delta R \ , \  p_0 =|\ell|\delta R_0)\\
&=& \frac{\overline d\sqrt{2}}{\text{Exp}(|\ell|\delta R_0)}  \left[ \frac{e^{|p|}}{\sqrt{|p|} }(1 + O(p^{-1}))\right]_{p=-p_0}^{p=-p_0/2} \\
&=& c_1 \frac{\overline d}{\sqrt{|\ell|}} (1+ O(|\ell|^{-1}))
\eea
\noindent for some constant $c_1>0$, because $\delta R_0=\rho_0$. Note that in the first line we have substituted $\delta^{-2}rdr=RdR$. Moreover, since the radial part of $B_\delta(\xi)\Phi_0$ and $\Psi_\ell$ are both positive functions, the inner product is non-zero for each $\ell$. The same calculation holds for $\overline \Psi_\ell$, thus we conclude 
$$\br  \pi_\mathcal K(B_\delta(\xi_{1\ell} e^{i\ell t}, \xi_{2\ell} e^{i\ell t})), a \Psi_\ell + b \overline \Psi_\ell \kt_{L^2}=\frac{c}{\sqrt{|\ell|}} \big\br \begin{pmatrix}(\xi_{1\ell} + i\xi_{2\ell} )\overline c \\ (\xi_{1\ell} - i\xi_{2\ell})\overline d\end{pmatrix},  \begin{pmatrix} a \\ b \end{pmatrix}\big\kt + O(|\ell|^{-3/2}).$$
\noindent It follows that $\pi_\mathcal K B_\delta: \mathcal B_\delta(K;\C^2)\to \mathcal K_\delta^{1/2}(K;\C^2)$ is an isomorphism when $|c|, |d|>0$. If only $|c|^2 + |d|^2>0$, the same alteration as in Lemma \ref{cancellationlem} applies. For $-\tfrac14<\mu<0$ the proof is the same carrying along additional factors involving $\mu$. 
\end{proof}

\subsubsection{Parametrix Patching}
Now let $\chi_Y, \chi_N$ be a partition of unity on $Y_K$ formed from logarithmic cut-off functions as follows. $\chi_Y$ is equal to $1$ on the bulk of $Y_1, Y_2$ for $|R|> \sqrt{R_0}$ and vanishing for $|R|\leq R_0^{3/8}$, and $\chi_N=1-\chi_Y$. We define global orthogonality constraints by 

\be r  \rho^{1+\mu} H^1_{\circ}(Y_K; S)= \left\{ \ph \in r \rho^{1+\mu} H^1_e(Y_K ; S)  \ \Big | \  \chi_N \ph \in \rho^{1+\mu} H^1_{e,\circ}(N_K; S) \right\}\label{orthogonalityspaces},\ee
\noindent where the latter space is as defined in (\refeq{calHdef}). Note that the cut-off is such that the boundary conditions are automatically satisfied, thus only the $(4L+2)$-orthogonality constraints apply. We also write $\rho^\mu L^2 \cap \mathcal K^s$ for the space of sections such that $\pi_\mathcal K(\chi_N \ph) \in \mathcal K^s(K;\C^2)$ on $N_K$.  The following implies strict uniformity in the sense of Definition \ref{strictlyuniform} in the case of a pinching torus neck.

\begin{prop} \label{PP1}For $\mu=-\tfrac{1}{8}$, and $\rho_0>0$ sufficiently small, 
$$\slashed D: r \rho^{1+\mu} H^1_{e,\circ}(Y_K; S )\lre \rho^\mu L^2 (Y_K; S) \cap \mathcal K^{1/2}(K;\C^2)$$
is left semi-Fredholm, and there is a constant $C_\mu$ independent of $\delta$ such that 
$$\|\ph\|_{r\rho^{1+\mu} H^1_e}\leq C_\mu (\|\slashed D\ph \|_{\rho^\mu H^1_\text{b}} \ + \ \|Q \ph\|_{\rho^\mu L^2})$$
\noindent where $Q$ is the restriction to a compact domain in $Y_K-(\mathcal Z \cup N_K)$. Moreover, there is a finite-dimensional space $\mathcal K$ spanned by $\chi_Y \Phi_\alpha$ for $\Z_2$-harmonic spinors $\Phi_\alpha$ on either $Y_1,Y_2$ such that the above estimate holds uniformly omitting $Q$ if $\ph \perp \mathcal K$. 

\end{prop} 
\begin{proof}
The proof is analogous to that of Propositions \ref{uniformdirac} and \ref{uniformderham}. 

\begin{enumerate}
\item[] \underline{Step 1:} Let $\zeta_N(\rho)$ be a logarithmic cut-off function equal to $1$ on $\text{supp}(\chi_N)$ and supported where $|R|\leq \delta^{3/8}$, and consider the metric $$g_\delta(N_K)= dt^2 + dr^2+ (\rho^2+\delta^2)d\theta^2 + \chi_N(\rho) (\chi_1 h_1 +\chi_2h_2),$$ so that $g_\delta(N_K)=g_\delta$ (defined in \refeq{gluedmetric}) on the support of $\chi_N$. Since $h_i=O(\rho)$, for $\delta$ sufficiently small Proposition  \ref{3Dinjectivity} holds equally well on $N_K$ using the metric $g_\delta(N_K)$. Note in this that since the elements of $K^{1/2}(K;\C^2)$ concentrate near $\rho=\rho_0$, the perturbation to $\slashed D$ arising from the change in metric has exponentially small (in $|\ell|$) pairing with $\Psi_\ell, \overline \Psi_\ell$. Thus the perturbation is indeed bounded (and small for small $\delta$) into the higher regularity subspace.   

\medskip

\item[] \underline{Step 2:} For weight $-1/4<\mu<0$, the operator $\delbar: r^{1+\mu}H^1_\text{b}(\C-\{0\}; \C) \to r^\mu L^2(\C-\{0\};\C)$ is invertible by similar considerations to Lemma \ref{infiniteneck}. Integration by parts similar to \ref{3Dinjectivity} shows that $\slashed D: S^1 \times (D^2-\{0\})$ is invertible for the model metric. A preliminary parametrix patching on the closed manifolds $Y_i$ then shows that 

$$\|\ph\|_{r \rho^{1+\mu}H^1_e(Y_i)}\leq C_\mu \left(\|\slashed D\ph\|_{\rho^{\mu}L^2(Y_i)} \ + \  \|Q\ph\|_{\rho^\mu L^2(Y_i)}\right)$$

\noindent holds on the punctured manifolds $Y_i - (\mathcal Z_i \cup K_i)$ for $\ph \in r \rho^{1+\mu}H^1_e(Y_i)$, where $Q$ is the projection on a fixed compact region not containing $K_i$ (where the $\Z_2$-harmonic spinors on $Y_1,Y_2$ are necessarily non-zero by analytic continuation). Moreover, if $\ph$ is $L^2$-orthogonal to the finite-dimensional space of $\Z_2$-harmonic spinors restricted to this fixed compact region by smooth bump functions, then the $Q$ term is unnecessary. 

\medskip 
\item[]\underline{Step 3:} Let $P_N$ be the parametrix from Step 1, and $P_Y=P_1 + P_2$ be the parametrices from Step 2. Set 
$$P=\zeta_Y P_Y \chi_Y \ + \ \zeta_N P_N \chi_N,$$
 where $\zeta_Y$ is a logarithmic cut-off function equal to $1$ on the support of $\chi_Y$. The conclusion now follows from the same calculation as (\refeq{patchingerrorterms}) since $d\chi_Y\sim O(\log(1/\delta)^{-1})$. Note that the terms involving $d\chi_N$ are small in the higher regularity space by the same argument as in Step 1.  
\end{enumerate}
\end{proof}

We now let $\Phi_\circ=a\chi_1\Phi_1$ be the approximate solution on $Y_1$ (so that the constant $R_1=\delta^{-1/4}$). Recall that $R_0=\rho_0\delta^{-1}$ and that the perturbation (\refeq{perturbation}) was defined including a yet-unspecified constant $\e$ extending $B(\xi)$ outside $N_K$. 

Recall that $\mathcal B_\delta(K;\C^2)\subseteq L^2(K;\C^2)$ denotes the finite-dimensional space spanned by the lowest $2\lfloor \delta^{-1}\rfloor$ Fourier modes. We now define the universal Dirac operator with perturbations by 
\begin{eqnarray}{\slashed {\mathbb D}}_{\mathbb B}: r \rho H^1_{e,\circ}\oplus L^{2,2}(\mathcal Z; \C)\oplus \mathcal B_\delta (K;\C^2)&\lre& \rho^{\mu}L^2 \\ 
(\mathcal Z, \ph, \xi)&\mapsto & \slashed {\mathbb D}(\mathcal Z,\ph) +  \delta^{-1+\mu}B(\xi(t))\Phi_\circ. \nonumber \end{eqnarray}

\noindent The (proof of) the previous proposition implies the following universal version. The precise meaning of the codomain in the following statement is given in the proof. Recall that a small yet unspecified constant $\epsilon$ appeared in the definition of (\refeq{perturbation}). 
\begin{prop} For $\mu=-\tfrac{1}{8}$ and for $\rho_0, \delta, \epsilon>0$ sufficiently small, 
\smallskip
\be (d\slashed {\mathbb D}_{\mathbb B})_{(\mathcal Z_\alpha,\Phi_\alpha^\delta)}:  r \rho^{1+\mu} H^1_{e,\circ}\oplus L^{2,2}(\mathcal Z; \C)\oplus \mathcal B_\delta(K;\C^2)\lre \rho^{\mu}L^2 \cap \mathcal K^{1/2} \cap \text{\bf Ob}^{3/2}(\mathcal Z_\alpha)\label{extendeddDB}\ee

\smallskip 

\noindent is Fredholm of Index 0. Moreover, on the complement of a fixed $\delta$-independent finite dimensional subspace $\mathcal H$, there is a $C_\mu$ such that 
$$\|(\ph, \eta, \xi)\|_{r\rho H^1_e \oplus L^{2,2}\oplus L^2} \leq C_\mu \left( \|\slashed{\mathbb D}_{\mathbb B}(\ph,\eta,\xi)\|_{\rho^\mu L^2\cap \mathcal K^{1/2}\cap \text{\bf Ob}^{3/2}} \ + \ \|Q\ph\|_{\rho^\mu L^2}\right)$$
\noindent holds uniformly in $\delta$. 
\end{prop}
\begin{proof}
As in Step 2 of Proposition \ref{PP1}, the operator \be \slashed D: r\rho^{1+\mu}H^1_e(Y_i -(\mathcal Z_i\cup K_i))\to \rho^\mu L^2(Y_i -(\mathcal Z_i \cup K_i))\label{onKcomplement}\ee
\noindent is left semi-Fredholm. In fact, since the indicial roots of the model operator at $K$, the analogue of Lemma \ref{expansions} along $K$ shows that the kernel and cokernel of (\refeq{onKcomplement}) coincide with the $\Z_2$-harmonic spinors and the cokernel described in Proposition \ref{cokernel} respectively.  

We now define the codomain to consist of those $g\in \rho^\mu L^2(Y_K; S)$ such that 
\bea \chi_Y g &\in& \rho^\mu L^2 \cap \text{\bf Ob}^{3/2}(\mathcal Z_i) \hspace{2cm}\text{on }Y_i\\
\chi_N g & \in & \rho^\mu L^2 \cap \mathcal K^{1/2} \hspace{2.8cm}\text{on }N_K\eea
\noindent where $\text{\bf Ob}^{3/2}$ is as in Corollary \ref{fredholmddD}. 

Let $\mathbb P_Y=\mathbb P_1 + \mathbb P_2$ and $\mathbb P_N$ be the universal parametrices provided by the analogue of Corollary  \ref{fredholmddD} for (\refeq{onKcomplement}), and Proposition  \ref{obstructioncancellation} respectively. Set 

$$\mathbb P=\zeta_Y \mathbb P_Y \chi_Y + \zeta_N \mathbb P \chi_N,$$
\noindent where it is understood that $\zeta_Y,\zeta_N$ multiply only the spinor component. Taking $P_1$ sufficiently small so that $\text{supp}(\zeta_Y)\cap P_1=\emptyset$, it is obvious that the image of $\mathbb P$ obeys the orthogonality constraints. A quick computation shows that $$(\text{d}\slashed {\mathbb D}_{\mathbb B}) \mathbb P(g)=g + O\left(\log(\delta)^{-1}\right) +  O(\epsilon) $$
\noindent where the first term arises from $d\chi_N,d\chi_Y$, and the second term arises from the portion of $B_\delta(\xi(t))$ supported outside $N_K$, which has $\rho^\mu L^2$-norm bounded by $C \epsilon \|\xi\|_{L^2}$. It follows that $\text{d}\slashed {\mathbb D}_{\mathbb B}$ is surjective with uniform estimate on the complement of the kernel for $\epsilon,\delta$ sufficiently small, with $Q$ as in Proposition \ref{PP1}. 

For the Fredholm index it suffices to consider $\xi(t)=0$ and eliminate the orthogonality constraints, as these have the same dimension. In this case, the reverse parametrix patching shows that $\text{d}\slashed{\mathbb D}_{\mathbb B}$ has finite-dimensional kernel and so is Fredholm. The index statement then follows from a standard excision argument, and this implies the operator is injective. 
\end{proof}

\subsubsection{Proof of Theorem \ref{mainb}}
\label{proofofmainb}
\begin{proof}[Proof of Theorem \ref{mainb}] By a simple transversality argument, we may perturb $K_1$ so that $\Phi_1$ is non-vanishing along $K_1$. Lemma \ref{toruspincherror} shows that the approximate solutions (\refeq{approximatesoltorus}) have error approaching zero for $\mu=-\tfrac{1}{8}$. The theorem is then a result of the following small variation of Theorem \ref{NMIFT}: consider the extended operator 
$$\slashed{\mathbb D}_{\mathbb B}: r \rho^\mu \mathbb H_\mathbb B \to p_1^\star(\rho^\mu  \mathbb L^2) $$
\noindent where the domain and codomain are tame Fr\'echet vector bundles modeled on 
$$\bigcap_{m \geq 0} H^{m+2}(\mathcal Z;\C) \oplus r \rho^{1+\mu}H^{m,1}_{\circ,\text{b}}(Y_K\setminus\mathcal Z;S) \oplus \mathcal B_\delta^m(K;\C^2) \hspace{2cm} \bigcap_{m\geq 0} \rho^\mu H^m_\text{b}(Y_K\setminus \mathcal Z; S)$$
\noindent respectively. The middle space in the domain denotes version of \eqref{doublebspaces} also satisfying the orthogonality constraints of (\refeq{orthogonalityspaces}). 
$\mathcal B^m(K;\C^2)$ is endowed with the standard family of smoothing operators given by truncating Fourier modes, which obviously preserves the subspace. The smoothing operators on the middle factor may be adjusted to respect the orthogonality constraint.  

It is easy to check the invertibility of $\text{d}\slashed{\mathbb D}_{\mathbb B}$ on an open neighborhood and the requisite tame estimates of \cite[Sec. 8]{PartII} hold in this case as well. Finally, the finite-dimensional cokernel of (\refeq{extendeddDB}) may be accounted for as in the proof of (ii) in Theorem \ref{NMIFT}.  
\end{proof}

\begin{rem}
Since $\mathcal B^m(K;\C^2)$ has finite $\delta$-dependent dimension, the above does not show the solution is smooth in $\delta$. This may easily be amended by adding a $H^m(K;\C^2)$ factor to the codomain, and extending the map by a smooth interpolation between the constraints defining $r\rho^{1+\mu}H^{m,1}_{\circ,\text{b}}$ and $\mathcal B^m$ (see \cite[Sec. 10.1]{PartIII} for similar arguments).
\end{rem}
 
\section{Examples and Applications}
\label{section4}
In this section we apply Theorems \ref{maina} and \ref{mainb} to construct new examples of $\Z_2$-harmonic 1-forms and spinors. To begin, we construct examples on Seifert--fibered 3-manifolds which are used as the building blocks in the gluing construction. 

\subsection{Orbifold Riemann Surfaces and Seifert--fibered spaces}
\label{section4.1}
To begin with, we introduce some background on orbifold Riemann surfaces and Seifert--fibered 3-manifolds. For more detailed explanations, we refer to \cite[Section 2]{MOY1996} and \cite{Orlik72Seifert}.

Recall that a Seifert--fibered 3-manifold $Y$ admits an action of $U(1)$ with finite stabilizers. Thus, we may view the manifold as a fiber bundle $\pi: Y \to \Sigma$ with fiber $S^1$ over a 2-dimensional orbifold $\Sigma$. Let $i\eta$ be the connection 1-form of a constant curvature $U(1)$ connection on $Y$. Then $Y$ may be endowed with the metric
\begin{equation}
    g_{s,V} = s^2\eta^2 + \pi^*(g_\Sigma) \label{gY}
\end{equation}
of fiber diameter $s$, where $g_\Sigma$ is a metric of volume $V$ on $\Sigma$. For the duration of this subsection, it is understood that terms such as "line bundle" and "metric" refer to the orbifold versions when referring to objects on $\Sigma$.

\subsubsection{Orbifold Riemann Surfaces} To explain more precisely, recall that an orbifold Riemann surface is a Hausdorff space $|\Sigma|$ with a finite set of marked points and integral multiplicities $(x_1, \alpha_1), \ldots, (x_n, \alpha_n)$ with $\alpha_i \geq 1$, and an atlas of coordinate charts 
$$
\phi_i: (D, 0) \to (U_i, x_i), \quad i = 1, \ldots, n, \quad \phi_x: D_x \to U_x, \quad \text{for} \; x \in \Sigma \setminus \{x_1, \ldots, x_n\},
$$
where $D$ is the standard complex disk such that $\phi_i$ induces a homeomorphism from $(D, 0) / \mathbb{Z}_{
\alpha_i}$ to $(U_i, x_i)$ for $i=1,...,n$, and $\phi_x$ are homeomorphisms for $x\neq x_i$. All transition functions are holomorphic. Additionally, the orbifold structure endows the underlying topological space $|\Sigma|$ with the structure of a complex curve as follows: in local coordinates $(U_i, x_i)$, if we denote the complex coordinate on $D$ by $w$, then $w^{\alpha_i}$ defines a complex coordinate over $D / \mathbb{Z}_{\alpha_i }$, which is a neighborhood of the marked point $x_i$. A basic topological invariant of the orbifold structure is the orbifold Euler characteristic, given by
$$
\chi^{\mathrm{orb}}(\Sigma) := 2 - 2\gamma + \sum_{i=1}^n \left( \frac{1}{\alpha_i} - 1 \right),
$$
where $\gamma$ is the genus of the underlying smooth curve $|\Sigma|$.

The notions of bundles, connections, and sections naturally extend to the orbifold setting by considering an equivariant structure over the orbifold points. For example, an $n$-dimensional orbifold bundle $E$ is a collection of $\mathbb{Z}_{\alpha_i}$-equivariant $n$-dimensional vector bundles $E_i$ over $U_i$ and vector bundles $E_x$ over $U_x$ together with a 1-cocycle of transition functions over the overlaps. Note that on each $U_i$, the data of a $\Z_{\alpha_i}$-equivariant vector bundle of rank $n$ (up to isometry) is equivalent to that of a representation $\rho_i: \mathbb Z_{\alpha_i} \to \mathrm{GL}_n(\C)$. The notion of holomorphic bundles extends similarly. An orbifold connection $\nabla$ on an orbifold bundle $E$ is a collection $\nabla_i$ of $\mathbb{Z}_{\alpha_i}$-equivariant connections over the disks $E|_{U_i}$ and a  connection in the standard sense over each $E|_{U_x}$, which are compatible on intersections. Similarly, a section of an orbifold bundle $E$ is a collection of compatible $\mathbb{Z}_{\alpha_i}$-equivariant sections on $E|_{U_i}$ and sections on $E|_{U_x}$.

Two types of orbifold line bundles are of particular importance: 
\begin{itemize}
\item[(1)] The orbifold canonical bundle $K_\Sigma$
\item[(2)] The canonical line bundles $H_{x_i}$ of the orbifold points. 
\end{itemize}

\noindent Since the rotation $\mathbb{Z}_{\alpha}$ on the disk $D$ lifts to the cotangent bundle $T^\star D$, it defines an orbifold line bundle over $D/\Z_{\alpha}$. The orbifold canonical bundle $K_\Sigma$ is the holomorphic line bundle formed by gluing cotangent bundles of $U_i$ and $U_x$ together via the complex derivatives of the transition functions. The canonical line bundle $H_{x_i}$ of an orbifold point whose neighborhood $U_{i}$ is isomorphic to $D/\Z_{\alpha_i}$ is defined as follows:  $H_{x_i}$ is trivial away from $x_i$, and over $U_{i}$, it is given by the $\mathbb{Z}_{\alpha_i}$-equivariant line bundle $D_w \times \mathbb{C}_z$ with the action of $l \in \mathbb{Z} / \alpha_i \mathbb{Z}$ given by $l \cdot (w, z)= (e^{\frac{2\pi i l}{\alpha_i}} w, e^{\frac{2\pi i l}{\alpha_i}} z)$. 

The line bundles $H_{x_i}$ serve as local generators for the topological isomorphism classes of orbifold line bundles over $\Sigma$ in the following sense. Given an orbifold line bundle $L$, near each orbifold point $x_i$, there exist local invariants $0 \leq \beta_i < \alpha_i$ such that $L \otimes H_{x_1}^{-\beta_1} \otimes \cdots \otimes H_{x_n}^{-\beta_n}$ is an orbifold line bundle which is naturally isomorphic to a smooth line bundle over the smooth curve $|\Sigma|$. This line bundle is called the desingularization of $L$ and denoted as $|L|$. Moreover, the holomorphic sections of a holomorphic orbifold line bundle $L$ over $\Sigma$ are identified with the holomorphic sections of the desingularization $|L|$ over $|\Sigma|$. 

\begin{defn}
\label{orbifolddegree}
    Given an orbifold line bundle $L$ over $\Sigma$, the collection of integers $(b; \beta_1, \cdots, \beta_n)$ is called the Seifert invariant of $L$ over $\Sigma$, where $b = \deg(|L|)$. The degree of the orbifold line bundle $L$ is defined as $\deg(L) = b + \sum_{i=1}^n \frac{\beta_i}{\alpha_i}$.
\end{defn}

Given an orbifold line bundle $L$ with Seifert invariant $(b; \beta_1, \cdots, \beta_n)$ such that $\beta_i$ are relatively prime to $\alpha_i$, then the circle bundle of $L$ forms a smooth 3-manifold $Y$ \cite[Page 9]{MOY1996}, known as a Seifert--fibered space. The collection of local invariants $(\gamma, b; (\alpha_1, \beta_1), \cdots, (\alpha_n, \beta_n))$ is called the {\bf Seifert invariant} of the 3-manifold $Y$.

\subsubsection{Orbifold sections} 
The notion of holomorphic sections extends to orbifolds and there is an analogue of the standard Riemann-Roch theorem. 

First, note that for two orbifold line bundles, $L$ and $L'$, Definition \ref{orbifolddegree} implies $\deg(L \otimes L') = \deg(L) + \deg(L')$. Moreover, if we denote the Seifert invariants of $L$ and $L'$ as $(b; \beta_1, \cdots, \beta_n)$ and $(b'; \beta_1', \cdots, \beta_n')$, respectively, then the Seifert invariant for their tensor product is given by the following formula: Let 
$$
c := \sum_{i=1}^n \lfloor (\beta_i + \beta_i') / \alpha_i \rfloor, \quad \delta_i := \beta_i + \beta_i' - \sum_{i=1}^n \lfloor (\beta_i + \beta_i') / \alpha_i \rfloor \alpha_i,
$$ 
then based on the local description, the Seifert invariant for $L \otimes L'$ is
\begin{equation}
    \label{eq_seifert_invariant_tensor}
    (b + b' + c; \delta_1, \cdots, \delta_n).
\end{equation}
For example, by a straightforward computation, $K_{\Sigma}$ has Seifert invariant $(2\gamma - 2; \alpha_1 - 1, \cdots, \alpha_n - 1)$, and the Seifert invariant for $K^2_{\Sigma}$ is therefore $(4\gamma - 4 + n, \alpha_1 - 2, \cdots, \alpha_n - 2)$.

We will need the following extensions of standard results to the orbifold case. We refer to \cite{nasatyr1995orbifold} for a more detailed discussion.

\begin{prop}[{\cite[Corollary 1.4]{nasatyr1995orbifold}}]
    \label{prop_vanishing_negative_degree}
    Suppose $\deg(L) \leq 0$, then $H^0(L) = 0$, unless $L$ is trivial.
\end{prop}

We also have the following Kawasaki-Riemann-Roch theorem.

\begin{thm}[{\cite{kawasaki1979riemann}}]
    \label{thm_orbifold_Riem_Roch}
    Let $L$ be a holomorphic orbifold line bundle over $\Sigma$ with $|L|$ as the desingularization, then 
    $$
    H^0(L) - H^0(L^{-1} \otimes K_{\Sigma}) = 1 - \gamma + \deg(|L|).
    $$
    \qed
\end{thm}

\medskip 
\subsubsection{Orbifold spin structure and Seifert--fibered space}
Now, we will introduce orbifold $\mathrm{Spin}$ and $\mathrm{Spin}^c$ structures, which have been studied in \cite{belgun2007singularity, geiges2012generalised}, and discuss their extensions to Seifert--fibered 3-manifolds. For more details, we refer to \cite[Section 5]{MOY1996}.

A spin structure $\mathfrak{s}_0$ on a Riemann surface $\Sigma$ is a square root of the tangent bundle $K_{\Sigma}^{\frac{1}{2}}$, which can also be understood as the complex line associated to a fiberwise connected double covering of the unit tangent bundle of $K_{\Sigma}$. 

For an orbifold point $x_i$, the existence of a spin structure on $\Sigma$ requires a lift of $\mathbb{Z}_{\alpha_i} \subset \mathrm{SO}(2)$ to some $G_x \subset \mathrm{Spin}(2)$ that projects isomorphically onto $\mathbb{Z}_{\alpha_i}$ via the projection from $\mathrm{Spin}(2) \to \mathrm{SO}(2)$. It is straightforward to verify that the group $\mathbb{Z}_{\alpha_i}$ can be lifted to $\mathrm{Spin}(2)$ if and only if $\alpha_i$ is odd. The converse statement is also true.

\begin{prop}[{\cite[Theorem 3]{geiges2012generalised}}]
    An orbifold $\Sigma$ has a spin structure if and only if $\alpha_1, \cdots, \alpha_n$ are odd.
\end{prop}

\noindent For an orbifold $\mathrm{Spin}^c(2)$ structure $\mathfrak{s}_0^c$, there is no obstruction, and the K\"ahler structure on $\Sigma$ induces a canonical orbifold $\mathrm{Spin}^c(2)$ structure $\mathfrak{s}_0^c \cong \mathbb{C} \oplus K_{\Sigma}^{-1}$.

Let $L_0$ be an orbifold line bundle over $\Sigma$ which defines a Seifert--fibered space $\pi: Y \to \Sigma$. Then, any orbifold line bundle $L$ and structures on $\Sigma$ naturally extend to the Seifert--fibered space $Y$ given by $\pi^*L$. This leads to a faithful correspondence if one equips it with a connection. In particular, we have 

\begin{prop}[{\cite[Proposition 5.1.3]{MOY1996}}]
    There is a natural one-to-one correspondence between pairs of (orbifold) bundles with connection over $\Sigma$ and (usual) bundles with connection over $Y$, whose curvature forms pull up from $\Sigma$ and whose fiberwise holonomy is trivial. Furthermore, this correspondence induces an identification between orbifold sections of the orbifold bundle over $\Sigma$ with fiberwise constant sections of its pull-back over $Y$.
\end{prop}

\noindent Furthermore, if $\mathfrak{s}_0$ is a $\mathrm{Spin}(2)$ structure on $\Sigma$, then $\pi^*(\mathfrak{s}_0)$ defines a spin structure on $Y$. In this case, the spinor bundle $\slashed{S}$ is isomorphic to $\pi^*(K_{\Sigma}^{1/2} \oplus K_{\Sigma}^{-1/2})$, where the summands are given by the $\pm i$ eigenspaces of Clifford multiplication $\gamma(\eta)$. Similarly, if $\mathfrak{s}_0^c$ is a $\mathrm{Spin}^c(2)$ structure on $\Sigma$, then $\pi^*(\mathfrak{s}_0^c)$ defines a $\mathrm{Spin}^c$ structure on $Y$. In this case, the spinor bundle $\slashed{S}_c$ is isomorphic to $\underline{\mathbb{C}} \oplus \pi^*(K_{\Sigma}^{-1})$, where the summands are given by the $\pm i$ eigenspaces of $\gamma(\eta)$.

\subsection{$\Z_2$-harmonic spinors on Seifert--fibered 3-manifolds}

In this subsection we consider the spin Dirac operator $D=\slashed D$ as in (\refeq{Diracoperatorone}). Let $\Sigma$ be an orbifold Riemann surface with a spin structure $\mathfrak s_0$, and denote the associated positive spinor bundle by $K_{\Sigma}^{1/2}$. Next, let $\pi: Y \to \Sigma$ be a Seifert--fibered 3-manifold induced by an orbifold line bundle $L$ with Seifert invariant $(b = \deg |L|, \beta_1, \cdots, \beta_n)$. $\mathfrak s_0$ induces a spin structure  $\pi^*(\mathfrak{s}_0)$ on $Y$; the associated spinor bundle decomposes as $\slashed{S} = K_{\Sigma}^{1/2} \oplus K_{\Sigma}^{-1/2}$ via Clifford multiplication by $\eta$. The following lemma gives a Fourier decomposition for spinors on $Y$.

\begin{lm}
The action of $U(1)$ on $Y$ induces a decomposition  
\begin{equation}
    L^2(Y; \slashed{S}) = \bigoplus_{k \in \Z} L^2(\Sigma; (K_{\Sigma}^{1/2} \oplus K_{\Sigma}^{-1/2}) \otimes L^k) \label{fourierdecomposition}
\end{equation}
\noindent into sections over $\Sigma$ of irreducible representations.
\end{lm} 

\begin{proof} 
Let $U_{\alpha}$ be a cover by local trivializations of $Y \to \Sigma$. We may assume that in the fiber coordinate $t \in \R / 2\pi \Z$, the transition functions are given by $t \mapsto t + \omega_{\alpha\beta}(z)$.

In local trivializations, the $S^1$ action decomposes sections as sums of $f_k(z)e^{ikt}$ for $k \in \Z$, where $f_k$ define sections of $(K_{\Sigma}^{1/2} \oplus K_{\Sigma}^{-1/2})$. The transition data of $f_k(z)e^{ikt}$ differs from that of $f_k(z)$ by a factor of $e^{ik \omega_{\alpha\beta}}$, which defines the line bundle $L^k$. 
\end{proof}

There is a particular perturbation of the Levi-Civita connection on $Y$ which, combined with \eqref{fourierdecomposition}, allows the Dirac operator to be reduced to differential operators on $\Sigma$. This perturbation was first introduced by Mrowka--Oszvath--Yu in \cite{MOY1996}, and was also studied by Nicolaescu in \cite{nicolaescu1996adiabatic}. The perturbed connection is defined by
$$
^\circ\nabla^{TY} := d \oplus \pi^*(\nabla_\Sigma)
$$
\noindent where $\nabla_\Sigma$ is the Levi-Civita connection on $\Sigma$. \cite[Lem. 5.2.1]{MOY1996} shows that the induced spin connection may be written $^\circ \nabla= \nabla^\text{spin}+B$ where $B\in \Omega^1(Y; \mathfrak{so}(\slashed{S}))$, and that the corresponding Dirac operators are related by
\begin{equation}
    ^\circ \slashed{D} = \slashed{D} - \frac{1}{2}\xi \hspace{2cm} \text{where} \hspace{2cm} \xi = \frac{b\pi}{V}. \label{perturbedDirac}
\end{equation}

As orbifold Riemann surfaces have K\"ahler structures, $\Z_2$-harmonic spinors can be produced by taking the square root of holomorphic sections of certain orbifold holomorphic line bundles. In the rest of this subsection, we construct examples of $\Z_2$-harmonic spinors coming from the pullback of orbifold $\Z_2$-spinors over Riemann surfaces. To begin with, we need the following computations of the dimension of holomorphic sections for the orbifold line bundle $K_{\Sigma} \otimes \mathcal{L}^2 \otimes L^{2k}$.

\begin{lm}
    \label{lemma_dimension_computation}
    Let $(b; \beta_1, \cdots, \beta_n)$ be the Seifert invariant for $L$, and let $\mathcal{L}$ be another orbifold line bundle with orbifold invariant $(\deg(\mathcal{L}); 0, \cdots, 0)$. For the orbifold bundle $K_{\Sigma} \otimes \mathcal{L}^2 \otimes L^{2k}$, we have:
    \begin{itemize}
        \item [(i)] Define $N := \deg(|K_{\Sigma} \otimes \mathcal{L}^2 \otimes L^{2k}|)$, then
        $$
        N = 2kb + 2\gamma - 2 + \sum_{i=1}^n \left\lfloor \frac{2k\beta_i + \alpha_i - 1}{\alpha_i} \right\rfloor + 2\deg(\mathcal{L}).
        $$
        \item [(ii)] Suppose $2\deg(\mathcal{L}) + 2kb > 0$, or $2\deg(\mathcal{L}) + 2kb = 0$ but $\mathcal{L}^2 \otimes L^{2k}$ is nontrivial, then
        $$
        \dim_{\mathbb{C}} H^0(K_{\Sigma} \otimes \mathcal{L}^2 \otimes L^{2k}) = N + 1 - \gamma.
        $$
        \item [(iii)] If $N \geq 2\gamma$, then generic holomorphic sections of $K_{\Sigma} \otimes \mathcal{L}^2 \otimes L^{2k}$ have $N$ simple zeros.
    \end{itemize}
\end{lm}

\begin{proof}
    (i) follows from a direct computation using \eqref{eq_seifert_invariant_tensor}. (ii) follows directly from Proposition \ref{prop_vanishing_negative_degree} and the Kawasaki-Riemann-Roch theorem \ref{thm_orbifold_Riem_Roch}. For (iii), by \cite[Proposition 2.0.14]{MOY1996}, the holomorphic sections of the orbifold line bundle $K_{\Sigma} \otimes \mathcal{L}^2 \otimes L^{2k}$ correspond naturally to the holomorphic sections of its desingularization $|K_{\Sigma} \otimes \mathcal{L}^2 \otimes L^{2k}|$, which implies our claim.
\end{proof}

\begin{prop}\label{Seifertfiberedprop}
    Under the previous conventions, for the Seifert--fibered manifold $\pi: Y \to \Sigma$ with the pull-back spin structure $\pi^*(\mathfrak{s}_0)$. For every $k$ such that
    $$
    \deg(|K_{\Sigma} \otimes \mathcal{L}^2 \otimes L^{2k}|) + 1 - \gamma \geq 0, \quad \deg(|K_{\Sigma} \otimes \mathcal{L}^2 \otimes L^{2k}|) \geq 2\gamma,
    $$
    there exists a metric $g_{s,V}$ defined in \eqref{gY} with $s, V$ depending on $k$, such that there are smooth, non-degenerate $\Z_2$-harmonic spinors $(\mathcal{Z}_k, \ell_k, \Phi_k)$ with respect to $g_k$, where $\mathcal{Z}_k = \pi^*\mathcal{Z}_{\Sigma}$ for some $\mathcal{Z}_{\Sigma}$ which are points in $\Sigma$.
\end{prop}

\begin{proof} 
    By Lemma \ref{lemma_dimension_computation}, the assumptions imply that there exists $q \in H^0(\Sigma; K_{\Sigma} \otimes \mathcal{L}^2 \otimes L^{2 k})$ which has isolated, simple zeros. The perturbed Dirac operator \eqref{perturbedDirac} may be written 
    $$
    ^\circ \slashed{D} = \gamma(\eta)\del_t + \slashed{D}_\Sigma,
    $$
    \noindent where $\slashed{D}_\Sigma$ is the Dirac operator on $\Sigma$ in the $\text{Spin}$ structure  $K_{\Sigma}^{\pm 1/2}$. In the decomposition \eqref{fourierdecomposition}, as $\gamma(\eta) = \mathrm{diag}(i, -i)$, the restriction of the operator $^\circ \slashed{D}$ on $(K_{\Sigma}^{1/2} \oplus K_{\Sigma}^{-1/2}) \otimes L^k$ takes the form 
    $$
    ^\circ \slashed{D} = \begin{pmatrix} -k & -2\del_{A_0} \\ 2\delbar_{A_0} & k \end{pmatrix},
    $$
    \noindent where $A_0$ is the spin connection on $K_{\Sigma}^{\pm 1/2}$ or the same with $\del_A, \delbar_A$ for a twisted spin connection $A$ on $K_{\Sigma}^{\pm 1/2} \otimes \mathcal{L}$.

    Next, let $\mathcal{Z}_{\Sigma} = q^{-1}(0)$. There is a flat line bundle $\ell_{\Sigma}$ defined by the property that its restriction to a punctured disk $D_i \setminus \{z_i\}$ is the Mobius bundle for each $z_i \in \mathcal{Z}_{\Sigma}$, and $\sqrt{q} \in \Gamma(\Sigma - \mathcal{Z}_{\Sigma}; K_{\Sigma}^{1/2} \otimes L^k \otimes \ell_{\Sigma})$ is a well-defined section satisfying $\delbar_{A_0'} \sqrt{q} = 0,$ where $A_0'$ is the connection induced by the spin connection $A_0$ and the unique flat connection with $\Z_2$-holonomy on $\ell$. Set $\Phi_k = ( e^{ikt} \sqrt{q} ,0 )$. Then 
    $$
    \slashed{D}\Phi_k = \left(-k + \tfrac{\xi}{2}\right) \begin{pmatrix} e^{ikt} \sqrt{q} \\ 0 \end{pmatrix} = 0, 
    $$
    when the metric $g_{\Sigma}$ is chosen so that $\Sigma$ has volume $V = \tfrac{\text{deg}(L) \pi}{k}$. That is to say, $\Phi_k$ is a $\Z_2$-harmonic spinor with respect to $g_{1,V}$ for this chosen $V$. It is non-degenerate because the zeros of $q$ were chosen to be non-degenerate, and the singular set is $\mathcal{Z} = \pi^*(\mathcal{Z}_{\Sigma})$.  

    If $c_1(L) = 0$ (i.e., $Y = S^1 \times \Sigma$), then $\sqrt{q}$ may be constructed similarly with the bundle $\mathcal{L}$ satisfying $c_1(\mathcal{L}) = 1$ in place of $L$, and setting $\Phi_k = \sqrt{q_k}$ to be invariant in the $S^1$ directions. 
\end{proof}

Even though not every orbifold is spin and thus $K_{\Sigma}^{\frac{1}{2}}$ might not always exist, we can consider a $\mathrm{Spin}^c$ structure with spinor bundle $\slashed{S}^c \cong \mathbb{C} \oplus K_{\Sigma}^{-1}$. Analogous to Proposition \ref{Seifertfiberedprop}, we can formally take $\mathcal{L}^2 \cong K_{\Sigma}^{-1}$, and a similar result holds. To avoid duplication, we only state the result.

\begin{prop}
    For the Seifert--fibered manifold $\pi: Y \to \Sigma$ with the pull-back spin structure $\pi^*(\mathfrak{s}_0^c)$, for every $k$ such that
    $$
    \deg(|\mathcal{L}^2 \otimes L^{2k}|) + 1 - \gamma \geq 0, \quad \deg(|\mathcal{L}^2 \otimes L^{2k}|) \geq 2\gamma,
    $$
    there exists a metric $g_{s, V}$ defined in \eqref{gY} with $s$ and $V$ depending on $k$, such that there are smooth, non-degenerate $\Z_2$-harmonic spinors $(\mathcal{Z}_k, \ell_k, \Phi_k)$ with respect to $g_k$, where $\mathcal{Z}_k = \pi^*\mathcal{Z}_{\Sigma}$ for some $\mathcal{Z}_{\Sigma}$ that are points in $\Sigma$.    
\end{prop}

In summary, we conclude the following:

\begin{thm}
    Let $\pi : Y \to \Sigma$ be a Seifert--fibered 3-manifold. Then for each $k \geq 1$, there exist metrics $g_k$ that admit smooth, non-degenerate $\Z_2$-harmonic spinors $(\mathcal{Z}_k, \ell_k, \Phi_k)$, where $\mathcal{Z}_k \subseteq Y$ is the union of disjoint fibers of $\pi$. 
\end{thm}

The previous theorem immediately gives a large class of interesting examples (given in Corollary \ref{example1.8}). 

\begin{cor} 
	\label{cor_examples_Z_2_harmonic_spinors}
	The following manifolds admit $\Z_2$-harmonic spinors, all of which are smooth and non-degenerate. 
\begin{enumerate}[(i)]
\item  $Y=S^3$ admits $\Z_2$-harmonic spinors $(\mathcal Z_k, \ell_k, \Phi_k)$ with respect to the Berger metrics $g_{B,V}$ such that $\mathcal Z_k$ is a Hopf link with $2k$-components. 
\item $Y=S^1 \times S^2$ admits $\Z_2$-harmonic spinors $(\mathcal Z_k, \ell_k, \Phi_k)$ with respect to metrics $g_k=dt^2 + V_k g_{S^2}$ for $V_k\in \R$, such that $\mathcal Z_k=S^1 \times \mathcal Z_{S^2}$  where $\mathcal Z_{S^2}\subseteq S^2$ is a collection of $2k$ points.  
\item $Y=\Sigma(2,3,5)$, the Poincar\'e  homology sphere, admits a $\Z_2$-harmonic spinor $(\mathcal{Z},\ell,\Phi)$ with a connected singular set $\mathcal{Z}=\pi^{-1}(p_0)$ and $p_0\in \Sigma$.
\end{enumerate}
\end{cor}
\begin{proof} 
For (i), consider $S^3\to S^2$ given by the Hopf fibration with degree $+1$, in which case the metrics (\refeq{gY}) are the Berger metrics. The disjoint fibers of the Hopf fibration are pairwise Hopf links. (ii) is immediate from the proof of Proposition \ref{Seifertfiberedprop}. For (iii), the Seifert-invariant of $\Sigma(2,3,5)$ is $(\gamma=0,b=-1,(2,1),(3,1),(5,1))$. By Lemma \ref{lemma_dimension_computation}, for $\mathcal{L}$ trivial, and $k=-2$, we have $N=\deg(|K_{\Sigma}\otimes L^{-4}|)=1$ and $\dim_{\mathbb{C}}H^0(K_{\Sigma}\otimes L^{-4})=2$. In this case, a generic section $q\in H^0(K_{\Sigma}\otimes L^{-4})$ has one simple zero. The claim follows from Proposition \ref{Seifertfiberedprop}.
\end{proof}

\subsection{$\Z_2$-harmonic 1-forms on Seifert--fibered 3-manifolds}
\label{subsec_Z2Seifert}
This section considers $\Z_2$-harmonic 1-forms over Seifert--fibered spaces. Since $\Z_2$-harmonic 1-forms are directly related to the non-compactness behavior of the $\mathrm{SL}(2; \mathbb{C})$ character variety, we do not expect their existence over every Seifert--fibered manifold. By Corollary \ref{cor_examples_Z_2_harmonic_spinors}, there exists a $\Z_2$-harmonic spinor (for the spin Dirac operator $D=\slashed D$) with a connected singular set over a homology sphere; in contrast, according to \cite{HaydysPSLR2022}, the singular set of a $\Z_2$ harmonic 1-form on a homology sphere must have at least two connected components.

Using the previous conventions, for a Seifert--fibered manifold $\pi: Y \to \Sigma$, with $\Sigma$ being an orbifold, we first consider the space of orbifold quadratic differentials. We consider the $s=1$ version of the metric \eqref{gY}, i.e. $g_{1, V} = \eta^2 + \pi^*(g_{\Sigma})$, where $g_{\Sigma}$ is an orbifold Riemannian metric on $\Sigma$. The orientation is given by $\mathrm{dvol}_{Y} = \eta \wedge \mathrm{dvol}_{\Sigma}$, and we have $d\eta = -\frac{2\pi b}{\mathrm{Vol}(\Sigma)}\mathrm{dvol}_{\Sigma}$.

The following shows that the pullback of the orbifold $\Z_2$-harmonic 1-form over $\Sigma$ is still a $\Z_2$-harmonic 1-form over $Y$.

\begin{lm}
    Let $(Z, \ell, \nu)$ be an orbifold $\Z_2$- harmonic 1-form over $\Sigma$, then $(p^{*}Z, p^*\ell, p^{*}v)$ is a $\Z_2$-harmonic 1-form over $Y$.
\end{lm}
\begin{proof}
    As $d\nu = 0$ and $d\star_{g_{\Sigma}}\nu = 0$, under the pullback, we obtain $d(p^{*}\nu) = d\star_{p^*g_{\Sigma}}(p^*\nu) = 0$. Since $\star_{g_Y}p^*\nu = -\eta \wedge \star_{p^*g_{\Sigma}}p^*\nu$, we compute
    \[
    d\star_{g_Y}p^*\nu = -\frac{2\pi b}{\mathrm{Vol}(\Sigma)}\mathrm{dvol}_{\Sigma} \wedge \star_{(p^*g_{\Sigma})}p^*\nu + \eta \wedge d\star_{p^*g_{\Sigma}}p^*\nu = 0.
    \]
    Moreover, as $|\nu|_{g_{\Sigma}}$ is bounded, we conclude that $|p^*(\nu)|_{g_{\Sigma}}$ is also bounded, hence $p^*\nu$ is a $\Z_2$-harmonic 1-form.
\end{proof}

Now, we will construct examples of orbifold $\Z_2$-harmonic 1-forms. Note that by \eqref{eq_seifert_invariant_tensor}, the Seifert invariant for $K_{\Sigma}^2$ is $(4\gamma - 4 + n; \alpha_1 - 2, \cdots, \alpha_n - 2)$, and by Theorem \ref{thm_orbifold_Riem_Roch}, we obtain 
\[
\dim_{\mathbb{C}} H^0(K^2_{\Sigma}) = 3\gamma - 3 + n.
\]
Moreover, when $4\gamma - 4 + n \geq 2\gamma$, generic sections will have simple zeros. In summary, we conclude the following:

\begin{prop}
    \label{prop_1_form_existence_Seifert_fibered}
    Let $Y$ be a Seifert--fibered space with Seifert invariant $(\gamma, b; (\alpha_1, \beta_1), \cdots, (\alpha_n, \beta_n))$, suppose $3\gamma - 3 + n > 0$ and $2\gamma - 4 + n \geq 0$. Then, for the metric $g_{1, V}$ in \eqref{gY}, there exist smooth, non-degenerate $\Z_2$-harmonic 1-forms $(\mathcal Z, \ell, \nu)$ with $\mathcal Z = p^{*}\mathcal Z_0$ where $\mathcal Z_0$ consists of $4\gamma - 4 + n$ points over $\Sigma$.
\end{prop}

We now consider several interesting examples of Seifert--fibered manifolds which admit $\Z_2$ harmonic 1-forms. For every choice of $n \geq 3$ pairwise relatively prime integers $(a_1, \cdots, a_n)$ greater than one, there is an associated Brieskorn homology sphere. These are described as the link of isolated singularities at zero of the complex variety
\[
V := \{c_{i1}z_1^{a_1} + \cdots + c_{in}z_n^{a_n} = 0, \; i = 1, \cdots, n - 2\} \subset \mathbb{C}^n,
\]
where $C = \{c_{ij}\}$ is an $((n - 2) \times n)$ matrix of real numbers such that each of its maximal minors is non-zero. We define $\Sigma(a_1, \cdots, a_n) := V \cap S^{2n - 1}$; the $U(1)$ action on $\C^n$ makes this a Seifert--fibered space over an orbifold with topology $S^2$. Therefore, when $n \geq 4$, there exist non-degenerate $\Z_2$-harmonic 1-forms over the Brieskorn homology spheres. 

\begin{cor}
    Let $\Sigma(a_1, \cdots, a_n)$ be a Brieskorn homology sphere with $n\geq 4$. Then there exist non-degenerate $\Z_2$-harmonic 1-forms on it.
\end{cor} 

Our method doesn't establish the existence of $\Z_2$- harmonic 1-forms over $\Sigma(a_1, a_2, a_3)$, and one should not expect any in this case. Indeed, by \cite[Page 9]{BodenCurtis06}, the $\mathrm{SL}(2,\mathbb{C})$ character variety of $\Sigma(a_1, \cdots, a_n)$ has positive dimension if and only if $n \geq 4$. As the $\mathrm{SL}(2,\mathbb{C})$ character variety is an affine variety, it is non-compact if and only if it is positive dimensional. We therefore shouldn't expect the existence of $\Z_2$-harmonic 1-forms for $\Sigma(a_1, a_2, a_3)$.
\subsection{Connected Sum Results}
In this subsection, we explore some implications of the connected sum formula for $\Z_2$-harmonic 1-forms for the geometry of the $\text{SL}(2,\C)$ representation variety. 

\subsubsection{Connected Sum of $\Z_2$-Harmonic 1-forms over a Riemann Surface}
We first consider the connected sum of Riemann surfaces. In this case $\Z_2$-harmonic 1-forms are closely tied to the space of holomorphic quadratic differentials, which plays an important role in Teichm\"uller theory and other aspects of the geometry of Riemann surfaces, as explored in foundational works such as Hubbard and Masur \cite{HubbardMasur1979}.

 Let $(\Sigma, g)$ be a closed Riemann surface. The space of $\Z_2$-harmonic 1-forms is identified with the space of holomorphic quadratic differentials as follows. Given a quadratic differential $q \in H^0(K_{\Sigma}^2)$, $\nu := \Re(\sqrt{q})$ defines a $\Z_2$-harmonic 1-form (see \cite{Taubes3dSL2C}). Conversely, given a $\Z_2$-harmonic 1-form $(\mathcal{Z}, \ell, \nu)$, we write $\nu^{(1,0)}$ to be the $(1,0)$ component of $\nu$, then $\nu^{(1,0)} \otimes \nu^{(1,0)} \in H^0(K_{\Sigma}^2)$ defines a holomorphic quadratic differential, and the above correspondence is an isomorphism. Therefore, the two-dimensional version of Theorem \ref{maina} for $\Z_2$-harmonic 1-forms (cf Remark \ref{Riemannsurfacegluing}) may be invoked to glue holomorphic quadratic differentials.

For $i = 1, 2$, let $(\Sigma_i, g_i)$ be Riemann surfaces with metric $g_i$ and genus $\gamma_i$. Consider $q_i \in H^0(K_{\Sigma_i}^2)$ quadratic differentials with simple zeros. We write $\mathcal{Z}_i := q_i^{-1}(0)$ for the set of zeros, of which there are $|\mathcal{Z}_i| = 4\gamma_i - 4$. We define $p_i: \Sigma_{\mathcal{Z}_i} \to \Sigma_i$ to be the double branched covering of $\Sigma_i$ along $\mathcal{Z}_i$. Its genus is $\gamma(\Sigma_{\mathcal{Z}_i}) = 4\gamma_i - 3$ (by the Riemann-Hurwitz formula).

Using the notation for the connected sum construction as in Section \ref{subsection_connected_sum},  we choose points $x_i \in \Sigma_i \setminus \mathcal{Z}_i$ at which the sum is performed. $\Sigma := \Sigma_1 \# \Sigma_2$ is equipped with the metric $g_\delta$ of neck diameter $O(\delta)$  as in (\refeq{torusmetric}). We write $\mathcal{Z} := \mathcal{Z}_1 \# \mathcal{Z}_2$, and let $p: \Sigma_{\mathcal{Z}} \to \Sigma$ be the double branched covering along $\mathcal{Z}$ with flat bundle $\ell$. Topologically, $\Sigma_{\mathcal{Z}} \simeq \Sigma_{\mathcal{Z}_1} \# \Sigma_{\mathcal{Z}_2} \# T^2$ with genus $\gamma(\Sigma_{\mathcal{Z}}) = 4(\gamma_1 + \gamma_2) - 5$.

An approximate quadratic differential can be constructed as in \eqref{approximatesol1form}, which we denote as $q^{\app}_\delta$. The gluing implies this may be corrected to a true holomorphic quadratic differential in the conformal structure defined by $g_\delta$. More precisely: 

\begin{thm}\label{holomorphicquadraticgluing}
    There exists $\delta_0$ such that for $\delta < \delta_0$, there exists an $f_\delta \in \Gamma(\ell\otimes \mathbb{C})$ and a diffeomorphism $\varphi_\delta$ with $\varphi_\delta = \mathrm{Id}$ near the gluing region such that
    $$
    q_\delta := \varphi_\delta^*(q^{\app}_\delta + \partial f_\delta \otimes \partial f_\delta)
    $$
    is a quadratic differential with respect to $g_\delta$. Moreover, $q_\delta$ is non-degenerate and the zeros of $q_\delta$ can be written as
    $$
    q_\delta^{-1}(0) = \varphi_\delta^{-1}(\mathcal{Z}_1 \cup \mathcal{Z}_2) \cup \mathcal Z',
    $$
    where $\mathcal{Z}_1 \cup \mathcal{Z}_2$ are simple zeros counted with multiplicity $|\mathcal{Z}_1 \cup \mathcal{Z}_2|=4(\gamma_1+\gamma_2)-8$, and $q_\delta$ has even vanishing order on $\mathcal Z'$ with multiplicity $|\mathcal{Z}'|=4$.
\end{thm}

\begin{proof}
    The existence follows from Theorem \ref{mainb} (cf. Remark \ref{Riemannsurfacegluing}). For the zeros, since $\Re(\sqrt{q_\delta})$ is also non-degenerate, it has only simple odd zeros. Moreover, the odd zeros of $\Re(\sqrt{q_\delta})$ are the branching set of the branched covering $\varphi_\delta^* \circ p$, which are exactly $\varphi_\delta^{-1}(\mathcal{Z}_1 \cup \mathcal{Z}_2)$. Furthermore, let $K_{\Sigma}$ be the canonical bundle defined using the holomorphic structure of $g_\delta$. Since $q_\delta \in H^0(K_{\Sigma}^2)$, $q_\delta$ has $4(\gamma_1 + \gamma_2) - 4$ zeros, which implies that even zeros must exist with multiplicity 4.
\end{proof}

The appearance of these additional zeros is not an artifact of the proof, but is forced by the degree count. Indeed, a quadratic differential on \(\Sigma_1\#\Sigma_2\) has $4(\gamma_1+\gamma_2)-4$
zeros counted with multiplicity, whereas the odd zeros inherited from \(q_1\) and \(q_2\) contribute only
\[
(4\gamma_1-4)+(4\gamma_2-4)=4(\gamma_1+\gamma_2)-8.
\]
Thus zeros of total multiplicity \(4\) must be created in the gluing process. These new zeros have even order as zeros of the quadratic differential, and hence do not contribute to the branching locus of the associated \(\mathbb Z_2\)-harmonic \(1\)-form; equivalently, \(\ell\) extends over them. Consequently, even when \(q_1\) and \(q_2\) are generic quadratic differentials with simple zeros, the glued quadratic differential \(q_\delta\) lies in a lower stratum of the space of quadratic differentials. This reflects the difference between gluing \(\mathbb Z_2\)-harmonic \(1\)-forms and gluing quadratic differentials directly: the former keeps track of the odd zeros, while the latter also records the even zeros. We also note that the top stratum may be obtained by gluing holomorphic quadratic differentials directly using Lemma~\ref{infiniteneck} with \(d=4\).

Theorem \ref{holomorphicquadraticgluing} also has implications for singular measured foliations. Given a quadratic differential $q$, $\Re(\sqrt{q})$ defines a  singular measured foliation over $\Sigma$, and the zeros of $q$ correspond to singular leaves of the foliation with explicit local structure. The work of Hubbard and Masur \cite{HubbardMasur1979} identifies the equivalence class of tame measured foliations and the space of quadratic differentials. Theorem \ref{holomorphicquadraticgluing} provides a direct way to glue two different measured foliations using the connected sum, suggesting that the gluing process creates new singularities in the foliations.

\subsubsection{Connected sum of $\Z_2$-harmonic 1-forms for 3-manifolds.}

As in the case of Riemann surfaces above, Theorems \ref{maina} and \ref{thm_non_degenerate_gluing}
have some somewhat surprising implications for the structure of the boundary strata of a hypothetical compactification for the $SL(2,\C)$ representation variety using $\Z_2$-harmonic 1-forms. At present, the existence of such a compactification is only speculation, and its construction is the subject of forthcoming work \cite{PartIV}.

Under the conventions in Theorem \ref{thm_non_degenerate_gluing}, suppose $\mathcal{Z}_1$ and $\mathcal{Z}_2$ are both non-empty. Then, by a straightforward computation, for the first cohomology, we have
\begin{equation}
\label{eq_homology_double_covering_connected_sum_S3}
    H^1_-(Y_\mathcal{Z}; \mathbb{R}) \cong H^1_-(Y_{\mathcal{Z}_1}; \mathbb{R}) \oplus H^1_-(Y_{\mathcal{Z}_2}; \mathbb{R}) \oplus \mathbb{R}.
\end{equation}

\noindent On the other hand, if $\mathcal{Z}_1$ is not empty but $\mathcal{Z}_2$ is empty, we write $Y_{\mathcal{Z}_2} = Y_2^+ \cup Y_2^-$ to be the disjoint union of two copies of $Y_2$ with the obvious involution. Then, $Y_\mathcal{Z} \cong Y_{\mathcal{Z}_1} \# Y_2^+ \# Y_2^-$, and we have
\[
H^1_-(Y_\mathcal{Z}; \mathbb{R}) \cong H^1_-(Y_{\mathcal{Z}_1}; \mathbb{R}) \oplus H^1_-(Y_2; \mathbb{R}).
\]

\noindent Thus when both $\mathcal{Z}_1$ and $\mathcal{Z}_2$ are non-empty, there is an additional $\mathbb{R}$ factor in \eqref{eq_homology_double_covering_connected_sum_S3} compared to the simple direct sum of the cohomologies.

By the results of \cite{DonaldsonMultivalued}, the moduli space of $\Z_2$-harmonic 1-forms with fixed singularity type near a smooth, non-degenerate point $(\mathcal{Z}, \ell, \nu)$  has the structure of a smooth manifold with dimension $k_{\Z_2}(Y;\mathcal Z) := H^1_-(Y_\mathcal{Z})$. The appearance of the extra $\mathbb{R}$ factor in \eqref{eq_homology_double_covering_connected_sum_S3} can potentially be understood in terms of $\text{SL}(2,\C)$ representations via the construction in \cite[Page 10]{HaydysPSLR2022} as follows. Let $\rho_i: \pi_1(Y_i) \to \mathrm{SL}(2,\mathbb{C})$ be two different representations; by Van-Kampen's Theorem $\pi_1(Y) = \pi_1(Y_1) * \pi_1(Y_2)$, so for any $\tau \in \mathrm{SL}(2,\mathbb{C})$, we can construct an additional family of representations $\rho_{\tau}: \pi_1(Y) \to \mathrm{SL}(2,\mathbb{C})$ given by $(\rho_1, \tau \rho_2 \tau^{-1})$. Note $\rho_\tau$ are pairwise distinct modulo conjugation on $Y$ for every $\tau$, despite the fact that 
$\tau \rho_2 \tau^{-1}$ is conjugate to $\rho_2$ on $Y_2$. 
As both $\tau$, $\rho_1$, and $\rho_2$ can vary within a non-compact family, the family $\rho_\tau$ might contribute to additional boundary strata on $Y$ that do not appear as products of boundary strata for either $Y_1$ or $Y_2$. In particular, it seems likely the $\Z_2$-harmonic 1-forms in these strata might represent the $\R$ summand.  

The reverse construction is also valid. Let $\mathcal{R}(Y)$ denote the $\text{SL}(2,\C)$ representation variety of $Y$. Then, for $\rho \in \mathcal R(Y)$, for $i = 1, 2$, since $\pi_1(Y_i)$ are subgroups of $\pi_1(Y)$, we write $\rho_i := \rho|_{\pi_1(Y_i)}$. This defines a map $\Psi: \mathcal R(Y)\to \mathcal R(Y_1)\times \mathcal R(Y_2)$. If we denote the equivalence class of a representation up to conjugation by $[\rho]$, then the pre-image is precisely $\Psi^{-1}([\rho_1], [\rho_2]) = \{ [\rho_{\tau}] \  |  \ \tau \in \text{SL}(2,\C)\}$. Counting dimensions, it follows that
\[
\dim_\R \mathcal R(Y) = \dim_\R \mathcal R(Y_1) + \dim_\R \mathcal R(Y_2) + 6.
\]

\noindent Gluing results for boundary strata in the case of Riemann surfaces \cite{MWWW} suggest that the relation $\dim_\R \mathcal R(Y) = 2k_{\Z_2}(Y;\mathcal Z)$ holds for $\Z_2$-harmonic 1-forms in the top boundary stratum (and \cite{PartIV} supports a similar relation for $3$-manifolds). Given this, one would expect that the top boundary stratum of $\mathcal R(Y)$ consists of $\Z_2$-harmonic spinors whose singular set $\mathcal Z_0$ is such that $$k_{\Z_2}(Y;\mathcal Z_0)=k_{\Z_2}(Y_1;\mathcal Z_1) + k_{\Z_2}(Y_2;\mathcal Z_2) + 3.$$ Analogous to the Riemann surface case, however, the connected sum gluing produces only $\Z_2$-harmonic 1-forms with singular sets $\mathcal Z_\#$ satisfying
\[
k_{\Z_2}(Y; \mathcal Z_\#)=
k_{\Z_2}(Y_1;\mathcal Z_1) + k_{\Z_2}(Y_2;\mathcal Z_2) + 1=k_{\Z_2}(Y;\mathcal Z_0)-2.
\]
This suggests that the $\Z_2$ harmonic 1-forms constructed by the gluing method in Theorem \ref{thm_non_degenerate_gluing}  lie in a lower stratum of the total space of $\Z_2$ harmonic 1-forms, similar to the situation for Riemann surfaces. In this case, the top stratum would have to consist of $\Z_2$-harmonic spinors whose singular set had additional components, generalizing the appearance of the extra zeros of the holomorphic quadratic differentials in Theorem \ref{holomorphicquadraticgluing}.

\appendix

\bibliographystyle{amsalpha}
{\small 
\bibliography{BibliographyGluingPaper}}
\end{document}